%% file: kernel_of_flux_maps.tex
\documentclass[11pt]{amsart}
\usepackage[utf8]{inputenc}
\usepackage{amsmath,amssymb,amsthm}
\usepackage{enumitem}
\usepackage{xcolor, wasysym}

\usepackage{booktabs}
\usepackage{makecell}
\usepackage{tabularx}
\usepackage{colortbl}
\usepackage{multirow}
\usepackage{adjustbox}

\usepackage{caption}

\usepackage{comment}
\usepackage[top=1in, bottom=1in, left=1.35in,right=1.35in, marginpar=1in]{geometry}
\usepackage{graphicx,import}
\usepackage{tikz, tikz-cd}

\newcommand{\lasso}{%
    \begin{tikzpicture}[scale=0.1, thick]
    {
    [yshift = -1]
    \draw circle(1);
    \draw (1.0,0.0)--(5.0,0.0);
    }
    \end{tikzpicture}%
}

\usepackage[colorlinks,citecolor=blue!70!black,linkcolor=red!60!black]{hyperref}
\usepackage[nameinlink]{cleveref}

\theoremstyle{plain}
\newtheorem{THM}{Theorem}[section]
\newtheorem{MAINTHM}{Theorem}

\newtheorem{MAINCOR}[MAINTHM]{Corollary}
\newtheorem{PROP}[THM]{Proposition}
\newtheorem{LEM}[THM]{Lemma}
\newtheorem{FACT}[THM]{Fact}
\newtheorem{COR}[THM]{Corollary}

\newtheorem*{CLAIME}{Claim}

\theoremstyle{definition}
\newtheorem{DEF}[THM]{Definition}
\newtheorem{RMK}[THM]{Remark}

\newtheorem{Q}[THM]{Question}

\newlist{PROPenum}{enumerate}{1} % should only occur inside definition env.
\setlist[PROPenum]{label=(\roman*),ref=\thePROP\,(\roman*)}
\crefname{PROPenumi}{Proposition}{Propositions}

\DeclareMathOperator{\id}{id}
\DeclareMathOperator{\Id}{Id}
\DeclareMathOperator{\Aut}{Aut}
\DeclareMathOperator{\SAut}{SAut}
\DeclareMathOperator{\Out}{Out}

\DeclareMathOperator{\Hom}{Hom}
\DeclareMathOperator{\Ext}{Ext}

\DeclareMathOperator{\PMap}{PMap}
\DeclareMathOperator{\PMapc}{PMap_{\it c}}

\DeclareMathOperator{\PPHE}{PPHE}
\DeclareMathOperator{\Homeo}{Homeo}
\DeclareMathOperator{\Map}{Map}

\DeclareMathOperator{\supp}{supp}
\DeclareMathOperator{\rk}{rk}
\DeclareMathOperator{\cork}{cork}

\newcommand{\N}{\mathbb{N}}
\newcommand{\Z}{\mathbb{Z}}
\newcommand{\bbQ}{\mathbb{Q}}
\newcommand{\R}{\mathbb{R}}

\newcommand{\cC}{\mathcal{C}}

\newcommand{\cE}{\mathcal{E}}
\newcommand{\cF}{\mathcal{F}}
\newcommand{\cL}{\mathcal{L}}
\newcommand{\cV}{\mathcal{V}}

\newcommand{\cP}{\mathcal{P}}

\newcommand{\cR}{\mathcal{R}}
\newcommand{\cS}{\mathcal{S}}

\newcommand{\cA}{\mathcal{A}}

\def\G{{\Gamma}}

\newcommand{\defeq}{:=}
\newcommand{\la}{\langle}
\newcommand{\ra}{\rangle}
\newcommand{\<}{\langle}
\renewcommand{\>}{\rangle}
\newcommand{\longto}{\longrightarrow}
\newcommand{\cech}{\text{{\v C}ech}}

\newcommand{\ceC}{\check{C}}
\newcommand{\dC}{\mathring{C}}
\newcommand{\wt}[1]{\widetilde{#1}}
\newcommand{\dd}{\setminus \hspace{-5pt}\setminus}

\let\oldtocsection=\tocsection

\let\oldtocsubsection=\tocsubsection

\let\oldtocsubsubsection=\tocsubsubsection

\renewcommand{\tocsection}[2]{\hspace{0em}\oldtocsection{#1}{#2}}
\renewcommand{\tocsubsection}[2]{\hspace{1em}\oldtocsubsection{#1}{#2}}
\renewcommand{\tocsubsubsection}[2]{\hspace{2em}\oldtocsubsubsection{#1}{#2}}

\title[Generating Sets and Algebraic Properties of $\PMap(\G)$]{Generating Sets and Algebraic Properties of Pure Mapping Class Groups of Infinite Graphs}
\author{George Domat, Hannah Hoganson, and Sanghoon Kwak}
\date{}

\begin{document}
\maketitle 
\begin{abstract}
    We completely classify the locally finite, infinite graphs with pure mapping class groups admitting a coarsely bounded generating set. We also study algebraic properties of the pure mapping class group: We establish a semidirect product decomposition, compute first integral cohomology, and classify when they satisfy residual finiteness and the Tits alternative. These results provide a framework and some initial steps towards quasi-isometric and algebraic rigidity of these groups.
\end{abstract} 

\addtocontents{toc}{\protect\setcounter{tocdepth}{0}}

\section{Introduction}\label{sec:intro}
A recent surge of interest in mapping class groups of infinite-type surfaces has prompted the emergence of a ``big" analog of $\Out(F_{n})$ as well. Algom-Kfir--Bestvina \cite{AB2021} propose that the appropriate analog is the group of self proper homotopy equivalences up to proper homotopy of a locally finite, infinite graph. 

One main difficulty of studying these ``big'' groups is that the classical approach of geometric group theory is not applicable. In particular, the mapping class groups of infinite-type surfaces and those of locally finite, infinite graphs are generally not finitely generated, and not even compactly generated. Fortunately, they are still Polish groups (separable and completely metrizable topological groups), to which Rosendal provides a generalized geometric group theoretic approach. The role of a finite or compact generating set is replaced with a \emph{coarsely bounded} (CB) generating set. For example, a group that admits a coarsely bounded generating set has a well-defined quasi-isometry type \cite[Theorem 1.2 and Proposition 2.72]{rosendal2022}, and a coarsely bounded group is quasi-isometric to a point. A group with a coarsely bounded neighborhood around the identity is said to be \emph{locally coarsely bounded}, which is equivalent to having a well-defined \emph{coarse equivalence} type, and is necessary to have a coarsely bounded generating set.
Using this framework, Mann--Rafi \cite{mann2023large} gave a classification of (tame) surfaces whose mapping class groups are coarsely bounded, locally coarsely bounded, and generated by a coarsely bounded set. This established a first step toward studying the coarse geometry of mapping class groups of infinite-type surfaces. Recently, Thomas Hill \cite{hill2023largescale} gave a complete classification of surfaces that have \emph{pure} mapping class groups with the aforementioned coarse geometric properties, without the tameness condition.

In the authors' previous work \cite{DHK2023}, we gave a complete classification of graphs with coarsely bounded, and locally coarsely bounded, pure mapping class groups, the subgroup of the mapping class group fixing the ends of the graph pointwise. In this paper, we provide the complete classification of infinite graphs with CB-generated pure mapping class groups, fulfilling our goal to provide a foundation for studying the coarse geometry of these groups. In the following statement, $E$ refers to the space of ends of the graph $\G$ and $E_{\ell}$ is the subset of ends accumulated by loops.

\begin{MAINTHM}\label{thm:CBgenClassification}
    Let $\G$ be a locally finite, infinite graph. Then its pure mapping class group, $\PMap(\G)$, is \emph{CB generated} if and only if either $\G$ is a tree, or satisfies both:
    \begin{enumerate}[label=\arabic*.]
        \item $\G$ has finitely many ends accumulated by loops, and
        \item there is no accumulation point in $E \setminus E_\ell$.
    \end{enumerate}
\end{MAINTHM}

\begin{RMK}\label{rmk:connectsum}
    Alternatively, we have a constructive description: $\PMap(\G)$ is CB generated if and only if $\G$ can be written (not necessarily uniquely) as a finite wedge sum of the four graphs from \Cref{fig:CBsummands}.
\end{RMK}

\begin{figure}[ht!]
	    \centering
		    \includegraphics[width=.6\textwidth]{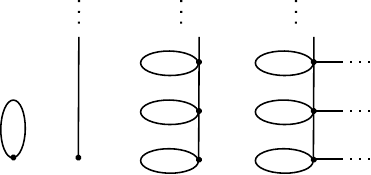}
		    \caption{Every graph with a CB-generated pure mapping class group can be written as a finite wedge sum of these four graphs. From left to right these are: a single loop, a single ray, a Loch Ness monster graph, and a Millipede monster graph.}
	    \label{fig:CBsummands}
\end{figure}
    
\Cref{table:fullclassification} illustrates the complete classification of graphs with CB, locally CB, and CB-generated pure mapping class group. Observe the trend that $\PMap(\G)$ admits more complicated coarse geometric properties when $\G$ has more complicated geometry.

\colorlet{clrCB}{blue!20}
\colorlet{clrLocCB}{yellow!20}
\colorlet{clrCBgen}{green!20}
\colorlet{clrNone}{pink!30}
\colorlet{clrNA}{gray!30}

\begin{table}[ht!]
\centering
\makebox[\textwidth][c]{\renewcommand{\arraystretch}{1.5}
\begin{tabular}{rlllllll}
    \toprule
         \multirow{2}{*}[-.3em]{\makecell[r]{\tiny$t=$ the number of\\\tiny components of $\G \setminus \G_c$ \\\tiny with infinite end spaces}}& \multicolumn{3}{c}{Finite rank $r$} &\phantom{}& \multicolumn{3}{c}{Infinite rank with $|E_\ell|=:n$}\\ \cmidrule{3-4} \cmidrule{6-8}
         && $r=0$\phantom{abc} & $r \in [1,\infty)$ && $n=1$ & $n \in [2,\infty)$ & $n=\infty$ \\
        \midrule \addlinespace
        \raisebox{2.8em}{$t=0$}
        &&\raisebox{.4em}{\makecell{CB\\~\adjustbox{valign=c}{\includegraphics[width=5em]{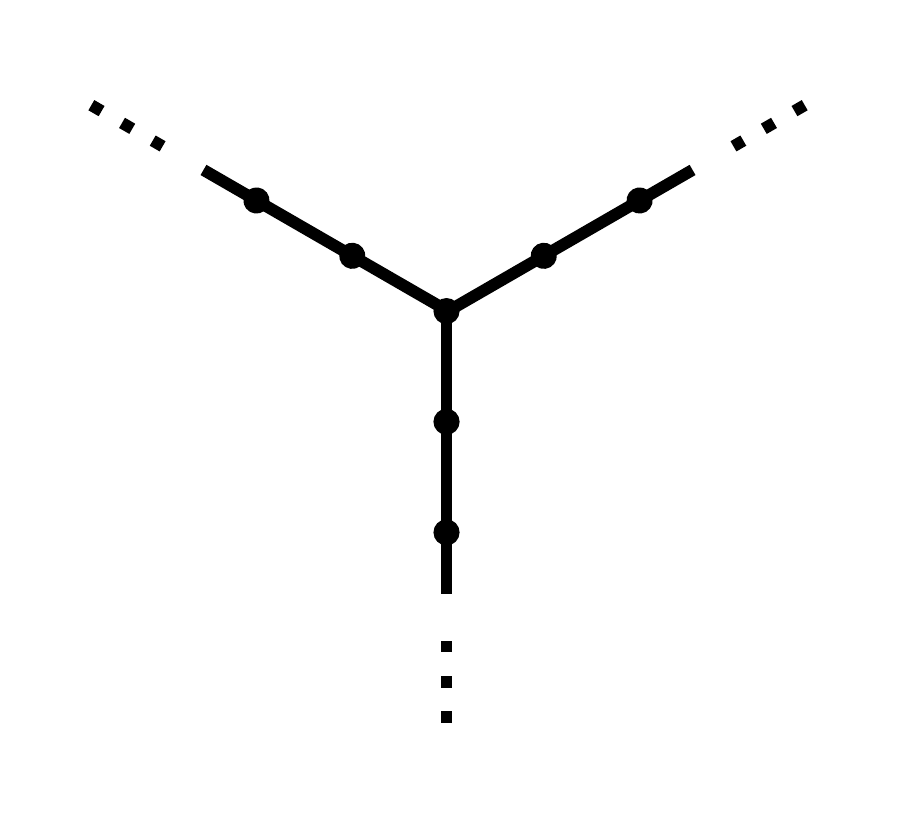}}}}\cellcolor{clrCB}
        & \begin{tabular}{@{}l@{}}\setlength\tabcolsep{1.5pt}%
        CB if $\G=\lasso$\cellcolor{clrCB}\\ \makecell{CB-gen otherwise\\~\adjustbox{valign=c}{\includegraphics[width=5em]{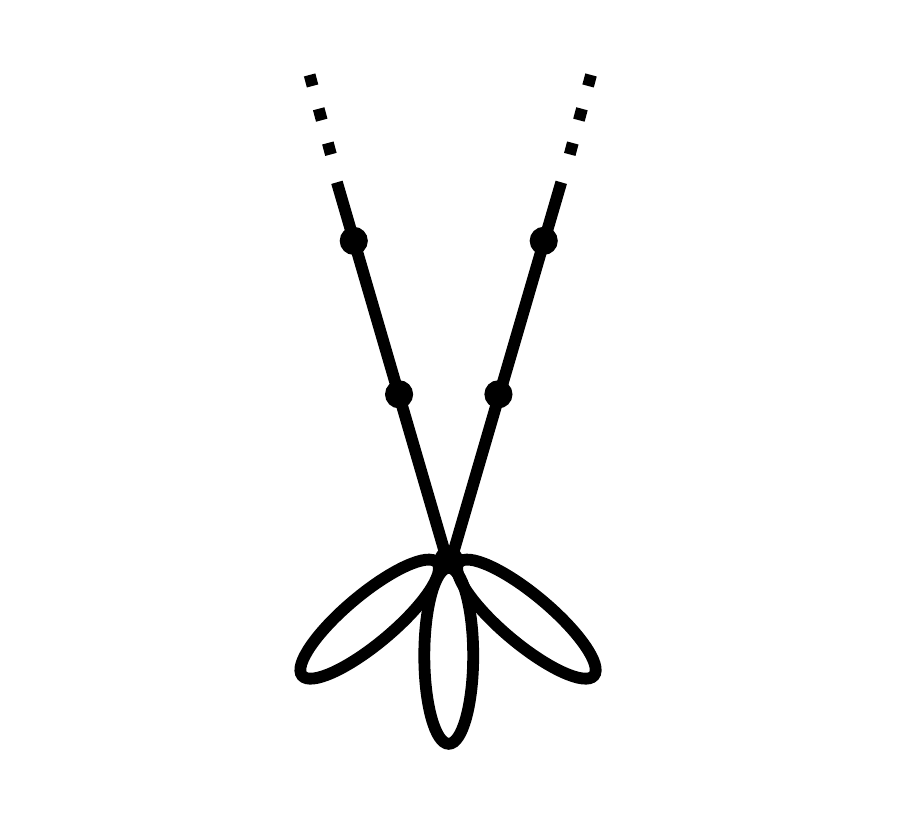}}}\cellcolor{clrCBgen}\end{tabular}
        & & \raisebox{.5em}{\makecell{\phantom{abcc.c}CB\phantom{abac\!\ c}\\~\adjustbox{valign=c}{\includegraphics[width=5em]{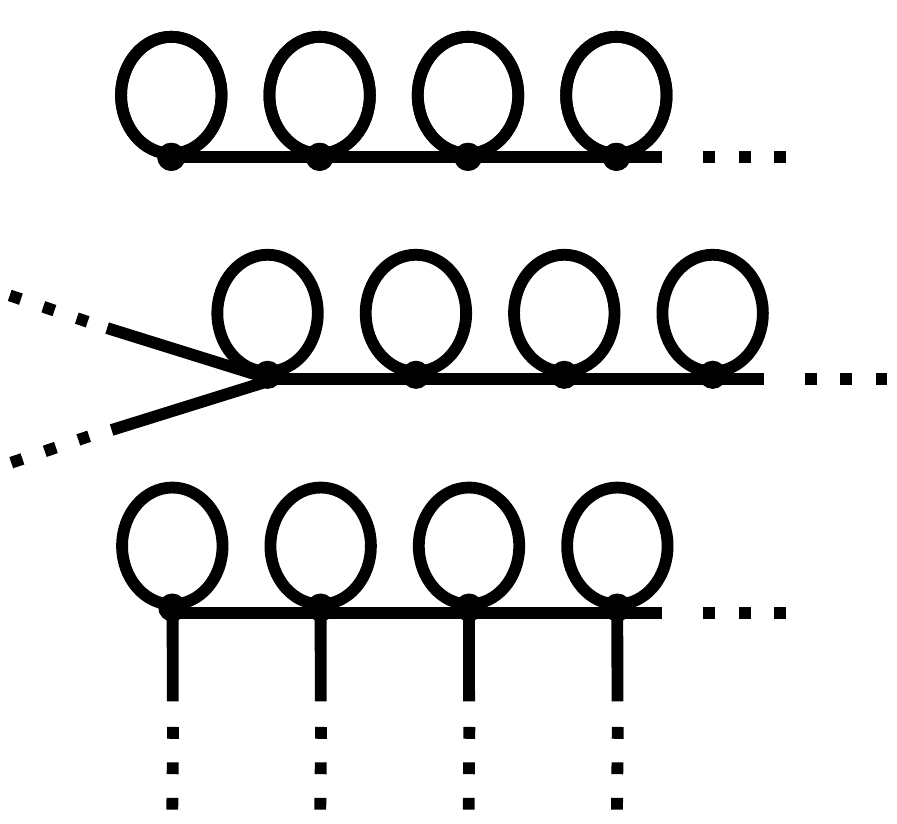}}}}\cellcolor{clrCB}
        & \raisebox{.5em}{\makecell{\phantom{a\!}CB-generated\phantom{a\!}\\~\adjustbox{valign=c}{\includegraphics[width=5em]{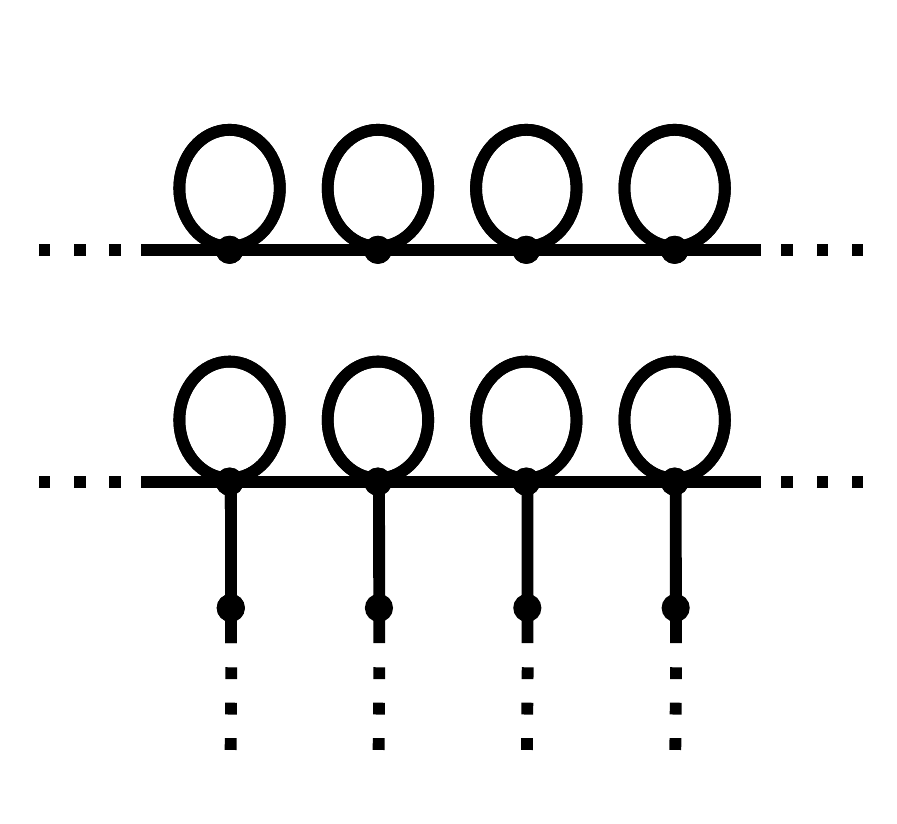}}}}\cellcolor{clrCBgen}
        & \raisebox{.5em}{\makecell{not locally CB\\~\adjustbox{valign=c}{\includegraphics[width=5em]{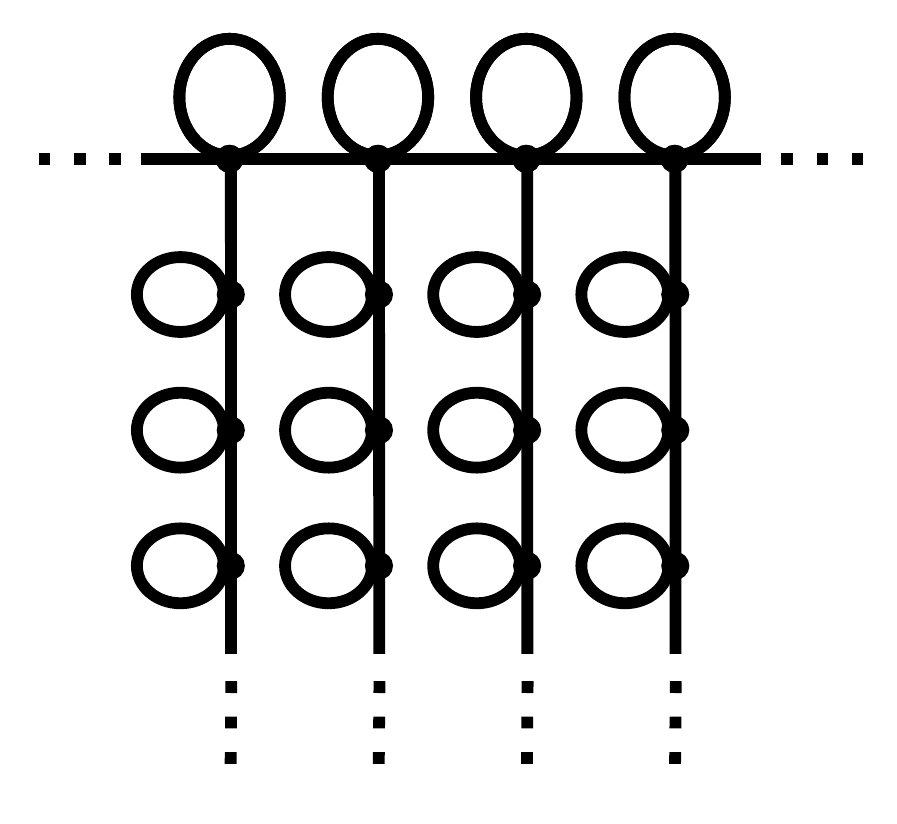}}}}\cellcolor{clrNone}\\[2ex]
        \raisebox{2.3em}{$t \in [1, \infty)$} & 
        & \makecell{CB\\~\adjustbox{valign=c}{\includegraphics[width=5em]{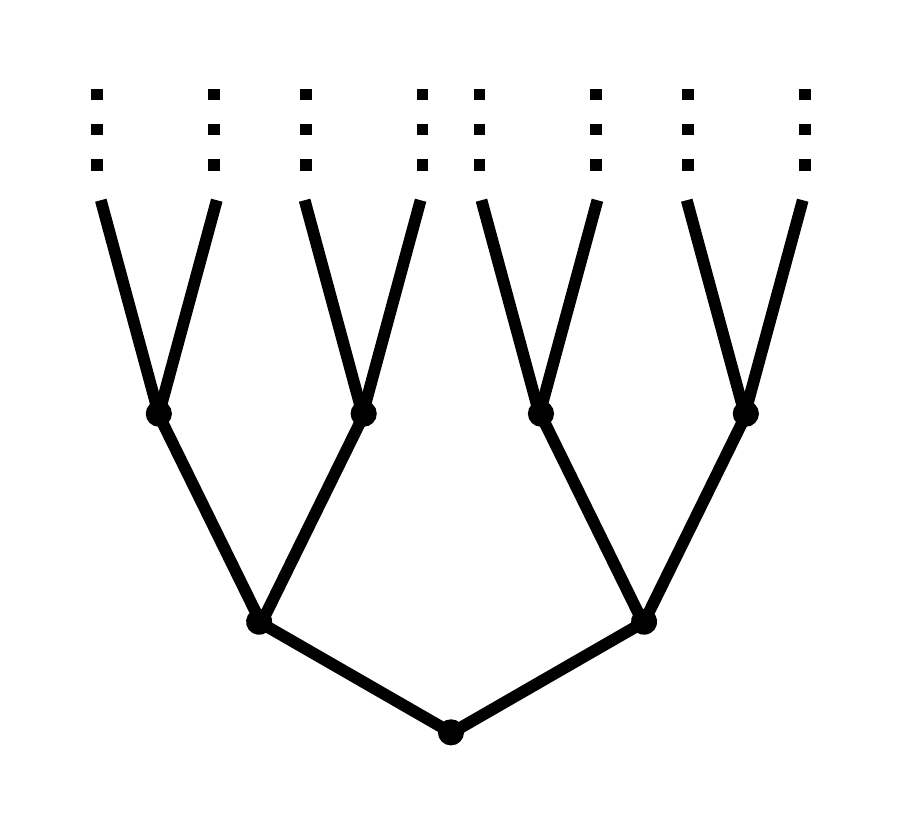}}}\cellcolor{clrCB}
        & \phantom{ab}\makecell{locally CB\\~\adjustbox{valign=c}{\includegraphics[width=5em]{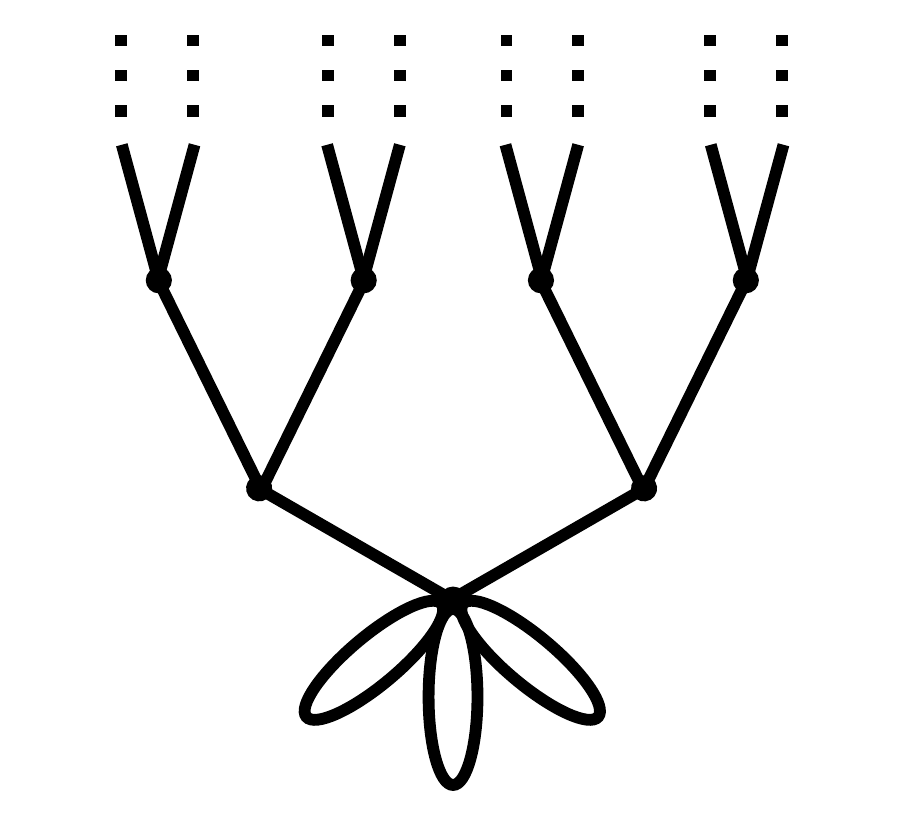}}}\phantom{ab\ }\cellcolor{clrLocCB}
        & &\makecell{\phantom{ab}locally CB\phantom{ab}\\~\adjustbox{valign=c}{\includegraphics[width=5em]{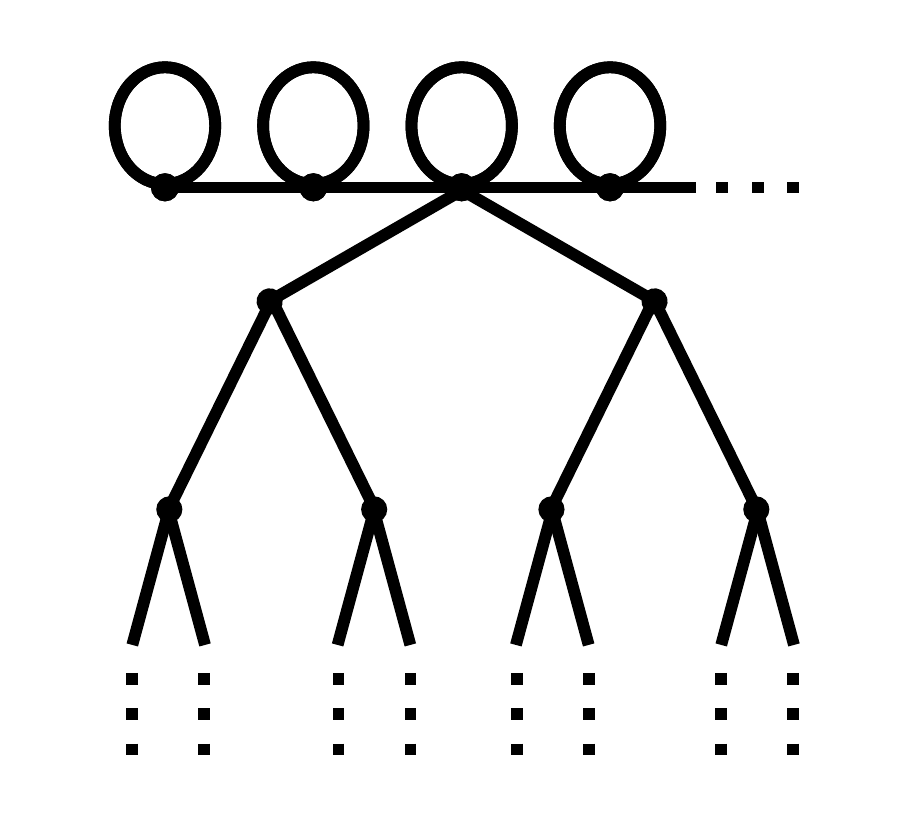}}}\cellcolor{clrLocCB}
        & \makecell{\phantom{ab}locally CB\phantom{ab}\\~\adjustbox{valign=c}{\includegraphics[width=5em]{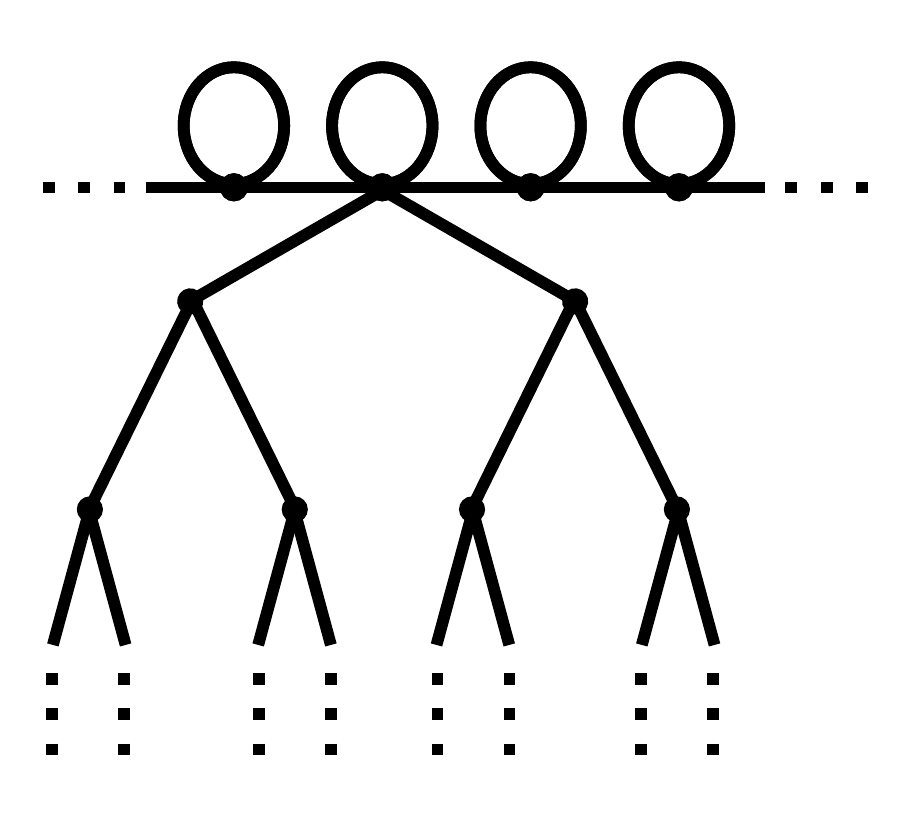}}}\cellcolor{clrLocCB}
        & \makecell{not locally CB\\~\adjustbox{valign=c}{\includegraphics[width=5em]{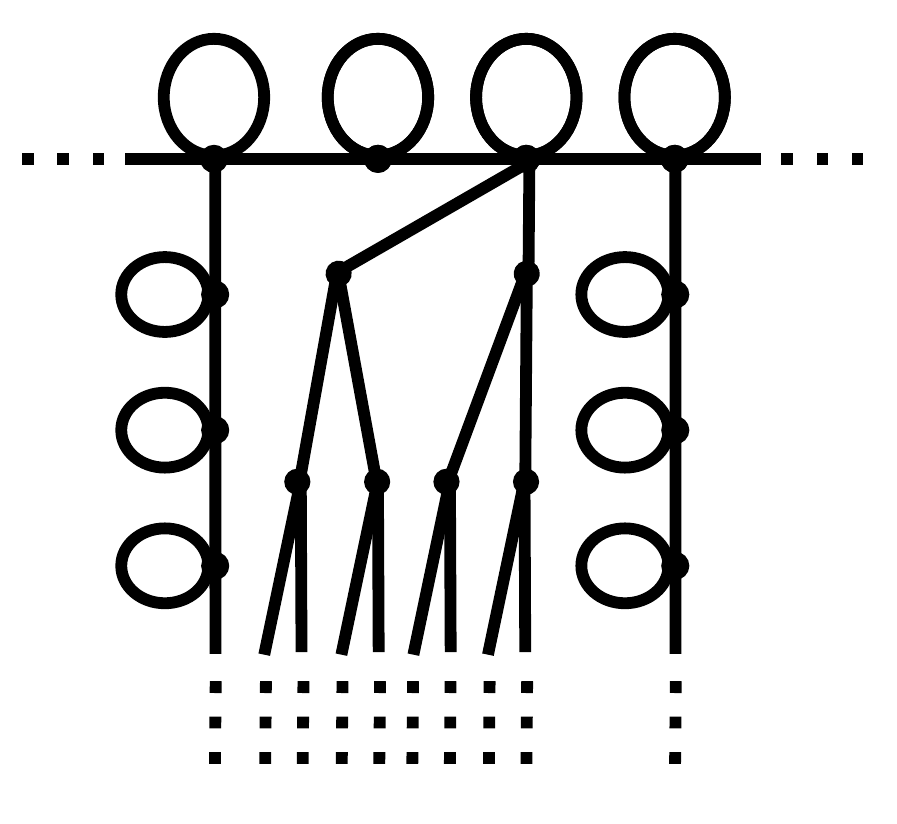}}}\cellcolor{clrNone}\\[1.5ex]
        \raisebox{2.3em}{$t = \infty$} &        
        & \raisebox{2.2em}{\makecell[c]{\phantom{abc}N/A}}\cellcolor{clrNA}
        & \raisebox{2.2em}{\makecell[c]{\phantom{abcdefg}N/A}}\cellcolor{clrNA}
        & &\makecell{not locally CB\\~\adjustbox{valign=c}{\includegraphics[width=5em]{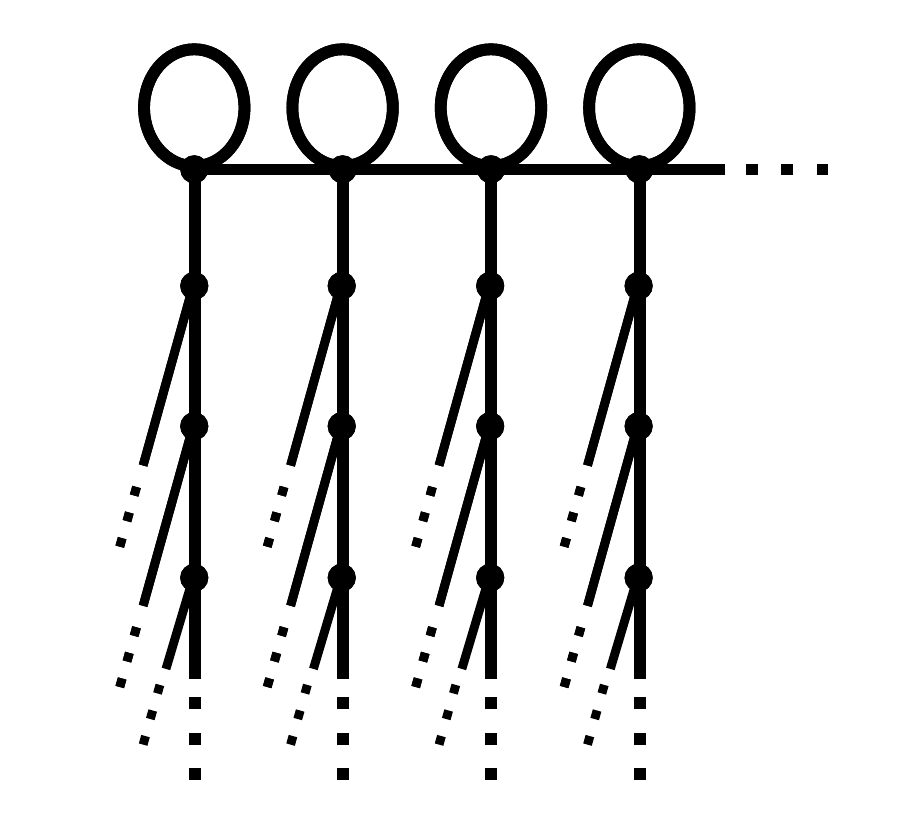}}} \cellcolor{clrNone}
        &\makecell{not locally CB\\~\adjustbox{valign=c}{\includegraphics[width=5em]{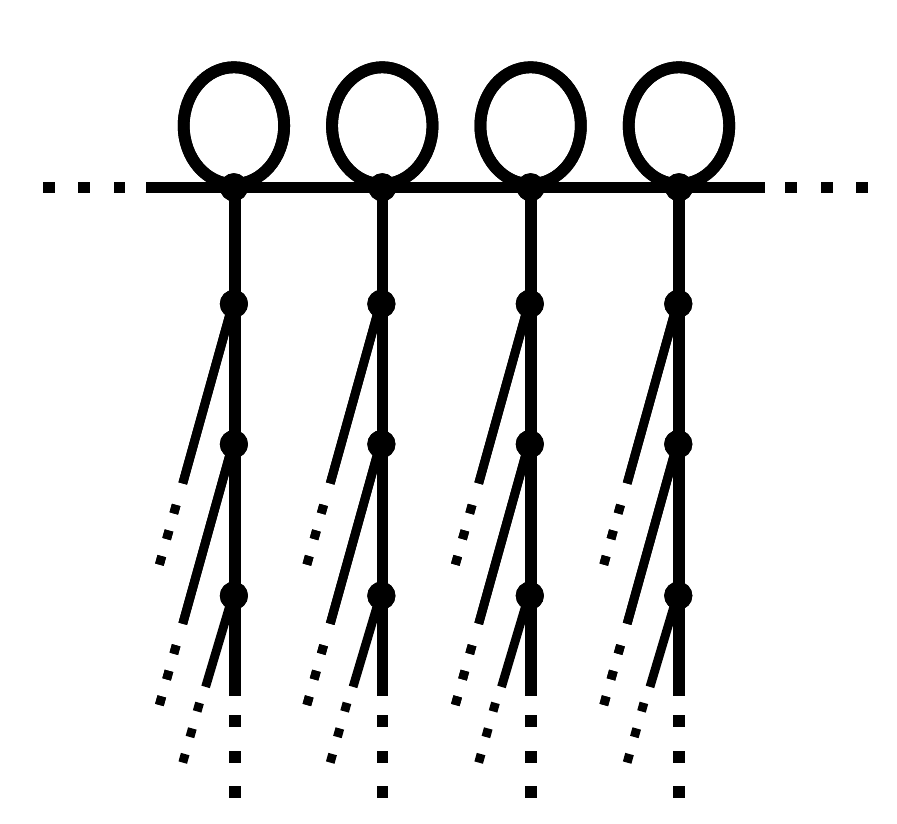}}} \cellcolor{clrNone}
        & \makecell{not locally CB\\~\adjustbox{valign=c}{\includegraphics[width=5em]{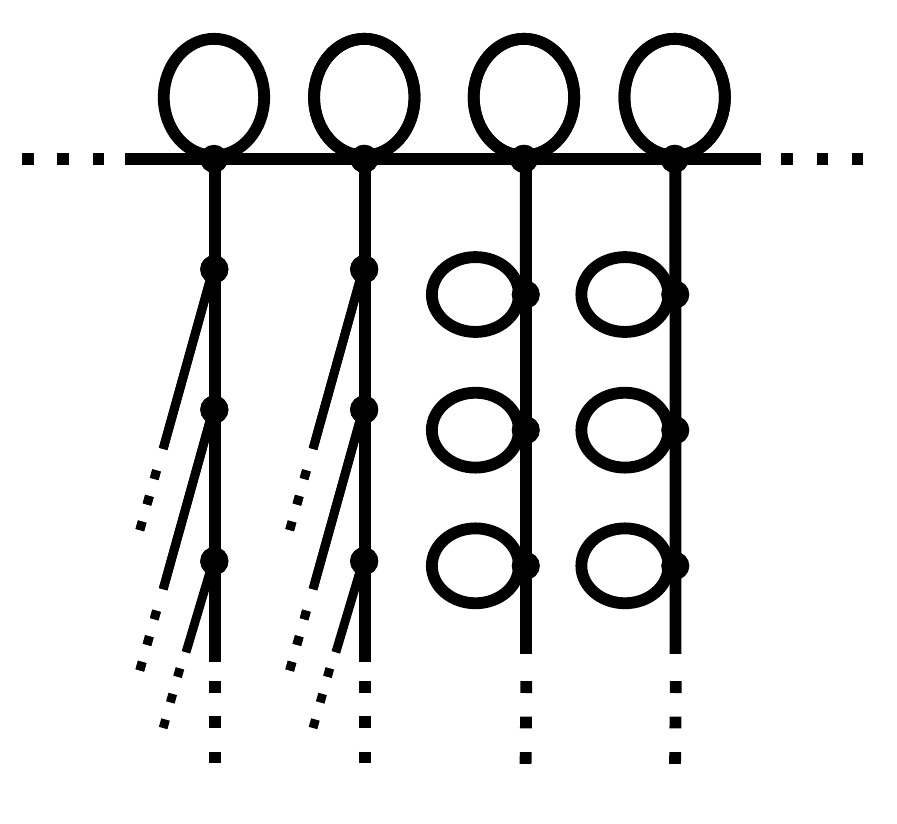}}}\cellcolor{clrNone}\\[1.5ex]
    \bottomrule
\end{tabular}
    } %End of makebox 
\caption{(\Cref{thm:CBgenClassification} and \cite{DHK2023}). Classification of $\G$ with CB, locally CB, and CB-generated $\PMap(\G)$. Here $\G_c$ is the smallest connected subgraph that contains all the immersed loops in $\G$. According to the implication CB $\Rightarrow$ CB-generated $\Rightarrow$ locally CB, each cell is filled with the most general property. For example, a cell labeled by CB-generated means that for such $\G$, $\PMap(\G)$ is CB-generated but not CB. }  \label{table:fullclassification}
\end{table}

The main tool used to prove \Cref{thm:CBgenClassification} is the following semidirect decomposition of the pure mapping class group (reminiscent of \cite[Corollary 4]{APV2020} in the surface setting):

\begin{MAINTHM}
\label{thm:semidirectprod}
Let $\G$ be a locally finite graph.
Let $\alpha = \max\{0,|E_\ell(\G)| - 1\}$ for $|E_\ell(\G)|<\infty$ and $\alpha = \aleph_0$ otherwise. Then we have the following short exact sequence,
\[
    1 \longto \overline{\PMap_c(\G)} \longto \PMap(\G) \longto \Z^\alpha \longto 1
\]
which splits. In particular, we have $\PMap(\G)=\overline{\PMap_c(\G)} \rtimes \Z^{\alpha}.$
\end{MAINTHM}

Here, $\overline{\PMapc(\G)}$ is the closure of the group of compactly supported mapping classes and $\Z^{\alpha}$ is generated by commuting loop shifts. 

As a corollary, we compute the rank of the first integral cohomology of $\PMap(\G)$. This allows us to see that the number of ends acummuluted by loops of a graph $\G$ is an algebraic invariant of $\PMap(\G)$.
\begin{MAINCOR}\label{cor:rankoffirstcohomology}
For every locally finite, infinite graph $\G$,
    \[
    \rk \left(H^1(\PMap(\G);\Z)\right) =
        \begin{cases}
            0 & \text{if $|E_\ell| \le 1$}, \\
            n-1 & \text{if $2 \le |E_\ell|=n< \infty$}, \\
            \aleph_0 & \text{otherwise}.
        \end{cases}
    \]
\end{MAINCOR}

We also show that $\PMap(\G)$ distinguishes graphs of finite rank from graphs of infinite rank. Recall a group is \emph{residually finite} if it can be embedded into a direct product of finite groups.
\begin{MAINTHM}\label{thm:PMapRF}
   $\PMap(\G)$ is residually finite if and only if $\G$ has finite rank.
 \end{MAINTHM}

A group satisfies the \textit{Tits Alternative} if every subgroup is either virtually solvable or contains a nonabelian free group.
Interestingly, it is exactly the graphs with $\PMap(\G)$ residually finite that satisfy the Tits Alternative.
 \begin{MAINTHM}\label{thm:TA}
    $\PMap(\G)$ satisfies the Tits Alternative if and only if $\G$ has finite rank.
 \end{MAINTHM}
These three results are steps towards determining when the isomorphism type of $\PMap(\G)$ determines the graph $\G$, as in the surface case \cite{BDR2020}. 
 
If $\G$ is the infinite rank graph with a single end (the Loch Ness monster graph) and $\G'$ is the wedge sum of $\G$ with a single ray, then the groups $\PMap(\G)$ and $\PMap(\G')$ inject into $\Out(F_{\infty})$ and $\Aut(F_{\infty})$, respectively, by \cite[Theorem 3.1 and Lemma 3.2]{AB2021}. Thus we immediately get the following corollary. We note that one can instead prove this directly, e.g. see \cite{OutFinfinity2011}.

\begin{MAINCOR}
    For $F_{\infty}$, the free group on a countably infinite set, $\Aut(F_{\infty})$ and $\Out(F_{\infty})$ are not residually finite and do not satisfy the Tits alternative. 
\end{MAINCOR}

\subsection*{Comparison with Surfaces} The statement of \Cref{thm:semidirectprod} is exactly the same as for pure mapping class groups of surfaces, seen in Aramayona--Patel--Vlamis \cite{APV2020}. Although the proof we give is similar in spirit as well, we have to make use of different tools. In \cite{APV2020} the authors make use of the \emph{homology of separating curves} on a surface and build an isomorphism between the first cohomology of the pure mapping class group and this homology group. For graphs, we do not have any curves to take advantage of. Instead we use partitions of the space of ends accumulated by loops. In order to make this precise and give this an algebraic structure we make use of the group of locally constant integral functions on $E_{\ell}(\G)$, i.e., the zeroth \cech{} cohomology of $E_{\ell}(\G)$, denoted as $\dC(E_{\ell}(\G))$. On a surface, any separating simple closed curve determines a partition of the end space. We can use this to show that the first cohomology groups of pure mapping class groups of graphs and surfaces are often in fact \emph{naturally} isomorphic. This also gives a slightly alternate proof of the main results in \cite{APV2020}. 

\begin{MAINCOR} \label{cor:surfacegraph}
    Let $S$ be an infinite-type surface of genus at least one and $\G$ a locally finite, infinite graph. If $E_{g}(S)$ is homeomorphic to $E_{\ell}(\G)$, then both $H^{1}(\PMap(S);\Z)$ and $H^{1}(\PMap(\G);\Z)$ are isomorphic to $\dC(E_{g}(S)) \cong \dC(E_{\ell}(\G))$.
\end{MAINCOR}

\subsection*{Rigidity Questions}
\sloppy
The results above fit naturally into a general question about (pure) mapping class groups of infinite graphs. Namely: ``How much does the group $\PMap(\G)$ determine the graph $\G$?" One can obtain more concrete questions by considering certain types of rigidity. We will focus on \textit{algebraic} and \textit{quasi-isometric} rigidity. In the finite-type setting, mapping class groups of surfaces and $\Out(F_n)$ are known to exhibit strong rigidity properties. Various results starting with Ivanov \cite{Ivanov1988} (see also \cite{McCarthy1986,Harer1986,BLM1983,Ivanov1997}) establish strong forms of algebraic rigidity and Behrstock\discretionary{--}{}{--}Kleiner\discretionary{--}{}{--}Minsky\discretionary{--}{}{--}Mosher \cite{BKMM2012} establish quasi-isometric rigidity for $\Map(S)$. For $\Out(F_n)$ we also have strong forms of algebraic rigidity from the work of Farb--Handel \cite{FH2007} building on \cite{Khramtsov1990,BV2000} (see also \cite{HW2020}). Quasi-isometric rigidity for $\Out(F_n)$ is still unknown.  

For infinite-type surfaces, the work of Bavard--Dowdall--Rafi \cite{BDR2020} established a strong form of algebraic rigidity \'{a} la Ivanov (see also \cite{HMV2018}). The question of quasi-isometric rigidity is still open, but Mann--Rafi \cite{mann2023large} give a classification of which mapping class groups of tame infinite-type surfaces have a well-defined quasi-isometry type and which of those are trivial. This allows one to begin to distinguish between some of the mapping class groups (see also \cite{schaffer2020}). 

One can ask the same rigidity questions for infinite graphs. The picture becomes less clear than in the surface case. In particular, trees fail to have algebraic rigidity for the \emph{pure} mapping class group, as they all have trivial pure mapping class group. Failure is also present for the full mapping class group. Let $T$ be the regular trivalent tree and let $T'$ be the wedge sum of $T$ with a single ray. Note that $E(T) = \cC$, a Cantor set, and $E(T') = \cC \sqcup \{*\}$, a Cantor set together with a single isolated point. Now we have that $\Map(T) = \Homeo(\cC)$ and $\Map(T') = \Homeo(\cC \sqcup \{*\})$. However, these two groups are isomorphic, as any homeomorphism fixes the extra end $*$ of $T'$. There are even more complicated examples of this failure of algebraic rigidity for mapping class groups of trees that come from work on Boolean algebras by McKenzie \cite{McKenzie1977automorphism} answering a rigidity conjecture of Monk \cite{Monk1975automorphism}. 

The results in this paper allow one to ask several natural rigidity questions for the pure mapping class groups of infinite graphs. We will restrict to some nice classes of graphs in order to state concrete questions. All of the following families of graphs are CB generated by \Cref{thm:CBgenClassification} and hence have a well-defined quasi-isometry type. Let $\G_{n}$ denote the graph with exactly $n$ ends, each of which is accumulated by loops. 

\begin{Q}
    Let $n,m \ge 2$. If $\PMap(\G_{n})$ is quasi-isometric to $\PMap(\G_{m})$, then does $n=m$? 
\end{Q}

By \Cref{cor:rankoffirstcohomology} we do know that $\PMap(\G_n)$ is algebraically isomorphic to $\PMap(\G_m)$ if and only if $n=m$. We can also use the fact that $\PMap(\G_1)$ is CB to see that $\PMap(\G_1)$ is not quasi-isometric to $\PMap(\G_n)$ for any $n \neq 1$. However, the general question is still open. In the authors' previous work \cite{DHK2023}, we computed asymptotic dimension for all of these groups. However, it is infinite for $n>1$. Therefore, in order to answer this question one would need to study and/or develop other ``big" quasi-isometry invariants. 

Instead of comparing the effect of changing the number of ends accumulated by loops, one could ask the same question for rays. Namely, let $\G_{n,r}$ denote the graph with $n$ ends accumulated by loops and $r$ rays. We start by asking for distinguishing features of ``no ray'' versus ``one ray.''

\begin{Q}
    Is $\PMap(\G_{n,0})$ quasi-isometric to $\PMap(\G_{n,1})$? 
\end{Q}

In fact, here we do not even know algebraic rigidity. 

\begin{Q}
    Is $\PMap(\G_{n,0})$ isomorphic to $\PMap(\G_{n,1})$?
\end{Q}

The other large family of graphs with CB-generated pure mapping class groups are the finite-type ones. Let $\Omega_{n,r}$ denote the graph of finite rank $n$ with $r<\infty$ rays attached. We know that no $\PMap(\Omega_{n,r})$ is isomorphic to any $\PMap(\G_{m})$ by using either residual finiteness \Cref{thm:PMapRF} or the Tits alternative \Cref{thm:TA}. We do not know if any of them are quasi-isometric. Note that $\PMap(\Omega_{n,r})$ is always finitely generated, but this does not preclude it from being quasi-isometric to an uncountable group.

\begin{Q}
    Is $\PMap(\Omega_{m,r})$ ever quasi-isometric to $\PMap(\G_{n})$, for $m,r,n>1$? 
\end{Q} 

\subsection*{Outline} In \Cref{sec:prelims}, we give background on mapping class groups of infinite graphs, examples of elements in the pure mapping class group, and coarse geometry of groups. In \Cref{sec:semidirect}, we prove our semidirect product decomposition, \Cref{thm:semidirectprod}. We also prove \Cref{cor:rankoffirstcohomology} in \Cref{ssec:firstcohomology}. By exploiting the semidirect decomposition of $\PMap(\G)$, we prove the CB-generation classification, \Cref{thm:CBgenClassification}, in \Cref{sec:CBgen}. In \Cref{sec:RF} and \Cref{sec:TitsAlternative}, we finish by proving the residual finiteness characterization \Cref{thm:PMapRF} and Tits alternative characterization \Cref{thm:TA}.

\subsection*{Acknowledgments}

Thank you to Mladen Bestvina for providing an idea of the proof for \Cref{lem:zeroflux} and suggestion to use the zeroth \cech~ cohomology to prove \Cref{lem:zerofluxbasis}. We also thank Priyam Patel for helpful conversations toward \Cref{sec:RF} and \Cref{thm:PMapRF}, along with answering questions regarding \cite{PatelVlamis} and \cite{APV2020}. Also we thank Camille Horbez for clarifying the proof of \Cref{fact:vsol}. 

The first author was supported in part by the Fields Institute for Research in Mathematical Sciences, NSF DMS--1745670, and NSF DMS--2303262. The second author was supported by NSF DMS--2303365. The third author acknowledges the support from the University of Utah Graduate Research Fellowship. 

\tableofcontents 

\addtocontents{toc}{\protect\setcounter{tocdepth}{3}}
\section{Preliminaries}\label{sec:prelims}

\subsection{Mapping class groups of infinite graphs}

Let $\G$ be a locally finite, infinite graph. 
Informally, an \emph{end} of a graph is a way to travel to infinity in the graph. The space of ends (or, the end space), denoted by $E(\G)$, is defined as:
\[
    E(\G) = \varprojlim_{K \subset \G}\pi_0(\G \setminus K),
\]
where $K$ runs over compact sets of $\G$ in the inverse limit. Then each element of $E(\G)$ is called an \textbf{end} of $\G$. An end $e$ of $\G$ is said to be \textbf{accumulated by loops}  if the sequence of complementary components in $\G$ corresponding to $e$ only consists of infinite rank graphs. Colloquially, if one continues to see loops along the way to $e$. We denote by $E_\ell(\G)$ the set of ends of $\G$ accumulated by loops. Note $E_\ell(\G)$ is a closed subset of $E(\G)$, and $E(\G)$ can be realized as a closed subset of a Cantor set (hence so is $E_\ell(\G)$). We say that the \textbf{characteristic triple} of $\G$ is the triple $(r(\G),E(\G),E_{\ell}(\G))$, where $r(\G) \in \Z_{\geq 0 } \cup \{\infty\}$ is the rank of $\pi_{1}(\G)$.

Now we define the mapping class group of a locally finite, infinite graph $\G$. Recall that a map is \textbf{proper} if the pre-image of every compact set is compact.
\begin{DEF} \cite{AB2021}
    The \textbf{mapping class group} of $\G$, denoted $\Map(\G)$, is the group of
    proper homotopy classes of proper homotopy equivalences of $\G$. The
    \textbf{pure mapping class group}, denoted $\PMap(\G)$, is the closed
    subgroup consisting of maps that fix the ends of $\G$ pointwise. More precisely,
    it is the kernel of the action of $\Map(\G)$ on the end space
    $(E(\G),E_\ell(\G))$ by homeomorphisms, hence fitting into the following short exact sequence:
    \[
      1 \longto \PMap(\G) \longto \Map(\G) \longto \Homeo(E,E_\ell) \longto 1
    \]
\end{DEF}

 When $E(\G) \setminus E_\ell(\G)$ is nonempty and compact, we can further decompose $\PMap(\G)$ into subgroups of  \emph{core maps} and of \emph{ray maps}. To state the result, we need to introduce a few concepts. 
 \begin{DEF} 
     Let $\G$ be a locally finite, infinite graph. Denote by $\G_c$ the \textbf{core graph} of $\G$, the smallest connected subgraph of $\G$ that contains all immersed loops in $\G$. When $E(\G) \setminus E_\ell(\G)$ is nonempty, pick $e_0 \in E(\G) \setminus E_\ell(\G)$ and denote by $\G_c^*$ be the subgraph consisting of $\G_c$ and a choice of embedded ray in $\G$ limiting to $e_0$ such that the ray intersects $\G_c$ in exactly one point.

     Define $\pi_1(\G_c^*,e_0)$ to be the set of proper homotopy equivalence classes of lines in $\G_c^*$, both ends of which limit to $e_0$. We endow it with a group structure by concatenation. This group is naturally isomorphic to $\pi_{1}(\G_{c}^{*},p)$ for any choice of basepoint $p \in \G_{c}^{*}$. Finally, define $\cR$ as the group of maps $h: E(\G) \to \pi_1(\G_c^*,e_0)$ such that
     \begin{enumerate}[label=(\roman*)]
         \item $h(e_0) =1$, and
         \item $h$ is locally constant,
     \end{enumerate}
     where the group operation is the pointwise multiplication in $\pi_1(\G_c^*,e_0)$.
 \end{DEF}

 We have the following decomposition of $\PMap(\G)$:

\begin{PROP}[{\cite[Corollary 3.9]{AB2021}}]
\label{prop:RsubgroupDecomposition}
    Let $\G$ be a locally finite, infinite graph with $E(\G) \setminus E_\ell(\G)$ nonempty and compact. Then
    \[
        \PMap(\G) \cong \cR \rtimes \PMap(\G_c^*).
    \]
    In particular, when $\G$ has finite rank $n \ge 0$ and finitely many ends, say $|E(\G)|=e$, then
    \[
        \PMap(\G) \cong
        \begin{cases}
            \Out(F_n), & \text{if $e=0$,} \\
            F_n^{e-1} \rtimes \Aut(F_n), & \text{if $e \ge 1$.}
        \end{cases}
    \]
\end{PROP}

\begin{RMK}\label{rmk:compactMCG}
    Any time $K$ is a connected, compact subgraph of a locally finite, infinite graph $\G$, we use $\PMap(K)$ to refer to the group of proper homotopy equivalences of $K$ that fix $\partial K$ pointwise up to proper homotopy fixing $\partial K$. This group is naturally isomorphic to the group $\PMap(\tilde{K})$ where $\tilde{K}$ is the graph $K$ together with a ray glued to each point in $\partial K$. Applying the above proposition we see that $\PMap(K)$ is always of the form $F_{n}^{e-1} \rtimes \Aut(F_{n})$ for some $n$ and $e$ because $K$ is always a proper subset of $\G$, so $\partial K$ is nonempty.
\end{RMK}

The pure mapping class group $\PMap(\G)$ records the internal symmetries of
$\G$. Contractible graphs (trees) have no internal
symmetries. This follows from the work of Ayala--Dominguez--M{\'a}rquez--Quintero \cite{ayala1990proper}. They give a proper homotopy equivalence classification of locally finite, infinite graphs.

\begin{THM}[{\cite[Theorem 2.7]{ayala1990proper}}] \label{thm:PHEclassification}
    Let $\G$ and $\G'$ be two locally finite graphs of the same rank. A homeomorphism of end spaces $(E(\G),E_\ell(\G)) \to (E(\G'),E_\ell(\G'))$ extends to a proper homotopy equivalence $\G \to \G'$. If $\G$ and $\G'$ are trees, then this extension is unique up to proper homotopy.
\end{THM}

The second statement of \Cref{thm:PHEclassification} implies the following.

\begin{PROP}
\label{prop:treePMap}
  Let $\G$ be a locally finite, infinite graph with $\pi_1(\G)=1$. Then $\PMap(\G)=1$.
\end{PROP}

In \cite{AB2021} the authors endow $\Map(\G)$ with the compact-open topology and show that this gives $\Map(\G)$, and hence $\PMap(\G)$, the structure of a Polish group. A neighborhood basis about the identity for the topology is given by sets of the form 
\begin{align*}
    \cV_{K} &\defeq \{[f] \in \Map(\G)\vert~ \exists f' \in [f] \text{ s.t. } f'\vert_{K} = \id \text{ and } \\ &f' \text{ preserves the complementary components of $K$ setwise}\}
\end{align*}
where $K$ is a compact subset of $\G$. 

Recall the \textbf{support} of a continuous map $\phi:X \to X$ is the closure of the set of points $x \in X$ such that $\phi(x) \neq x$. The group of compactly supported mapping classes, denoted by $\PMapc(\G)$, is the subgroup of $\PMap(\G)$ consisting of classes that have a compactly supported representative. Its closure in this topology is denoted by $\overline{\PMapc(\G)}$ and it is a closed (hence Polish) subgroup of $\PMap(\G)$. 

As proper homotopy equivalences are not necessarily injective, unlike homeomorphisms, we need the following alternate notion of support for a proper homotopy equivalence.
\begin{DEF}
    We say that $[f] \in \Map(\G)$ is \textbf{totally supported} on $K \subset \G$ if there is a representative $f' \in [f]$ so that $f'(K) = K$ and $f'\vert_{\G \setminus K} = \id$. 
\end{DEF}

To see how a proper homotopy equivalence can have different support and total support, consider a rose graph with two loops labeled by $a_1$ and $a_2$. Then a (proper) homotopy equivalence mapping $a_1$ to $a_1a_2$ which is the identity elsewhere is supported on $a_1$, but not totally supported on $a_1$. It is totally supported on $a_1 \cup a_2$.
This is in contrast with homeomorphisms on surfaces, where $f$ is supported on $K$ if and only if $f$ is totally supported on $K$. 

As mapping class groups of graphs are independent of the proper homotopy equivalence representative of the graph, it is often useful to consider a `standard' representative within a proper homotopy equivalence class of graphs.
\begin{DEF}\label{def:standardForm}
    A locally finite graph, $\G$, is in \textbf{standard form} if $\G$ is a tree with loops attached at some of the vertices. We endow $\G$ with the path metric that assigns each edge length $1$.
\end{DEF}

We also give special names to specific graphs that we will reference often.

\begin{DEF} \label{def:monsters}
    The \textbf{Loch Ness Monster} graph is the graph with characteristic triple $(\infty,\ \{*\},\ \{*\})$. The \textbf{Millipede Monster} graph is the graph with characteristic triple $(\infty,\ \{0\} \cup \{\frac{1}{n}\ |\ n \in \Z_{>0}\},\ \{0\})$. A \textbf{monster} graph refers to either one of these. 
\end{DEF}

\subsection{Elements of $\PMap(\G)$} \label{ssec:elements}
Here we give a brief treatment of elements of $\PMap(\G)$. For more detailed definitions with examples, see \cite[Section 3]{DHK2023}.

\subsubsection{Loop swaps}
   A Loop swap is an order 2 proper homotopy equivalence induced from a transposition automorphism of a free group. It is totally supported on a compact set. More precisely, we define it as follows.

    \begin{DEF}
        Let $\G$ be a locally finite graph in standard form with $\rk \G \ge 2$. Let $A$ and $B$ be disjoint finite subsets of loops such that $|A|=|B|$. Then the \textbf{loop swap} $\cL(A,B)$ is a proper homotopy equivalence induced from the group isomorphism on $\pi_1(\G)$ swapping the free factors corresponding to $A$ and $B$.
        
        More concretely, pick a basepoint $p \in \G$ and collapse each maximal tree of the subgraphs corresponding to $A$ and $B$ in $\pi_1(\G,p)$. This results in two roses of $|A|=|B|$ petals. Then swap the two roses, followed by blowing up each rose to the original subgraph. Define $\cL(A,B)$ as the composition of these three maps. Note $\cL(A,B) \in \PMapc(\G).$
    \end{DEF}
    As mentioned above, loop swaps of a graph correspond to the transposition free group automorphisms, which are part of generating set for $\Aut(F_n)$. (See \Cref{sec:CBgen}).

\subsubsection{Word maps}
    Next, word maps, which are the most diverse kind of mapping classes among the three kinds of maps introduced in this section.
\begin{DEF}
    Let $\G$ be a locally finite graph with $\rk \G \ge 1$, with a base point $p \in \G$. Let $w \in \pi_1(\G,p)$ and $I$ be an interval in an edge of $\G$. Then the \textbf{word map} $\varphi_{(w,I)}$ is a proper homotopy equivalence that maps $I$ to a path in $\G$ determined by $w \in \pi_1(\G,p)$ and is the identity outside of $I$.
\end{DEF}

    See \cite[Section 3.3]{DHK2023} for a careful construction of these maps. Note $\varphi_{(w,I)}$ is supported on $I$, but in general not \emph{totally} supported on $I$. Rather, it is totally supported on the compact set that is the union of $I$ with a path in $\G$ determined by $w \in \pi_1(\G,p)$.

    The following two properties of word maps will be important in \Cref{sec:CBgen}.
    \begin{LEM}[{\cite[Lemma 3.5]{DHK2023}}] \label{lem:compositionLaw}
        If $I$ is contained in an edge of $\G \setminus \G_c$ and $w_1,w_2$ are elements in $\pi_1(\G,p)$, then
        \[
            [\varphi_{(w_1,I)} \circ \varphi_{(w_2,I)}] = [\varphi_{(w_1w_2,I)}]
        \]
        in $\PMap(\G)$.
    \end{LEM}

    \begin{LEM}[{\cite[Lemma 3.10]{DHK2023}}] \label{lem:conjugationLaw}
        Let $I$ be an interval of $\G$ which is outside of $\G_c$, and $\psi \in \PMap(\G)$ be totally supported on a compact subgraph of $\G_{c}$. Then
        \[
            \psi \circ [\varphi_{(w,I)}] \circ \psi^{-1} = [\varphi_{(\psi_*(w),I)}].
        \]
    \end{LEM}
    In particular, we can use \Cref{lem:conjugationLaw} when $\psi$ is a loop swap.

\subsubsection{Loop shifts}
Loop shifts are to graphs as handle shifts are to surfaces, which were introduced in Patel--Vlamis \cite{PatelVlamis}. We first define a loop shift on the standard form of the graph $\Lambda$, the graph with characteristic triple $(\infty, \{e_{-},e_{+}\},\{e_{-},e_{+}\})$. (See \Cref{def:standardForm}.) Embed $\Lambda$ in $\R^2$ by identifying the maximal tree with the $x$-axis such that $e_{\pm}$ is identified with $\pm \infty$ of the $x$-axis, and each vertex is identified with an integer point in the $x$-axis. Identify the loops with the circles $\{(x-n)^2 + (y-\frac{1}{4})^2 = \frac{1}{16}\}_{n \in \Z}$. Note these circles are tangent to the integer points $\{(n,0)\}_{n \in \Z}$, thus representing the loops in $\Lambda$. Now define the \textbf{primitive loop shift} $h$ on $\Lambda$ as the horizontal translation $x \mapsto x+1$. One can also omit some loops from $\Lambda$ and define the loop shift to avoid those loops. For a more general definition, see \cite[Section 3.4]{DHK2023}.

\begin{DEF}
Now we define the loop shift on a locally finite, infinite graph $\G$ with $|E_\ell| \ge 2$. Pick two distinct ends $e_{-}, e_+ \in E_\ell(\G)$ accumulated by loops. By considering a standard form of $\G$, we can find an embedded ladder graph $\Lambda$ in $\G$ such that $e_{\pm}$ is identified with $e_{\pm}$ of $\Lambda$, respectively. Now define the \textbf{primitive loop shift} $h$ on $\G$ associated to $(e_-, e_+)$ as the proper homotopy equivalence induced from the primitive loop shift on the embedded ladder graph $\Lambda$. For the rest of the graph, define $h$ to be the identity outside of the $\frac{1}{2}$-neighborhood of $\Lambda$ and interpolate between the shift and the identity on the $\frac{1}{2}$-neighborhood.

Finally, a proper homotopy equivalence $f$ on $\G$ is a \textbf{loop shift} if $f = h^n$ for some primitive loop shift $h$ and $n \in \Z \setminus \{0\}$.
\end{DEF} 

\subsection{Coarse geometry of groups}

\begin{DEF}
    Let $A$ be a subset of a topological group $G$. Then $A$ is \textbf{coarsely bounded (CB)} in $G$ if for every continuous isometric action of $G$ on a metric space, every orbit is bounded.
\end{DEF}

We say a group is \textbf{CB-generated} if it has an algebraic generating set that is CB. Similarly, a group is \textbf{locally CB} if it admits a CB neighborhood of the identity. In \Cref{sec:CBgen}, we will construct a CB-generating set for the pure mapping class groups of certain graphs, proving the if direction of \Cref{thm:CBgenClassification}. On the other hand, we have previously classified which graphs have CB or locally CB mapping class groups: 

\begin{THM}[{\cite[Theorem A, D]{DHK2023}}]
\label{THM:CBclassification}
    Let $\G$ be a locally finite, infinite graph. Then its pure mapping class group $\PMap(\G)$ is \emph{coarsely bounded} if and only if one of the following holds:
    \begin{itemize}
        \item $\G$ has rank zero, or
        \item $\G$ has rank one, and has one end, or
        \item $\G$ is a monster graph with finitely many rays attached.
    \end{itemize}
    Moreover, $\PMap(\G)$ is \emph{locally coarsely bounded} if and only if one of the following holds:
        \begin{itemize}
        \item $\G$ has finite rank, or
        \item $\G$ satisfies both:
    \begin{enumerate}[label=\arabic*.]
            \item $|E_\ell(\G)| <\infty$, and
            \item only finitely many components of $\G \setminus \G_c$ have infinite end spaces.
        \end{enumerate}
    \end{itemize}
\end{THM}

\begin{RMK}
    Mirroring the constructive description in \Cref{rmk:connectsum} of the CB-generated $\PMap(\G)$ classification, we can alternatively characterize the locally CB condition as: $\PMap(\G)$ is locally CB if and only if $\G$ can be written as a finite wedge sum of single loops, monster graphs, and trees. 
\end{RMK}

After confirming that a group is CB-generated, the Rosendal framework enables the exploration of the group through the lens of coarse geometry.

\begin{THM}\cite[Theorem 1.2, Proposition 2.72]{rosendal2022}
    Let $G$ be a CB-generated Polish group. Then $G$ has a well-defined quasi-isometry type. Namely, any two CB-generating sets for $G$ give rise to quasi-isometric word metrics on $G$.
\end{THM}

\section{Semidirect product structure and cohomology} \label{sec:semidirect}

In this section, we prove:

\begin{THM}[\Cref{thm:semidirectprod}, revisited]
\label{thm:semidirectprod_revisited}
Let $\G$ be a locally finite graph.
Let $\alpha = \max\{0,|E_\ell(\G)| - 1\}$ for $|E_\ell(\G)|<\infty$ and $\alpha = \aleph_0$ otherwise. Then we have the following short exact sequence,
\[
    1 \longto \overline{\PMap_c(\G)} \longto \PMap(\G) \longto \Z^\alpha \longto 1
\]
which splits. In particular, we have $\PMap(\G)=\overline{\PMap_c(\G)} \rtimes \Z^{\alpha}.$
\end{THM}

The map to $\Z^{\alpha}$ is defined using \emph{flux maps}, which were first defined for locally finite, infinite graphs in \cite{DHK2023}. We quickly treat the case when the graph has at most one end accumulated by loops in \Cref{ssec:atmostOneEnd}. Then in \Cref{ssec:fluxmaps}, we recap the necessary definitions for flux maps and further expand on their properties. In \Cref{ssec:semidirect}, we characterize $\overline{\PMap_c(\G)}$ as the common kernel of all flux maps (\Cref{thm:fluxzeromaps}), which provides the left side of the desired splitting short exact sequence. Then in \Cref{ssec:SpaceOfFluxmaps}, we construct the other side of the short exact sequence by finding a section, proving \Cref{thm:semidirectprod}. This requires us to study the space of flux maps, which is done in the same subsection. As an application, in \Cref{ssec:firstcohomology} we compute the first integral cohomology of $\PMap(\G)$. Finally, we show the same approach could have been applied to infinite-type surfaces in \Cref{ssec:surfacescohom} to recover the surface version of \Cref{thm:semidirectprod} by Aramayona--Patel--Vlamis \cite{APV2020}, by showing that there is a natural isomorphism between the first cohomology of the pure mapping class groups of infinite-type surfaces and infinite graphs.

\subsection{The case $|E_\ell| \le 1$} \label{ssec:atmostOneEnd}
\begin{PROP} \label{prop:noloopends}
    Let $\G$ be a locally finite, infinite graph with $|E_\ell| \le 1$. Then $\PMap(\G) = \overline{\PMap_c(\G)}$. Furthermore, if $\lvert E_{\ell} \rvert =0$, then $\PMap(\G)=\PMap_c(\G)$.
\end{PROP}

\begin{proof}
    The case when $|E_\ell(\G)|=1$ is the result of \cite[Corollary 4.5]{DHK2023}. Now we assume $|E_\ell(\G)|=0$, i.e., $\G$ has finite rank. 
    
    Let $f \in \PMap(\G)$. Because $f$ is proper, $f^{-1}(\G_{c})$ is compact. Thus, there is some connected compact set $K$ such that $\G_{c} \cup f^{-1}(\G_{c}) \subset K$. Now $f\vert_{\G \setminus K}$ is a proper homotopy equivalence between two contractible sets and thus $f$ can be homotoped to be totally supported on $K$. Hence, we conclude $f \in \PMap_{c}(\G)$.
\end{proof}

\subsection{Flux maps}\label{ssec:fluxmaps}
 We begin the case when $|E_\ell| \ge 2$, where the flux maps come onto the scene.
Here we recap the definitions and properties of flux maps developed in \cite[Section 7]{DHK2023}.

Let $\G$ be a locally finite, infinite graph with $|E_\ell| \ge 2$. For each nonempty, proper, clopen subset $\cE$ of $E_\ell$, we will construct a flux map $\Phi_\cE$, which will evaluate to 1 for every primitive loop shift that goes from an end in $E_\ell \setminus \cE$ to an end in $\cE$. We fix such a subset $\cE$ for this discussion.

After potentially applying a proper homotopy equivalence, we can put $\G$ into a standard form so that there is a maximal tree $T$ and a choice of $x_0$ in $T$ such that $\G \setminus \{x_0\}$ defines a partition of the ends that is compatible with the partition $\cE \sqcup (E_\ell \setminus \cE)$ of $E_\ell$. That is, the components of $\G \setminus \{x_{0}\}$ determine a partition $E = \bigsqcup_{i=1}^{m} \cF_{i}$ so that we can rewrite as $\cE = \bigsqcup_{i=1}^{k} (\cF_{i} \cap  E_{\ell})$ and $E_{\ell} \setminus \cE = \bigsqcup_{i=k+1}^{m} (\cF_i \cap E_{\ell})$.

Now we group the components of $\G \setminus \{x_{0}\}$ by the set $\cE$. Let $\G_{+}$ and $\G_{-}$ be the unions of the closures of the components of $\G \setminus \{x_{0}\}$ so that $E_{\ell}(\G_{+}) = \cE$ and $E_{\ell}(\G_{-}) = E_{\ell} \setminus \cE$. More precisely, $\G_{+}$ is exactly the union of the complementary components of $x_0$ with end spaces corresponding to $\cF_1,\ldots,\cF_k$ together with adding back in $x_{0}$. Similarly, $\G_{-}$ is the union of the components corresponding to $\cF_{k+1},\ldots, \cF_{m}$, together with $x_0$.
Finally, let $T_{-}$ be the maximal tree of $\Gamma_{-}$ contained in $T$. Define for each $n \in \Z$:

\begin{align*}
    \Gamma_{n} &\defeq \begin{cases} 
    \overline{\Gamma_{-} \cup B_{n}(x_{0})} &\text{ if }n \geq 0, \\
    \left(\Gamma_{-} \setminus B_{n}(x_{0})\right) \cup T_{-} &\text{ if }n < 0, 
    \end{cases}
\end{align*}
where $B_{n}(x_{0})$ is the open metric ball of radius $n$ about $x_{0}$. See \cite[Section 7.2]{DHK2023}
 for more details and pictures of the $\G_n$'s.

Recall that a subgroup $A$ of a group $G$ is a \textbf{free factor} if there exists another subgroup $P$ such that $G = A * P$. Given a free factor $A$ of $B$, we define the \textbf{corank} of $A$ in $B$, denoted by $\cork(B,A)$, as the rank of $B/\<\!\<A\>\!\>$, the quotient of $B$ by the normal closure of $A$. For the $\G_{n}$ defined above we write $A_{n} = \pi_{1}(\G_{n},x_{0})$, the free factor determined by the subgraph $\G_{n}$. 

Denote by $\PPHE(\G)$ the group of proper homotopy equivalences on $\G$ that fix the ends of $\G$ pointwise and fix the basepoint $x_{0}$, i.e., the group of \emph{pure} proper homotopy equivalences. Any pure mapping class can be properly homotoped to fix a point, hence every pure mapping class has a representative in $\PPHE(\G)$. Note a proper homotopy equivalence on $\G$ induces an isomorphism on the level of fundamental group. Hence, with our choice of basepoint $x_0 \in \G$, for each element $f \in \PPHE(\G)$, we denote by $f_*$ the induced map on $\pi_1(\G,x_0)$.

\begin{DEF}[{\cite[Definition 7.9]{DHK2023}}]
    Given $f \in \PPHE(\Gamma)$, we say that a pair of integers, $(m,n)$, with $m>n$, is \textbf{admissible} for $f$ if 
    \begin{enumerate}
        \item
            $A_{n}$ and $f_{*}(A_{n})$ are free factors of $A_{m}$, and
        \item
            both $\cork(A_{m},A_{n})$ and $\cork(A_{m},f_{*}(A_{n}))$ are finite.
    \end{enumerate}
\end{DEF}

In \cite[Corollary 7.8]{DHK2023}, we showed that for every $f \in \PPHE(\G)$, and $n \in \Z$, there exist $m \in \Z$ such that $m>n$ and $(m,n)$ is admissible for $f$. Hence, we can define:

\begin{DEF}
For a map $f \in \PPHE(\Gamma)$ and an admissible pair $(m,n)$ for $f$, we let
\begin{align*}
    \phi_{m,n}(f) := \cork(A_{m},A_{n}) - \cork(A_{m},f_{*}(A_{n})). 
\end{align*}
Call such a $\phi_{m,n}$ a \textbf{PPHE-flux map}.
\end{DEF}

\begin{LEM}[{\cite[Lemma 7.10]{DHK2023}}] \label{LEM:ind}
    The PPHE-flux of a map $f \in \PPHE(\G)$ is well-defined over the choice of admissible pair $(m,n)$. That is, if $(m,n)$ and $(m',n')$ are two admissible pairs for the map $f \in \PPHE(\Gamma)$ then $\phi_{m,n}(f) = \phi_{m',n'}(f)$. 
\end{LEM}

Furthermore:

\begin{PROP}[{\cite[Proposition 7.11 and Lemma 7.12]{DHK2023}}] \label{prop:fluxmapsWD}The PPHE-flux maps are homomorphisms. Moreover, for any nonempty proper clopen subset $\cE$ of $E_\ell$, if $f,g \in \PPHE(\G)$ are properly homotopic, then $\phi_\cE(f) = \phi_\cE(g)$.
\end{PROP}
Hence, the PPHE-flux map factors through $\PMap(\G)$, so
we can define the flux map on $\PMap(\G)$ as follows.

\begin{DEF}\label{def:fluxmapPMap}
For each nonempty proper clopen subset $\cE$ of $E_\ell$, we define the \textbf{flux map} as:
\begin{align*}
    \Phi_\cE:\PMap(\Gamma) \rightarrow \Z \\
    [f] \mapsto \phi_\cE(f),
\end{align*}
which is a well-defined homomorphism by \Cref{prop:fluxmapsWD}.
\end{DEF}

This independence of the choice of admissible pairs further implies the independence of the choice of the basepoint $x_0$.

\begin{LEM}[Independence to choice of $x_0$] \label{LEM:indBP}
    For a nonempty proper clopen subset $\cE$ of $E_\ell$, let $x_0$ and $x_0'$ be two different points that realize the partition $E_\ell = \cE \sqcup (E_\ell \setminus \cE)$. Say $\phi_\cE$ and $\phi_{\cE}'$ are the flux maps constructed from $x_0$ and $x_0'$ respectively, with the same orientation; $E_\ell(\G_+) = E_\ell(\G_+') = \cE$. Then $\phi_{\cE} = \phi_{\cE}'$.
\end{LEM} 

\begin{proof}
    Note $x_0$ and $x_0'$ together cut $\G$ into three parts (not necessarily connected), where two of them are of infinite rank and realize $\cE$ and $E_\ell \setminus \cE$ respectively, and the middle part is of finite rank (but not necessarily compact), and we call it $M$. 
    
    Let $\{\G_n\}$ and $\{\G_n'\}$ be the chains of graphs used to define $\phi_\cE$ and $\phi_{\cE}'$ respectively. Then since $\phi_\cE$ and $\phi_{\cE'}$ are in the same direction, there exists $k \in \Z$ such that $A_{n+k} = A_{n}'$ for all $n \in \Z$. To be precise, this holds for $k$ such that $\G_{k}$ and $\G_{0}'$ have the same core graph. Now, given $f \in \PMap(\G)$ and an admissible pair $(m,n)$ for $f$ at $x_{0}$, the pair $(m-k,n-k)$ is admissible for $f$ at $x_{0}'$. Then\
    \begin{align*}
    (\phi_{\cE})_{m,n}(f) &= \cork(A_m,A_n) - \cork(A_m,f_*(A_n))\\
    &= \cork(A'_{m-k},A'_{n-k}) - \cork(A'_{m-k}, f_*(A'_{n-k}))
    =(\phi_{\cE}')_{m-k,n-k}(f).
    \end{align*}
    
    All in all, the independence of the choice of admissible pairs by \Cref{LEM:ind} proves that $\phi_{{\cE}}(f) = \phi_{{\cE}}'(f)$. Since $f$ was chosen arbitrarily, this concludes the proof.
\end{proof}

Therefore, for each nonempty proper clopen subset $\cE$ of $E_\ell$, we can write the resulting flux map as $\phi_{\cE}$ without specifying $x_0$.

We end this subsection by exploring basic properties of flux maps, to be used in subsequent subsections. Note that flux maps inherit the group operation from $\Hom(\PMap(\G),\Z)$; pointwise addition.

\begin{PROP} \label{prop:fluxProperties}
    Let $\cE \subset E_\ell$ be a nonempty proper clopen subset of $E_\ell$, where $|E_\ell| \ge 2$. Let $A,B$ and $B'$ be nonempty proper clopen subsets of $E_\ell$, such that $A$ and $B$ are disjoint, and $B$ is a proper subset of $B'$. Then the following hold:
    \begin{PROPenum}
        \item \label{prop:fluxcomplement} $\Phi_{\cE^c} = -\Phi_\cE$.
        \item \label{prop:flux_disjoint_sets} $\Phi_{A \sqcup B} = \Phi_A + \Phi_B.$
        \item \label{prop:fluxdifference} $\Phi_{B' \setminus B} = \Phi_{B'} - \Phi_{B}$.
    \end{PROPenum}
\end{PROP}

\begin{proof}
    We first note that (iii) follows from (ii), noting that $B' \setminus B$ and $B$ are disjoint. Hence, it suffices to prove (i) and (ii).
    \begin{enumerate}[label=(\roman*)]
        \item Let $f \in \PPHE(\G)$ and $\cE \subset E_\ell$ be a nonempty proper clopen subset. Choose $g \in \PPHE(\G)$ to be a proper homotopy inverse of $f$. 
    Take $\G_L$ and $\G_R$ with $\G_L \subset \G_R$ to be an admissible pair of graphs for $f$ and $g$ with respect to $\cE.$
    Fixing $\G_L$, we can enlarge $\G_R$ so that $(\G \setminus \G_L, \G \setminus \G_R)$ is an admissible pair for $f$ with respect to $\cE^c$. Note $(\G_R, \G_L)$ is still an admissible pair of graphs for $f$ with respect to $\cE$. In summary, we have:
    \begin{itemize}
        \item $f(\G_L) \subset \G_R, \quad g(\G_L) \subset \G_R$
        \item $f(\G \setminus \G_R) \subset \G \setminus \G_L$,
        \item $\cork(\pi_1(\G_R), \pi_1(\G_L))<\infty, \quad \cork(\pi_1(\G_R), f_*(\pi_1(\G_L))) < \infty$.
        \item $\cork(\pi_1(\G_R), g_*(\pi_1(\G_L))) < \infty$.
        \item $\cork(\pi_1(\G \setminus \G_L), \pi_1(\G \setminus \G_R))<\infty, \quad \cork(\pi_1(\G \setminus \G_L), f_*(\pi_1(\G \setminus \G_R))) < \infty$.
    \end{itemize}
    Because $f_*$ is a $\pi_1$-isomorphism, we have the following three different free factor decompositions of $\pi_{1}(\G)$:
    \begin{align*}
        \pi_{1}(\G) &= f_*(\pi_1(\G_R)) \ast f_*(\pi_1(\G \setminus \G_R)), \\
        \pi_{1}(\G) &= \pi_1(\G_R) \ast \pi_1(\G \setminus \G_R), \text{ and} \\
        \pi_{1}(\G) &= \pi_1(\G_L) \ast \pi_1(\G \setminus \G_L).
    \end{align*}
    We also have the free factor decompositions
    \begin{align*}
        f_{*}(\pi_{1}(\G_{R})) &= \pi_{1}(\G_{L}) \ast B, \text{ and}\\
        \pi_{1}(\G\setminus\G_{L}) &= f_{*}(\pi_1(\G\setminus\G_{R})) \ast C,
    \end{align*}
    for some free factors $B$ and $C$ of $\pi_1(\G)$. Putting together these decompositions, we get:
    \begin{align*}
    \pi_{1}(\G) &= \pi_{1}(\G_{L}) \ast B \ast f_*(\pi_1(\G \setminus \G_R)) \\
    \pi_{1}(\G) &= \pi_1(\G_L) \ast f_{*}(\pi_1(\G\setminus\G_{R})) \ast C.
    \end{align*}
    Therefore, we have $\rk(B)=\rk(C)$.
    
    Translating these equalities, we compute:
    \begin{align*}
        \Phi_{\cE^c}(f) &= \cork(\pi_1(\G \setminus \G_L), f_*(\pi_1(\G \setminus \G_R))) - \cork(\pi_1(\G \setminus \G_L), \pi_1(\G \setminus \G_R)) \\
        &= \cork(f_*(\pi_1(\G_R)), \pi_1(\G_L)) - \cork(\pi_1(\G_R),\pi_1(\G_L)) \\
        &= \cork(\pi_1(\G_R), g_*(\pi_1(\G_L))) - \cork(\pi_1(\G_R),\pi_1(\G_L)) \\
        &= \Phi_\cE(g) = -\Phi_\cE(f),
    \end{align*}
where the last equation follows from that $g$ is a proper inverse of $f$ and $\Phi_\cE$ is a homomorphism.
    \item
        Let $f \in \PPHE(\G)$. Choose an $x_{0}$ that determines a partition that is compatible with both $A^{c}$ and $B^{c}$ as in the beginning of this section. Then there exist admissible pairs $(\G_{R_{A^{c}}},\G_{L_{A^{c}}})$ and $(\G_{R_{B^{c}}}, \G_{L_{B^{c}}})$ of $f$ with respect to $A^{c}$ and $B^{c}$ respectively. By taking small enough $\G_{L_{A^{c}}}$ and $\G_{L_{B^{c}}}$, we can ensure that $\G_{R_{A^{c}}}$ and $\G_{R_{B^{c}}}$ have contractible intersection in $\G$; See \Cref{fig:flux_disjoint_sets}.
        \begin{figure}[ht!]
            \centering
            \includegraphics[width=.5\textwidth]{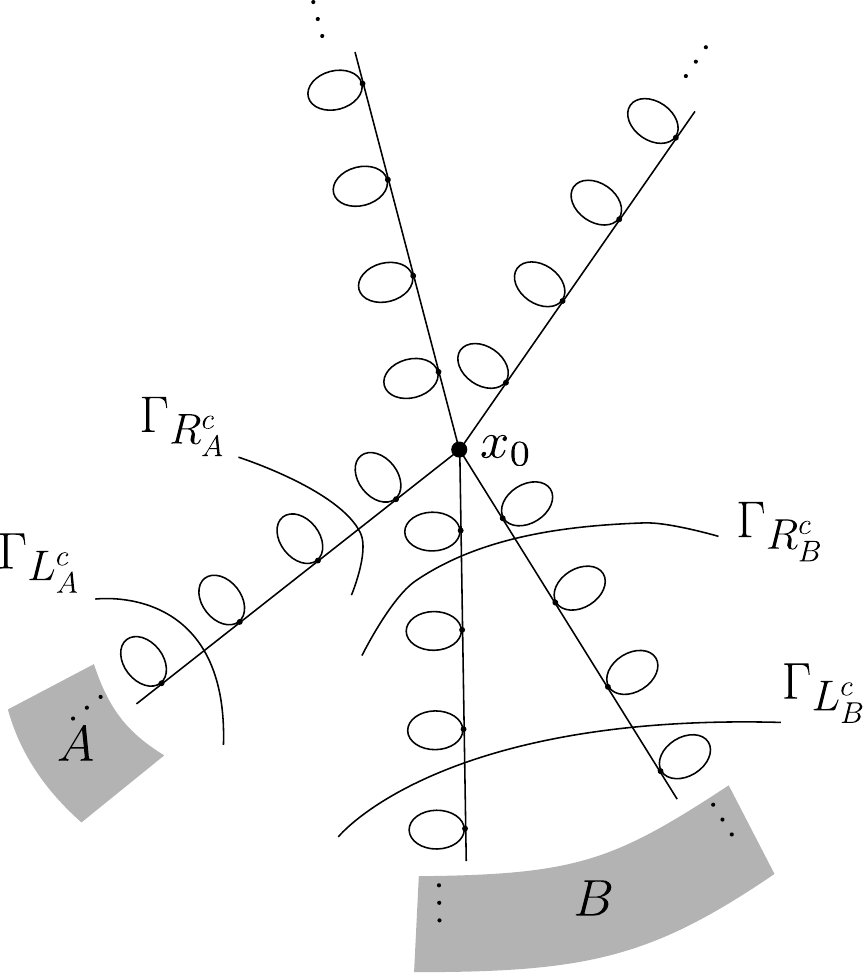}
            \caption{Illustration of the choices of subgraphs for the proof of \Cref{prop:flux_disjoint_sets}. Here the paths from $x_0$ to each subgraph are omitted. We can choose pairs of graphs $(\G_{R_A^c},\G_{L_A^c})$ and $(\G_{R_B^c},\G_{L_B^c})$ such that the graphs from different pairs have contractible intersections.}
            \label{fig:flux_disjoint_sets}
        \end{figure}
    
        Then we observe that $(\G_{R_{A^{c}}} \cup \G_{R_{B^{c}}}, \G_{L_{A^{c}}} \cup \G_{L_{B^{c}}})$ is an admissible pair for $f$ with respect to $A^{c} \cap B^{c} = (A \sqcup B)^{c}$ (still with the basepoint $x_{0}$). We then have a free decomposition
            \[
                \pi_1(\G_{R_{A^{c}}} \cup \G_{R_{B^{c}}}, x_0) \cong \pi_1(\G_{R_{A^{c}}},x_0) \ast \pi_1(\G_{R_{B^{c}}},x_0),
            \]
            and the same for $\pi_1(\G_{L_{A^{c}}} \cup \G_{L_{B^{c}}},x_0)$.
            Finally, we compute 
            \begin{align*}
                \Phi_{(A \sqcup B)^{c}} &=
                \cork\left(A_{R_{A^{c}}} \ast A_{R_{B^{c}}}, f_*(A_{L_{A^{c}}} \ast A_{L_{B^{c}}}) \right) -
                \cork\left(A_{R_{A^{c}}} \ast A_{R_{B^{c}}}, A_{L_{A^{c}}} \ast A_{L_{B^{c}}} \right)\\
                &= \left(\cork(A_{R_{A^{c}}}, f_*(A_{L_{A^{c}}})) + \cork(A_{R_{B^{c}}}, f_*(A_{L_{B^{c}}})) \right)  \\
                &\hspace{20pt}-
                \left(\cork(A_{R_{A^{c}}}, A_{L_{A^{c}}}) + \cork(A_{R_{B^{c}}}, A_{L_{B^{c}}}) \right)\\
                &= \left(\cork(A_{R_{A^{c}}}, f_*(A_{L_{A^{c}}})) - \cork(A_{R_{A^{c}}}, A_{L_{A^{c}}}) \right) \\
                &\hspace{20pt}+
                \left(\cork(A_{R_{B^{c}}}, f_*(A_{L_{B^{c}}})) - \cork(A_{R_{B^{c}}}, A_{L_{B^{c}}}) \right) \\
                &= \Phi_{A^{c}} + \Phi_{B^{c}}.
        \end{align*}
        Finally we apply Part (i) to see that 
        \[
        \Phi_{A \sqcup B} = -\Phi_{(A \sqcup B)^{c}} = -\Phi_{A^{c}} - \Phi_{B^{c}} = \Phi_{A} + \Phi_{B}. \qedhere
        \]
\end{enumerate}
\end{proof}

\begin{RMK}\label{rmk:fluxwholeempty}
    We remark that by \Cref{prop:fluxcomplement} and \Cref{prop:flux_disjoint_sets}, we can even formally define the flux map with respect to the empty set or the whole set $E_\ell$: \[
\Phi_{\emptyset}:= \Phi_A - \Phi_A \equiv 0, \qquad \Phi_{E_\ell} := \Phi_{A} + \Phi_{A^c } \equiv 0.
\] This allows us to define a flux map for any clopen $\cE\subset E$ by $\Phi_{\mathcal{E}}=\Phi_{\mathcal{E}\cap E_{\ell}}$.
\end{RMK}

\subsection{Flux zero maps} \label{ssec:semidirect}

In this section we will prove the following characterization of flux zero maps.

\begin{THM} \label{thm:fluxzeromaps}
    Let $\G$ be a locally finite, infinite graph with $|E_{\ell}(\G)| \geq 2$, and $f \in \PMap(\G)$. Then $f \in \overline{\PMapc(\G)}$ if and only if $\Phi_{\cE}(f) = 0$ for every clopen subset $\cE$ of $E(\Gamma)$.
\end{THM}

We have proved the forward direction already in a previous paper.
\begin{PROP}[{\cite[Proposition 7.13]{DHK2023}}] \label{prop:compactFluxZero}
    If $f \in \overline{\PMap_c(\G)}$, then $\Phi_\cE(f)=0$ for every clopen subset $\cE$ of $E(\G)$. 
\end{PROP}

We will first assume that $\G$ is a core graph, i.e., $E_{\ell}(\G)=E(\G)$. For brevity, we will temporarily drop the subscript $\ell$ for $E_\ell$ while we work under this assumption. To leverage the algebraic information (flux 0) to obtain topological information (homotopy equivalence), we need the following fact:

\begin{LEM}[{\cite[Proposition 1B.9]{Hatcher2002}}] \label{lem:AlgToGeom}
    Let $X$ be a connected CW complex and let $Y$ be $K(G,1)$. Then every homomorphism $\pi_1(X,x_0) \to \pi_1(Y,y_0)$ is induced by a continuous map $(X,x_0) \to (Y,y_0)$ that is unique up to homotopy fixing $x_0$.
\end{LEM}

Recall that a graph is $K(F,1)$ for $F$ a free group (the fundamental group of the graph). Now we prove a preliminary lemma to construct a compact approximation of a proper homotopy equivalence.
\begin{LEM} \label{lem:zeroflux}
    Let $\cE \subset E(\G)$ be a nonempty proper clopen subset and $f \in \PMap(\G)$. If $\Phi_{\cE}(f) = 0$, then given any compact $K \subset \G$, there exists $\psi \in \PMap(\G)$ such that
    \begin{enumerate}[label=(\roman*)]
        \item \textrm{\bf (Compact approximation)} 
            $\psi f^{-1} \in \cV_{K}$,
        \item \textrm{\bf (Truncation)} there exist disjoint subgraphs $\G_{\cE}$, and $\G_{\cE^{c}}$ of $\G$ with end spaces $\cE$ and $\cE^{c}$ respectively, such that 
            $\psi\vert_{\G_{\cE}} = \id$ and $\psi\vert_{\G_{\cE^{c}}} = f\vert_{\G_{\cE^{c}}}$, and
        \item \textrm{\bf (Same flux)}
            $\Phi_{\eta}(\psi) = \Phi_{\eta}(f)$ for every clopen subset $\eta \subset E(\G) \setminus \cE$.
    \end{enumerate}
\end{LEM}

\begin{proof}
    Let $\{\G_{n}\}_{n \in \Z}$ be as in the definition of $\Phi_{\cE}$, for some choice of basepoint $x_{0}$. Now, given $f \in \PMap(\G)$ and any $n$ there is some $m_{n}>n$ that makes $(m_n,n)$ into an admissible pair for $f$. See \Cref{fig:compactApproximation}.
    
\begin{figure}[ht!]
    \centering
    \includegraphics[width=.6\textwidth]{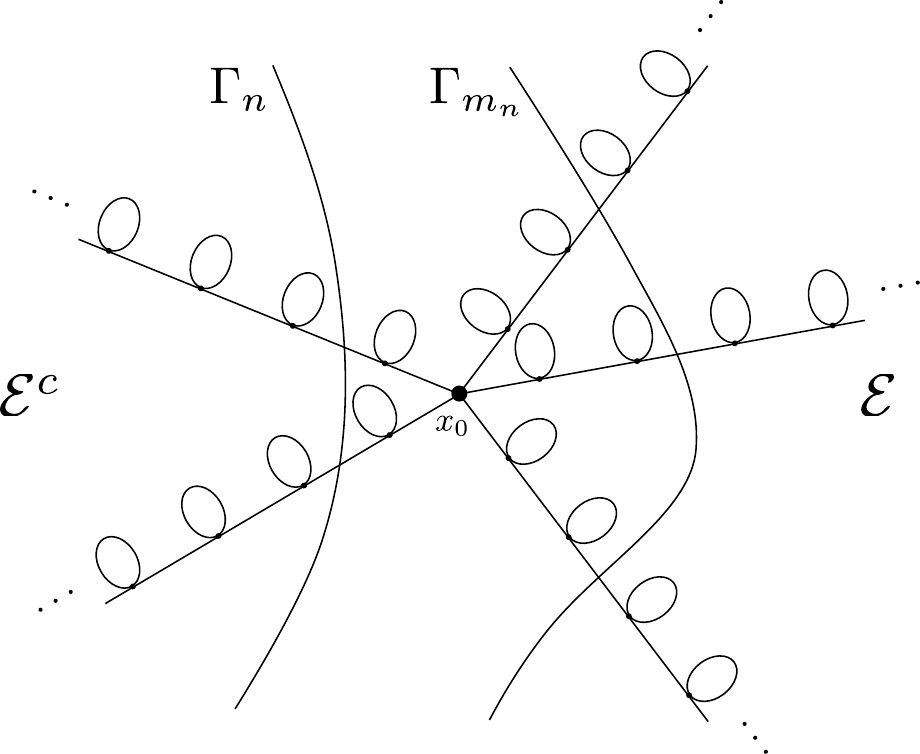}
    \caption{Illustration of $\G_n$ and $\G_{m_n}$ for the flux map $\Phi_{\cE}$. Note $x_0$ is realizing a partition of ends compatible with the partition $\cE \sqcup \cE^c$.}
    \label{fig:compactApproximation}
\end{figure}
    
    Since $\Phi_{\cE}(f) = 0$ we have
    \begin{align} \label{eq:flux0}
\cork(\pi_1(\G_{m_{n}},x_0),\pi_1(\G_{n},x_0)) = \cork(\pi_1(\G_{m_{n}},f(x_0)), f_*(\pi_1(\G_{n},x_0))) \tag{$\ast$}
    \end{align}
    for each $n \in \Z$. This allows us to define an isomorphism $\Psi_{n}: \pi_{1}(\G,x_{0}) \rightarrow \pi_{1}(\G,f(x_{0}))$  for each $n$. Here we use the notation $G \dd H$ to denote the complementary free factor of $H$ in $G$.  Define
    \begin{align*}
        \Psi_{n}=
      \begin{cases}
        \Id & \text{on $\pi_{1}(\Gamma,x_{0}) \dd\ \pi_{1}(\G_{m_{n}},x_{0})$}, \\
        \sigma_n & \text{on } \pi_{1}(\Gamma_{m_{n}},x_{0}) \dd\ \pi_{1}(\G_{n},x_{0}), \\
        f_{*} & \text{on $\pi_{1}(\G_{n},x_{0})$},
      \end{cases}
    \end{align*}
     where $\sigma_n: \pi_{1}(\G_{m_{n}},x_{0}) \dd\ \pi_{1}(\G_{n},x_{0}) \rightarrow \pi_{1}(\G_{m_{n}},f(x_{0})) \dd\ f_{*}(\pi_{1}(\G_{n},x_{0}))$ is any isomorphism. Such $\sigma_n$ is guaranteed to exist by \eqref{eq:flux0}.

    Now by 
    \Cref{lem:AlgToGeom}, for each $n$ there exists a homotopy equivalence $\psi_{n}:(\G,x_{0}) \to (\G,f(x_{0}))$ such that
    \[
      \psi_{n}=
      \begin{cases}
        \Id & \text{on $\G \setminus \G_{m_{n}}$,} \\
        f & \text{on $\G_{n}$}.
      \end{cases}
    \]    
    Also note $\psi_{n}$ is a \emph{proper} homotopy equivalence, as it can be defined in pieces as proper maps.  Further, $\psi_n$ fixes the ends of $\G$, because $f$ does and $\G_{m_n} \setminus \G_n$ is compact. One can similarly define its proper homotopy inverse. Hence, for each $n$ we have $[\psi_n] \in \PMap(\G)$.

    The subgraphs $\{\G_{n}\}_{n \in \Z}$ form an exhaustion of $\G$, so $\psi_{n} \rightarrow f$ in $\PMap(\G)$. Therefore, for a compact $K \subset \G$, there exists an $n'\in \Z$ such that $\psi_{n'}f^{-1} \in \cV_{K}$. Take $\psi = \psi_{n'}$ and set $\G_{\cE^{c}} = \G_{n'}$ and $\G_{\cE} = \overline{\G \setminus \G_{m_{n'}}}$. This gives (i) and (ii) by construction.
    
    We now check that (iii) follows from (ii). Let $\eta$ be a clopen subset of $E(\G)$ that is disjoint from $\cE$. We will actually check that $\Phi_{\eta^{c}}(\psi) = \Phi_{\eta^{c}}(f)$. This will imply (iii) by \Cref{prop:fluxcomplement}.

    Note $\eta \subset \cE^c$. Now let $\G_{m}$ be a subgraph from the definition of $\Phi_{\eta^{c}}$ so that $\G_{m} \subset \G_{\cE^{c}}$.
    Then there exists $n \le m$ such that $(m,n)$ is admissible for $\psi$ with respect to the flux map $\Phi_{\eta^c}$. Since $f=\psi$ on $\G_n \subset \G_m \subset \G_{\cE^{c}}$ by (ii), we see that $\Phi_{\eta^c}(\psi) = \Phi_{\eta^c}(f)$ with the admissible pair of graphs $(\G_m,\G_n)$.
\end{proof}

\begin{RMK}
    The reader may wonder why in the proof above we chose to define this sequence of maps and argue via convergence in place of constructing the map $\psi$ by hand as in \cite{APV2020}. While it is not too difficult to construct a $\psi$ so that $\psi f^{-1}$ is the identity on a given compact $K$, it is significantly more finicky to guarantee that $\psi f^{-1}$ preserves the complementary components of $K$. The convergence argument given above allows us to avoid the messy details of this.
\end{RMK}

\begin{PROP}\label{prop:coregraphcase}
    Let $\G$ be a locally finite, infinite graph with $E(\G) = E_{\ell}(\G)$, $|E(\G)| \geq 2$, and $f \in \PMap(\G)$. If $\Phi_{{\cE}}(f) = 0$ for every clopen subset $\cE$ of $E(\G)$, then $f \in \overline{\PMap_c(\G)}$.
\end{PROP}

\begin{proof}
    Assume $f \in \PMap(\G)$ has $\Phi_{\cE}(f) = 0$ for every nonempty proper clopen subset $\cE$ of the end space $E(\G)$. Given any compact $K \subset \G$ we will find $\psi \in \PMapc(\G)$ such that $\psi f^{-1} \in \cV_{K}$. 

    Without loss of generality we may enlarge $K$ so that it is connected, has at least two complementary components, and every complementary component of $K$ is infinite. Then the complement of $K$ induces a partition of the ends. Write 
    \[
    \cP_K = \cE_1 \sqcup \ldots \sqcup \cE_n
    \]
    for this partition.

    Apply \Cref{lem:zeroflux} to $f$ using $\cE_{1}$ to obtain $\psi_{1}$. Note that by (iii) we still have $\Phi_{\cE_{2}}(\psi_{1}) = \Phi_{\cE_2}(f) = 0$. Thus we can apply the lemma again to $\psi_{1}$ using $\cE_{2}$ to obtain a $\psi_{2}$. Continue this process recursively to obtain $\psi_{n}$. 

    Now, by (i) of \Cref{lem:zeroflux}, there exist $v_1,\ldots,v_n \in \cV_K$ such that
    \begin{align*}
        \psi_{i} &=
        \begin{cases}
            v_i\psi_{i-1} & \text{for $1 < i \le n$,}\\
            v_1f & \text{for $i=1$}.
        \end{cases}
    \end{align*}
    Putting these together gives $\psi_{n}f^{-1} = v_{n}v_{n-1}\cdots v_{1} \in \cV_{K}$ as $\cV_{K}$ is a subgroup.

    It remains to check that $\psi_{n} \in \PMapc(\G)$. However, by (ii), we have that $\psi_{n}$ is equal to the identity on $\bigcup_{i=1}^{n}\G_{\cE_{i}}$. This exactly covers all of the ends of $\G$ as $\cP_{K}$ was a partition of the ends. Therefore we see that $\psi_{n}$ is supported on $\overline{\bigcap_{i=1}^{n} \G \setminus \G_{\cE_{i}}}$, a compact set. Taking $\psi = \psi_{n}$ gives the desired compact approximation of $f$. 

    Finally, since the $K$ above was taken to be arbitrary, starting with a compact exhaustion of $\G$ we can apply the above to obtain a sequence of compactly supported maps that converge to $f$. 
\end{proof}

Now we turn to the case where $\G$ is not necessarily a core graph. 

\begin{proof}[Proof of \Cref{thm:fluxzeromaps}]
    The forward direction follows from \Cref{prop:compactFluxZero}.

    For the backward direction, we first homotope $f$ so that it fixes the vertices of $\G$. Then we see that we can write $f=f_{T}f_{c}$ where $f_{T}$ has support on $\G \setminus \G_{c}$ and $f_{c}$ has support on $\G_{c}$. 

    We can see that $f_{T} \in \overline{\PMap_{c}(\G)}$. Indeed, enumerate the components of $\G\setminus \G_{c}$ as $\{R_{i}\}_{i\in I}$ where each $R_{i}$ is a tree and $I$ is either finite or $I=\N$. Then we can decompose $f_{T} = \prod_{i\in I} f_{i}$ where each $f_{i}$ has compact support on $R_{i}$. Indeed, each has compact support as $f_{T}$ is proper and thus the pre-image of the cutpoint $\overline{R_{i}} \cap \G_{c}$ is compact and $f_{i}$ can be homotoped to have support contained within the convex hull of this full pre-image of the cutpoint. Furthermore, all of the $f_{i}$ pairwise commute as each $f_{i}$ can be homotoped so that it is totally supported away from the support of each other $f_{j}$. Thus, we see that $f_{T} \in \overline{\PMapc(\G)}$ as it is realized as the limit of partial products of the $f_{i}$. 

    This also shows that given any flux map $\Phi_{\cE}$ we must have that $\Phi_{\cE}(f_{T}) = 0$, again by \Cref{prop:compactFluxZero}. Therefore, given an $\cE$ with $\Phi_{\cE}(f) = 0$ we must have that $\Phi_{\cE}(f_{c})=0$ as $\Phi_\cE$ is a homomorphism. We can then apply \Cref{prop:coregraphcase} to conclude the desired result. 
\end{proof}

\subsection{Space of flux maps}\label{ssec:SpaceOfFluxmaps}

Before we can prove \Cref{thm:semidirectprod} we need to endow the set of flux maps with an algebraic structure. In the surface case, \cite{APV2020} could utilize the first integral (co)homology of separating curves on the surface to give structure to the flux maps they defined. Here we will be using the group of locally constant $\Z$-valued functions on $E_{\ell}(\G)$ in place of the homology of separating curves. We remark that this is really the zeroth \cech{} cohomology of $E_{\ell}(\G)$ with coefficients in the constant sheaf $\Z$. In \Cref{ssec:surfacescohom} we observe that this perspective also works in the surface case. 

For a topological space $X$, we denote by $\ceC(X)$ the group of locally constant $\Z$-valued functions on $X$. The group operation is given by addition of functions. We let $\dC(X) = \ceC(X)/\Z$, the quotient obtained by identifying the constant functions with zero. We will now give a collection of some facts about $\dC(E)$ when $E$ is a compact, totally disconnected, and metrizable space (i.e. a closed subset of a Cantor set).

We identify the Cantor set, $\cC = 2^{\N} = \{0,1\}^{\N}$, with the set of countable binary sequences. A countable basis of clopen sets for the topology is then given by the cylinder sets
    \begin{align*}
        C_{a_{1}\cdots a_{k}} \defeq \{ (x_{n}) \in 2^{\N}\ \vert\ x_{i} = a_{i}, \; i=1,\ldots, k\}
    \end{align*}
where $a_{1}\cdots a_{k}$ is some finite binary sequence of length $k$. Say such a cylinder set has \textbf{width} $k$. For $E$ a closed subset of the Cantor set $\cC$, a \textbf{cylinder set} of $E$ is the intersection of a cylinder set for $\cC$ with $E$, i.e., a set of the form $C_{w} \cap E$ where $w \in 2^{k}$ for some $k \ge 0$. The standard tree model for the Cantor set is the usual rooted binary tree, and for an arbitrary closed subset $E \subset \cC$ we take the subtree with the end space $E$.
Given a subset, $A$, of a topological space we let $\chi_{A}$ denote the indicator function on $A$. 

\begin{THM}[Countable Basis for $\dC(E)$]\label{thm:chombasis}
\setlength{\itemsep}{5pt}
    Let $E$ be a compact, totally disconnected, and metrizable space. There exists a countable collection $\cA=\{A_i\}_{i \in I}$ of cylinder sets of $E$ so that 
    \begin{enumerate}
        \item Any cylinder set $C$ of $E$ that is not in $\cA$ can be written as $C = A_0 \setminus (A_{1} \sqcup \cdots \sqcup A_{n})$ for some $A_0\in \cA$, and some $A_{j} \in \cA$, with $A_j\subset A_0$ and $A_{j} \cap A_{k} = \emptyset$ for all distinct $j,k \in \{1,\ldots,n\}$,
        \item $\mathcal{B} = \{\chi_{A_{i}}\}_{i\in I}$ is a free basis for $\dC(E)$. In particular, $\dC(E) = \oplus_{i\in I} \Z$, and
        \item for $T$ the standard tree model of the end space $(E,\emptyset)$, there exists an injective map $\iota:\cA \rightarrow T$ so that $\iota$ maps into the interior of edges and $\iota(\cA)$ cuts the graph into a collection of one ended graphs. 
    \end{enumerate}
    
\end{THM}

\begin{proof}
    Note that if $|E| = n <\infty$ then the result is immediate by taking $\cA$ to be the collection of all individual ends except one. Hence, we will assume that $E$ is infinite. 

    We first prove the result for $E = \cC$ the Cantor set. We define $\cA'$ to be the set of all cylinder sets consisting of cylinders of the form $C_{a_{1}\cdots a_{k-1} 0}$ together with the whole space $\cC$. That is, 
    \begin{align*}
        \cA' = \{\cC,C_{0},C_{00},C_{10},C_{000},C_{100},C_{010},C_{110},\ldots \}
    \end{align*}
    We claim that $\{\chi_{A}\}_{A \in \cA'}$ forms a free basis for $\ceC(\cC)$. We first have

    \begin{CLAIME}
        For every $f \in \ceC(\cC)$, there exist finitely many disjoint clopen subsets $B_1,\ldots,B_n$ and integers $b_{1},\ldots, b_{n}$ such that
    \[
        f = \sum_{j=1}^n b_{j}\chi_{B_j}.
    \]
    \end{CLAIME}
    
\begin{proof}
        \renewcommand{\qedsymbol}{$\triangle$}
    Suppose $f$ is a locally constant function on $\cC$ with \emph{infinitely many} distinct $\Z$-values $b_1,b_2,\ldots$. Then $\{f^{-1}(b_j)\}_{j=1}^\infty$ forms a clopen cover of $\cC$ which does not have a finite subcover, contradicting the compactness of $\cC$. Therefore, $f$ can assume at most finitely different values in $\Z$, and taking $B_{j} = f^{-1}(b_{j})$ proves the claim.
\end{proof}

    Thus we can check that $\{\chi_{A}\}_{A \in \cA'}$ generates $\ceC(\cC)$ by verifying that for an arbitrary clopen set $B$ of $\mathcal{C}$, we can write $\chi_{B}$ as a finite linear combination of elements from $\{\chi_{A}\}_{A \in \cA'}$. Since the cylinder sets form a clopen basis for the topology, we only need to check when $B$ is a cylinder set. Take $B = C_{a_1\cdots a_k}$ for some $k > 0$ and $a_{1}\cdots a_{k} \in 2^{k}$. Then we have either $B \in \cA'$ or $a_{k} = 1$. Supposing the latter, let
    \[
        m = \begin{cases}
            0 & \text{if $a_1=\ldots=a_k=1$,}\\
            \max \{j\vert a_{j}=0\} & \text{otherwise}
        \end{cases}
    \]
    Then we can write 
    \begin{align*}
        \chi_B = \chi_{C_{a_{1}\cdots a_{k}}} = \chi_{C_{a_{1}\cdots a_{m}}} - \left(\sum_{j=m}^{k-1} \chi_{C_{a_{1}\cdots a_{j}0}} \right),
    \end{align*} 
    where we take $a_1\cdots a_m$ as an empty sequence when $m=0$.
    Thus we see that $\{\chi_{A}\}_{A \in \cA'}$ generates $\ceC(\cC)$. This also shows that property (1) holds.

    Next we verify that the set $\mathcal{B}':=\{\chi_{A}\}_{A \in \mathcal{A}'}$ is linearly independent. Suppose
    \begin{align*}
        0=\sum_{j=1}^{n} a_{j} \chi_{A_{j}},
    \end{align*}
    for some distinct $A_1,\ldots,A_n \in \cA'$.
    We will proceed by induction on $n$. The case when $n=1$ is straightforward. Now let $n>1$ and without loss of generality we can assume that $A_{n}$ is of minimal width. Let $w$ be the word defining $A_{n}$, i.e. $A_{n} = C_{w}$. Note that $w$ may be the empty word (when $A_n = \cC)$. Consider the sequence $w\bar{1}$ consisting of the starting word $w$ followed by the constant infinite sequence of $1$s. Then by minimality of $w$, we have
    \begin{align*}
        0 = \sum_{j=1}^{n} a_{j} \chi_{A_{j}}(w\bar{1}) = a_{n}.
    \end{align*}
    Therefore, we have 
        $0=\sum_{j=1}^{n} a_{j} \chi_{A_{j}} = \sum_{j=1}^{n-1} a_{j} \chi_{A_{j}}$
    so by the induction on $n$ we see that $a_{j} = 0$ for all $j$. Thus we see that $\mathcal{B}'$ is a free basis for $\ceC(\cC)$. Taking $\mathcal{A}:= \mathcal{A}' \setminus \{\mathcal{C} \} = \{C_{0},C_{00},C_{10},C_{000},C_{100},C_{010},C_{110},\ldots\}$, the free basis $\mathcal{B}'$ for $\ceC(\cC)$ descends (allowing for a slight abuse of notation) to a free basis $\mathcal{B}:=\{\chi_{A}\}_{A \in \cA}$ for $\dC(\cC)$, proving (2).

    Finally, we can define $\iota: \cA \rightarrow T$ by using the labels on each of the cylinder sets to map each cylinder set to the midpoint of its corresponding edge in the standard binary tree model of the Cantor set. See \Cref{fig:cantorhom} for a picture of the map. The components of $T\setminus \iota(\cA)$ each contains one end from $T$.

    \begin{figure}[ht!]
	    \centering
	    \def\svgwidth{.6\textwidth}
		    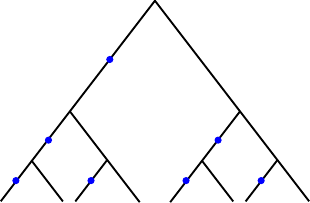
		    \caption{The image of the map $\iota: \mathcal{A} \to T$ is given in blue.}
	    \label{fig:cantorhom}
\end{figure}

    Now to go from the Cantor set to a general infinite end space we identify $E$ with a subspace of $\cC$ and take $\cA = \{C_{0} \cap E, C_{00} \cap E,C_{10} \cap E,\ldots \}$, deleting duplicated sets if necessary.
    Then the set $\{\chi_A\}_{A \in \cA}$ will still determine a free basis for $\dC(E)$. 
\end{proof}

Apply this theorem to $E_{\ell}(\G)$ in order to obtain the set $\cA=\{A_i\}_{i \in I}$. We now define the homomorphism 
\begin{align*}
    \Pi: \PMap(\G) &\rightarrow \prod_{i \in I} \Z \\
    f &\mapsto (\Phi_{A_{i}}(f))_{i \in I}.
\end{align*}
We will check that this map is surjective and has kernel exactly $\overline{\PMapc(\G)}$, i.e. it forms the following short exact sequence:
\begin{align*}
    1 \longrightarrow \overline{\PMapc(\G)} \longrightarrow \PMap(\G) \overset{\Pi}{\longrightarrow} \prod_{i \in I} \Z \longrightarrow 1.
\end{align*}

\begin{LEM}\label{lem:zerofluxbasis}
    Let $\cE$ be a clopen subset of $E(\G)$ so that $\cE \cap E_{\ell}(\G)$ is a proper nontrivial subset. If $f \in \PMap(\G)$ satisfies $\Phi_{A} (f) = 0$ for all $A \in \cA$, then $\Phi_{\cE}(f) = 0$ as well. 
\end{LEM}

\begin{proof}
    We first note that $\cE$ can be written as a disjoint union of finitely many cylinder sets. Thus, by \Cref{prop:flux_disjoint_sets} it suffices to check when $\cE$ is a cylinder set $C$ of $E(\G)$.
    Assume that $f \in \PMap(\G)$ satisfies $\Phi_{A_{i}} (f) = 0$ for all $i \in I$. Then $C\cap E_{\ell}(\G)$ is again a cylinder set of $E_{\ell}$. Applying property (1) of \Cref{thm:chombasis} we have either $C \in \cA$, or $C = A_{0} \setminus (\bigsqcup_{j=1}^{n} A_{j})$ for some $A_{0} \in \cA$ and $A_j \in \cA$. If $C \in \cA$, then we conclude $\Phi_C(f)=0$. For the other case, we can apply \Cref{prop:flux_disjoint_sets} and \Cref{prop:fluxdifference} to write
    \begin{align*}
        \Phi_{C} (f) = \Phi_{A_{0}}(f) - \sum_{j=1}^{n} \Phi_{A_{j}}(f) = 0 - 0 =0. 
    \end{align*}
\end{proof}

\begin{COR} \label{cor:kernelpi}
    For $\G$ and $\Pi$ as above, $\ker(\Pi) = \overline{\PMapc(\G)}$. 
\end{COR}

\begin{proof}
The forward direction of \Cref{thm:fluxzeromaps} implies $\ker(\Pi) \supset \overline{\PMapc(\G)}$. On the other hand, \Cref{lem:zerofluxbasis} together with the backward direction of \Cref{thm:fluxzeromaps} imply the other containment $\ker(\Pi) \subset \overline{\PMapc(\G)}$.
\end{proof}

Next, we will build a section to show $\Pi$ is surjective, and more importantly, this sequence splits. This gives us our desired semidirect product decomposition in \Cref{thm:semidirectprod}.

\begin{PROP} \label{prop:sectionconstruct}
    There exists an injective homomorphism $\hat{\iota}:\prod_{i\in I} \Z \rightarrow \PMap(\G)$ so that $\Pi \circ \hat{\iota}$ is the identity on $\prod_{i \in I}\Z$.
\end{PROP}

\begin{proof}
    Let $T$ be the maximal tree of the graph $\G_{c}$ in standard form. Note that the end space of $T$ is homeomorphic to $E_{\ell}(\G)$ and let $\cA=\{A_i\}_{i \in I}$ be the set obtained from (2) of \Cref{thm:chombasis} applied to the set $E_{\ell}(\G)$ and $\iota:\cA \rightarrow T$ be the map given by property (3) of \Cref{thm:chombasis}. The closure in $\G_{c}$ of every complementary component of $\iota(\cA)$ is a one-ended subgraph with infinite rank. Call one such component $\G'$. It has at most a countably infinite number of half edges coming from the points of $\iota(\cA)$. Now we will modify $\G'$ via a proper homotopy equivalence that fixes $\partial\G'$ so that the new graph has a ``grid of loops" above $\partial\G'$. See \Cref{fig:oneendblowup} for how this replacement is done. Such a replacement by a proper homotopy equivalence is possible by the classification of infinite graphs. 
    \begin{figure}[ht!]
	    \centering
	    \def\svgwidth{\textwidth}
		    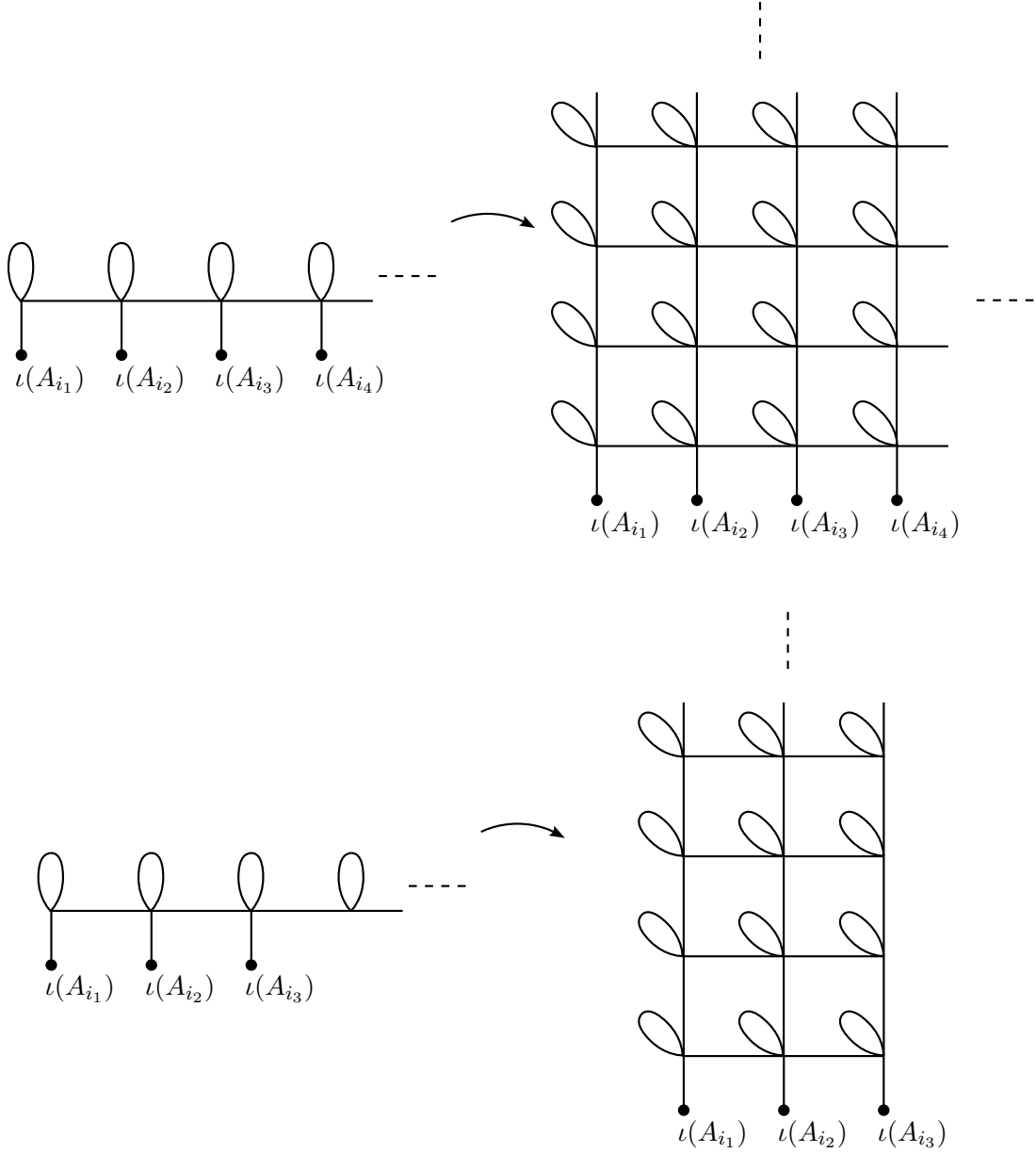
		    \caption{The new replacement graphs for each component of $T \setminus \iota(\cA)$. The top picture shows the case when a component has infinitely many cut points and the bottom for finitely many. Note that above each cut point one sees a ``column of loops" within the grid.}
	    \label{fig:oneendblowup}
    \end{figure}

    After replacing each component of $\G_{c} \setminus\iota(\cA)$ we obtain a new graph that is proper homotopy equivalent to the original $\G_{c}$. We can also extend this proper homotopy equivalence to the entire graph $\G$, as our proper homotopy equivalence fixes the boundary points of each complementary component of $\iota(\cA)$. Now for each $i$, there are exactly two complementary components whose closures in $\G_c$ contain $\iota(A_{i})$. Let $\ell_{i} \in \PMap(\G)$ be the loop shift supported on the two columns of loops sitting above $\iota(A_{i})$ in these components. Orient the loop shift so that it is shifting towards the end in $A_{i}$.

    Note that each $\ell_{i}$ has total support disjoint from each other $\ell_{j}$ so that $\ell_{i}\ell_{j} = \ell_{j}\ell_{i}$ for all $i,j \in I$. Therefore, $\prod_{i \in I} \la \ell_{i} \ra < \PMap(\G)$, and we can define the homomorphism $\hat{\iota}: \prod_{i \in I}\mathbb{Z} \to \PMap(\G)$ by 
    \begin{align*}
        \hat{\iota}\left((n_{i})_{i\in I}\right) := \prod_{i\in I} \ell_{i}^{n_{i}}.
    \end{align*}

    It remains to check that $\Pi \circ \hat{\iota}$ is the identity on $\prod_{i\in I}\Z$. By the construction of the loop shifts, $\ell_{i}$ crosses exactly one of the clopen subsets in $\cA$, namely $A_{i}$. Therefore, we have
    \begin{align*}
        \Phi_{A_{j}} (\ell_{i}) = \delta_{ij}:=\begin{cases} 1 \;\; \text{ if } i=j, \\ 0 \;\; \text{ if } i\neq j.\end{cases}
    \end{align*}
    Now, given any tuple $(n_{i})_{i \in I} \in \prod_{i\in I}\Z$ we compute

        \begin{align*}
        (\Pi \circ \hat{\iota})\left((n_{i})_{i\in I}\right) &= \Pi\left(\prod_{i\in I} \ell_{i}^{n_{i}}\right) = \left(\Phi_{A_{j}}\left(\prod_{i\in I} \ell_{i}^{n_{i}}\right)\right)_{j\in I} = (n_{i})_{i\in I}. \qedhere
    \end{align*}  
\end{proof}

\begin{proof}[Proof of \Cref{thm:semidirectprod}]
    \Cref{cor:kernelpi} and \Cref{prop:sectionconstruct} above give the desired splitting short exact sequence $1 \longrightarrow \overline{\PMap_c(\G)} \longrightarrow \PMap(\G) \longrightarrow \Z^\alpha \longrightarrow 1$, with $\alpha = |\mathcal{A}|-1$.
\end{proof}

\subsection{The rank of integral cohomology}
\label{ssec:firstcohomology}

As pointed out in \Cref{rmk:fluxwholeempty}, we define $\Phi_{\emptyset} = \Phi_{E_\ell} \equiv 0$.

\begin{LEM}\label{lem:cech2FirstCohomology}
    Let $\{A_i\}_{i \in I}$ be a clopen collection of subsets of $E_\ell(\G)$ such that $\mathcal{B} = \{\chi_{A_i}\}_{i \in I}$ is a free basis for $\dC(E_\ell)$ as in \Cref{thm:chombasis}.
    Then the map
    \begin{align*}
        \Theta: \dC(E_\ell(\G)) &\longrightarrow H^1(\PMap(\G);\Z), \\
        \sum_{i \in I}{n_i}\chi_{A_i} &\longmapsto \sum_{i \in I}{n_i}\Phi_{A_i}.
    \end{align*}
    is a well-defined injective homomorphism.
\end{LEM}

\begin{proof}
    Since $\mathcal{B}$ is a free basis for $\dC(E_\ell)$, the map $\chi_{A_i} \mapsto \Phi_{A_i}$ on $\mathcal{B}$ extends to a well-defined homomorphism on the whole group $\dC(E_\ell)$. To see $\Theta$ is injective, suppose $\Theta(\sum_i
    n_i\chi_{A_i}) = \sum_i n_i\Phi_{A_i} =0$ for $\chi_{A_i} \in
    \mathcal{B}$. 
    Then for each $j$ that arises as an index of the summation, we evaluate the sum at the loop shift $\ell_j$ constructed in the proof of \Cref{prop:sectionconstruct}:
    \[
        0 = \sum_i n_i \Phi_{A_i}(\ell_j) = n_j\Phi_{A_j}(\ell_j) = n_j,
    \]
    which implies that $\sum_i n_i\chi_{A_i} \equiv 0$, concluding that $\Theta$ is injective.
\end{proof}

Here we collect relevant results on the first homology of the pure mapping class group of graphs of rank $n$ with $s$ rays.

\begin{FACT}[{\cite[Theorem 1.1]{HV1998}}]
    \label{fact:TrivialRationalCohomology}
    $H_1(\Aut(F_n);\mathbb{Q})=0$ for all $n \ge 1$.
\end{FACT}

\begin{FACT}[{\cite[Section 4]{Hatcher1995}}]
    \label{fact:ThornedRoseCohomology}
    For $n \ge 3$ and $s \ge 1$,
    \[
    H_1(F_n^{s-1} \rtimes \Aut(F_n);\Z) \cong  H_1(F_n^{s} \rtimes \Aut(F_n);\Z).
    \]
    This still holds for $n=1,2$ if $s \ge 2$.
\end{FACT}

\begin{PROP}
  \label{prop:h1PmapcZero}
  $H^1(\PMap_c(\G);\Z)=0$ for every locally finite, infinite graph $\G$.
  \end{PROP}

  \begin{proof}
    Let $\{\G_k\}$ be a compact exhaustion of $\G$.
    Then $\PMap_c(\G)$ is a direct limit of $\PMap(\G_k)$'s, each of which is isomorphic to $F_{n_k}^{e_k} \rtimes \Aut(F_{n_k})$ for some $e_k \ge 0$ and $n_k \ge 1$ (Recall \Cref{rmk:compactMCG}). Since the direct limit commutes with $H^1(-;\Z) \equiv \Hom(-,\Z)$, it suffices to show that groups of the form $F_n^e \rtimes \Aut(F_n)$ have trivial first cohomology.
    We first show $H^1(\Aut(F_n);\Z)=0$. By the universal coefficient theorem for cohomology,
    \[
        0 \longrightarrow \Ext\left(H_0(\Aut(F_n);\bbQ),\Z\right) \longrightarrow
        H^1(\Aut(F_n);\Z)
        \longrightarrow
        \Hom(H_1(\Aut(F_n);\bbQ),\Z)
        \longrightarrow
        0
    \]
    where $\Ext\left(H_0(\Aut(F_n);\bbQ);\Z\right)=0$ as $H_0(\Aut(F_n);\bbQ)\cong \bbQ$ is free. Also, $H_1(\Aut(F_n);\bbQ)=0$ by \Cref{fact:TrivialRationalCohomology}, so it follows that $H^1(\Aut(F_n);\Z)=0$.
    
    On the other hand, repeatedly applying \Cref{fact:ThornedRoseCohomology} together with the universal coefficient theorem for homology shows that for $n \ge 3$,
    \[
    H_1(F_n^{s} \rtimes \Aut(F_n);\bbQ) = H_1(F_n^{s-1} \rtimes \Aut(F_n);\bbQ) = \ldots = H_1(\Aut(F_n);\bbQ)=0.
    \]
    The last equality comes from \Cref{fact:TrivialRationalCohomology}.
    For $n=1,2$, the argument is the same, except we reduce the problem of showing $H^1(F_n^{s-1} \rtimes \Aut(F_n);\Z)=0$ to checking $H_1(F_n \rtimes \Aut(F_n);\bbQ)=0$. One can check $\Z \rtimes \Z_2$ and $F_2 \rtimes \Aut(F_2)$ have finite abelianization to conclude this. (See e.g. \cite[Corollary 2]{AFV2008presentation} for a finite presentation of $\Aut(F_2)$.) This completes the proof of $H^1(\PMap_c(\G); \Z)=0$.
  \end{proof}

\begin{THM} \label{thm:thetaSurjective}
    The map $\Theta$ in \Cref{lem:cech2FirstCohomology} is an isomorphism.
\end{THM}

\begin{proof}
  We only need to check the surjectivity of $\Theta$. Pick $\phi \in H^1(\PMap(\G);\Z) = \Hom(\PMap(\G),\Z)$.  
  By \Cref{prop:h1PmapcZero}, we have $\phi(\PMapc(\G)) = \{0\}$. By Dudley's 
  automatic continuity \cite{dudley1961}, $\phi$ is continuous, so
  $\phi(\overline{\PMapc(\G)})=\{0\}$.
  Recall the semidirect product decomposition $\PMap(\G) \cong
  \overline{\PMap_c(\G)} \rtimes L$ from
  \Cref{thm:semidirectprod},
  where $L \cong \prod_{i \in I} \la \ell_{i} \ra$, the product of commuting loop shifts. Furthermore, these loop shifts are dual to the collection of $\{\Phi_{A_{i}}\}_{i\in I} \subset H^{1}(\PMap(\G);\Z)$ so that $\Phi_{A_j}(\ell_i) = \delta_{ij}$.
  Since $\phi$ is zero on the $\overline{\PMap_c(\G)}$-factor, it follows
  that $\phi$ is completely determined by its value on $L$. Note also that $L \cong \prod_{i \in I}\Z$ so that $H^{1}(L;\Z) \cong \oplus_{i \in I}\Z$ where a basis for $H^{1}(L;\Z)$ is given exactly by the set $\{\Phi_{A_{i}}\}_{i\in I}$, as in \Cref{thm:chombasis}(2). Hence,
  $\phi = \phi|_{L} \in H^1(L; \Z)$ can be described
  by a finite linear combination of $\Phi_{A_i}$'s. Such a finite linear combination is the image of a finite linear combination of $\chi_{A_i}$ under $\Theta$, so
  $\Theta$ is surjective.
\end{proof}

\begin{COR}[\Cref{cor:rankoffirstcohomology}, revisited]\label{cor:rankoffirstcohomology_revisited}
For every locally finite, infinite graph $\G$,
    \[
    \rk \left(H^1(\PMap(\G);\Z)\right) =
        \begin{cases}
            0 & \text{if $|E_\ell| \le 1$} \\
            n-1 & \text{if $2 \le |E_\ell|=n< \infty$} \\
            \aleph_0 & \text{otherwise}.
        \end{cases}
    \]
\end{COR}

\begin{proof}
    This follows from the isomorphism $\Theta: \dC(E_\ell(\G)) \cong H^1(\PMap(\G);\Z)$ in \Cref{thm:thetaSurjective}.
\end{proof}

\subsection{Relation to surfaces} \label{ssec:surfacescohom}

Aramayona--Patel--Vlamis in \cite{APV2020} obtain a result similar to \Cref{thm:thetaSurjective} in the infinite-type \emph{surface} case using the homology of separating curves in place of $\dC(E_{\ell}(\G))$. Here we show that these approaches can be unified, as they each rely solely on the subspace of ends accumulated by loops or genus. Let $S$ be an infinite-type surface and let $\hat{S}$ be the surface obtained from $S$ by forgetting the planar ends of $S$. Let $H_{1}^{sep}(\hat{S};\Z)$ be the subgroup of $H_{1}(\hat{S};\Z)$ generated by homology classes that have separating simple closed curves of $\hat{S}$ as representatives. Note that when $S$ has only planar ends, $H_{1}^{sep}(\hat{S};\Z)$ is trivial. 

\begin{THM}[{\cite[Theorem 4]{APV2020} for genus $\geq 2$, \cite[Theorem 1.1]{DP2020} for genus $1$}]\label{thm:apvcohom}
    Let $S$ be an infinite-type surface of genus at least one. Then 
    \begin{align*}
        H^{1}(\PMap(S);\Z) \cong H_{1}^{sep}(\hat{S};\Z).
    \end{align*}
\end{THM}

Let $E_{g}(S)$ denote the space of ends of $S$ accumulated by genus (i.e., the non-planar ends). 

\begin{PROP}\label{prop:sephomend}
    Let $S$ be an infinite-type surface. Then 
    \begin{align*}
        H_{1}^{sep}(\hat{S};\Z) \cong \dC(E_{g}(S)). 
    \end{align*}
\end{PROP}

\begin{proof}
    We first note that by definition, $E_{g}(S) = E(\hat{S})$. Let $v \in H_{1}^{sep}(\hat{S};\Z)$ be a primitive element, i.e. $v$ has a representative $\gamma$ that is an oriented and separating simple closed curve. Now $v$ determines a partition of $E(\hat{S})$ into two clopen subsets, $v^{+}$, those ends to the right of $\gamma$, and $v^{-}$, those ends to the left of $\gamma$. Note that these are proper subsets if and only if $v \neq 0$ if and only if $\chi_{v^+} \neq 0$ in $\dC(E)$. Define 
    \begin{align*}
        \Xi(v) \defeq \chi_{v^{+}} \in \dC(E),
    \end{align*}
    for each nonzero primitive element $v$ of $H_1^{sep}(\hat{S};\Z)$.
    This linearly extends to define an isomorphism $\Xi: H_{1}^{sep}(\hat{S};\Z) \xrightarrow{\sim} \dC(E_{g}(S))$.
\end{proof}

\begin{COR} \label{cor:graphtosurf}
    Let $S$ be an infinite-type surface of genus at least one and $\G$ be a locally finite, infinite graph. If $E_{g}(S)$ is homeomorphic to $E_{\ell}(\G)$, then there is a natural
    isomorphism between $H^{1}(\PMap(S);\Z)$ and $H^{1}(\PMap(\G);\Z)$.
\end{COR}

\begin{proof}
    We first note that if $E_{g}(S)$ is empty (i.e. $S$ has finite genus), then $H^{1}(\PMap(S);\Z)$ is trivial by \cite[Theorem 1]{APV2020} and \cite[Theorem 1.1]{DP2020}. Similarly, if $E_{\ell}(S)$ is empty, then $H^{1}(\PMap(S);\Z)$ is trivial by \Cref{prop:noloopends} and \Cref{prop:h1PmapcZero}.

    Otherwise, the isomorphism is obtained by composing the maps from \Cref{thm:thetaSurjective}, \Cref{thm:apvcohom}, and \Cref{prop:sephomend}:
    \[
    H^1(\PMap(\G);\Z) \overset{\Theta}{\cong} 
    \dC(E_\ell(\G)) \cong \dC(E_g(S)) \overset{\Xi}{\cong} H_1^{sep}(\hat{S};\Z) \cong H^1(\PMap(S);\Z).
    \]
\end{proof}

\section{CB generation classification} \label{sec:CBgen} 
As an application of \Cref{thm:fluxzeromaps} and \Cref{thm:semidirectprod}, in this section we obtain \Cref{thm:CBgenClassification}, the classification of infinite graphs with CB generating sets. We revisit the theorem for convenience.

\begin{THM}[\Cref{thm:CBgenClassification}, revisited] \label{thm:CBgenClassification_revisited}
    Let $\G$ be a locally finite, infinite graph. Then $\PMap(\G)$ is CB generated if and only if either $\G$ is a tree, or satisfies the following:
    \begin{enumerate}
        \item $\G$ has finitely many ends accumulated by loops, and
        \item there is no accumulation point in $E \setminus E_\ell$.
    \end{enumerate}
\end{THM}

The only if direction of \Cref{thm:CBgenClassification} comes from \cite{DHK2023}:
\begin{PROP}[{\cite[Theorem 6.1]{DHK2023}}]\label{prop:nonCBgenAccumulation}
    Let $\G$ be a locally finite, infinite graph. If $\rk(\G)>0$ and $E \setminus E_\ell$ has an accumulation point, then $\PMap(\G)$ is not CB-generated.
\end{PROP}

\begin{PROP}[{\cite[Theorem 8.2]{DHK2023}}]\label{prop:nonCBgenInfEell}
    Let $\G$ be a locally finite, infinite graph. If $\G$ has infinitely many ends accumulated by loops, then $\PMap(\G)$ is not locally CB. In particular, $\PMap(\G)$ is not CB-generated.
\end{PROP}

Now we show that those conditions are also sufficient for CB-generation.
First, recall by \Cref{prop:treePMap} that when $\G$ is a tree, $\PMap(\G)$ is the trivial group.
We proceed to show (1) and (2) are sufficient to show that $\PMap(\G)$ is CB
generated. We start with the case where $\G$ has finite rank and satisfies Condition (2):

\begin{PROP}\label{prop:CBgenFinite}
       Let $\G$ be a locally finite, infinite graph. If $\G$ has finite rank with no accumulation point in $E$, then $\PMap(\G)$ is finitely generated.
\end{PROP}

\begin{proof}
    Note in this case $E_\ell$ is the empty set, so having no accumulation point in
    $E \setminus E_\ell$ is equivalent to having a finite end space. Hence $\PMap(\G)$ is isomorphic to one of $\Out(F_n), \Aut(F_n)$ or $F_n^e \rtimes \Aut(F_n)$ for some $e=|E|-1 \ge 1$, all of which are finitely generated, concluding the proof.
\end{proof}

Now assume $\G$ has infinite rank but finitely many ends accumulated by loops
with no accumulation point in $E \setminus E_\ell$. As in \Cref{rmk:connectsum}, $\G$ can be realized as a finite wedge sum of rays, Loch Ness monster graphs (infinite rank graph with end space $(E,E_{\ell}) \cong (\{*\},\{*\})$), and Millipede monster graphs (infinite rank graph with end space $(E,E_{\ell}) \cong (\N \cup \{\infty\}, \{\infty\})$).  
Then $\G$ is characterized by the triple $(r,l,m)$ where $r$ is the number of ray summands, $l$ is the number of Loch Ness monster summands, and $m$ is the number of Millipede monster summands. Then $\G$ is as in \Cref{fig:wedgesum}. Note that this triple is not unique, in fact, if $m>0$ then we do not need to keep track of $r$ as any additional ray can simply be moved via a proper homotopy into a Millipede monster summand. However, in order to avoid a case-by-case analysis we prove that $\PMap(\G)$ is CB-generated for \emph{any} triple $(r,l,m)$.  Note that we already know by \Cref{THM:CBclassification} that both the Loch Ness monster graph, $(0,1,0)$, and the Millipede monster graph, $(0,0,1)$, have CB and thus CB-generated pure mapping class groups. Therefore we will ignore these two graphs throughout this section.

\begin{figure}[ht!]
	    \centering
	    \def\svgwidth{.7\textwidth}
		    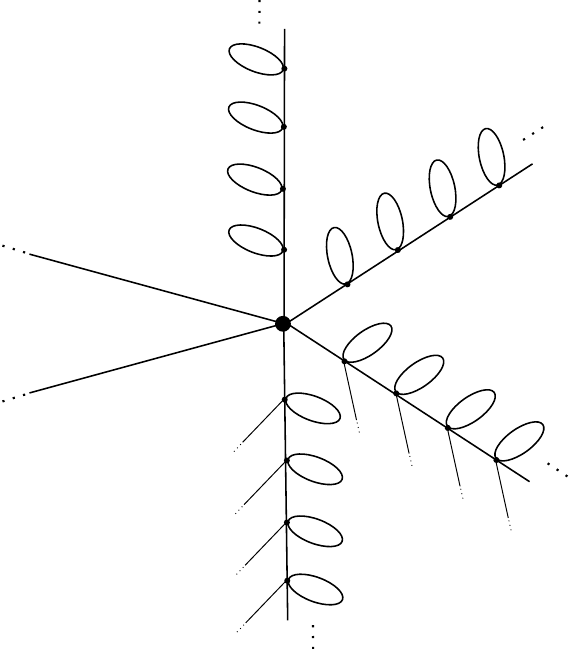
		    \caption{The graphs $\G$ that we prove have a CB-generating pure mapping class group. Each such $\G$ has a single wedge point $K$ and $\G\setminus K$ has $r$ ray components, $l$ Loch Ness monster components, and $m$ Millipede monster components.}
	    \label{fig:wedgesum}
    \end{figure}

The foundation for our choice of CB-generating set for $\PMap(\G)$ will be the set $\cV_K$, where $K$ is the wedge point as in \Cref{fig:wedgesum}. Recall that an appropriate choice of a compact set $K$ provides a CB neighborhood of the identity certifying that $\PMap(\G)$ is locally CB. 
\begin{PROP}[{\cite[Proposition 8.3]{DHK2023}}]\label{prop:CBVK}
  Let $\G$ be a locally finite, infinite graph with finitely many ends
  accumulated by loops. Then $\PMap(\G)$ is locally CB if and only if $\G
  \setminus \G_c$ has only finitely many components whose ends space is
  infinite. Moreover, for any choice of connected compact subgraph $K$ whose
  complementary components are either trees or monster graphs, $\cV_K$ is a
  CB neighborhood of the identity in $\PMap(\G)$.
\end{PROP}

We remark that the moreover statement is absent in \cite[Proposition 8.3]{DHK2023}; however, it can be deduced readily from the proof. We thus have that our choice of $\cV_{K}$ is CB. This is the starting point for our CB generating set; we now describe how to choose the remaining elements. 

Enumerate each of the ray summands of $\G$ as $R_{1},\ldots,R_{r}$, the Loch Ness monster summands as $L_{1},\ldots,L_{l}$, and the Millipede monster summands as $M_{1},\ldots,M_{m}$ (skip the enumeration if there are no summands of a given type). We also sequentially label the loops in $L_{i}$ by $a_{i,j}$ where $a_{i,1}$ is the loop closest to $K$. We similarly label the loops in $M_{i}$ by $b_{i,j}$.  For each $R_{i}$ let $I_{i}$ be an interval in the interior of $R_{i}$. Then we have the following finite collection of word maps:
\begin{align*}
    W \defeq \{\phi_{(a_{1,1},I_{i})}\}_{i=1}^{r}.
\end{align*}
If $l=0$ then we use $W:= \{\phi_{(b_{1,1}, I_{i})}\}_{i=1}^{r}$ instead. Note we cannot have $l=m=0$ as $\G$ has infinite rank. If $r=0$, we set $W := \emptyset$.

Next, we have the following finite collection of loop swaps:
\begin{align*}
    B \defeq &\{\alpha_{ij}:=\text{swaps } a_{i,1} \leftrightarrow a_{j,1}\ |\ 1 \le i < j \le l\}\\
        &\cup\{\beta_{ij}:= \text{swaps } b_{i,1} \leftrightarrow b_{j,1}\ |\ 1 \le i < j \le m \}\\
        &\cup\{\gamma_{ij}:= \text{swaps } a_{i,1} \leftrightarrow b_{j,1}\ |\ 1 \le i \le l,\ 1\le j \le m\}. 
\end{align*}
In words, $B$ is the collection of all loop swaps between loops that are adjacent to $K$. 

Finally, we need a finite collection of loop shifts. The graph $\G$ has only finitely many ends accumulated by loops, so by \Cref{cor:rankoffirstcohomology}, $H^{1}(\PMap(\G);\Z)$ has finite rank. Let $H$ be a finite collection of primitive loop shifts dual to a finite basis of $H^{1}(\PMap(\G);\Z)$.

We claim that the set 
\begin{align*}
    \cS := \cV_{K} \cup W \cup B \cup H
\end{align*}
is a CB generating set for $\PMap(\G)$. Note that $\cS$ is CB since $\cV_{K}$ is CB by \Cref{prop:CBVK} and each of $W,B,$ and $H$ is simply a finite set. Thus we only need to verify that $\cS$ is a generating set for $\PMap(\G)$. We will first check that $\cS$ generates $\PMapc(\G)$. 

\begin{LEM}\label{lem:compactcbgen}
    If $K' \subset \G$ is any connected compact subset of $\G$, then $\PMap(K') \subset \la \cS \ra$.  
\end{LEM}

Before we give the proof of this lemma we review finite generating sets for $\Aut(F_{n})$. 
Let $F_n$ be a free group of rank $n$, and denote by $a_1,\ldots,a_n$ its free generators.
In 1924, Nielsen \cite{Nielsen1924} proved a finite presentation of $\Aut(F_n)$, with the
generating set $\{\tau_i\}_{i=1}^n \cup \{\sigma_{ij}, \lambda_{ij},
\rho_{ij}\}_{1 \le i \neq j \le n}$, where:
\begin{align*}
    &\tau_i=
    \begin{cases}
        a_i \mapsto a_i^{-1}, \\
        a_j \mapsto a_j & \text{for $j \neq i$.}
    \end{cases}
    &&\sigma_{ij}=
    \begin{cases}
        a_i \leftrightarrow a_j, \\
        a_k \mapsto a_k &\text{for $k \neq i,j$.}
    \end{cases}\\
    &\lambda_{ij}=
    \begin{cases}
        a_i \mapsto a_ja_i, \\
        a_k \mapsto a_k & \text{for $k \neq i,j$.}
    \end{cases}
    &&\rho_{ij}=
    \begin{cases}
        a_i\mapsto a_ia_j, \\
        a_k \mapsto a_k & \text{for $k \neq i,j$.}
    \end{cases}
\end{align*}
We call $\tau_i$ a \textbf{flip}, $\sigma_{ij}$ a
\textbf{transposition,} and $\lambda_{ij},\rho_{ij}$ \textbf{left/right Nielsen
  automorphisms} respectively. In fact, Armstrong--Forrest--Vogtmann \cite[Theorem 1]{AFV2008presentation}
reduced this generating set to consist only of involutions:
\begin{equation}\label{eqn:AFVgenset}
\{\tau_i\}_{i=1}^n
\cup \{\sigma_{i,i+1}\}_{i=1}^{n-1}\cup\{\tau_2\lambda_{12}\}. \tag{$\dagger$}
\end{equation}

\begin{proof}[Proof of \Cref{lem:compactcbgen}]
    Let $K'$ be a connected compact subset of $\G$. Without loss of generality, we can increase the size of $K'$ so that it is as in \Cref{fig:compactcbgen}. In particular, $K'$ satisfies the following:
    \begin{itemize}
        \item $K'$ contains at least two loops of $L_{1}$ (or $M_{1}$ if $l=0$),
        \item $K'$ contains at least one loop from every monster summand,
        \item the vertices in $K'$ are contained in its interior,
        \item every component of $\G\setminus K'$ is infinite, 
        \item $K'$ is connected and contains the wedge point $K$. 
    \end{itemize}
    Note that the last two properties imply that $K'$ contains a subsegment of every ray summand $R_i$.

    \begin{figure}[ht!]
	    \centering
        \includegraphics[width=.7\textwidth]{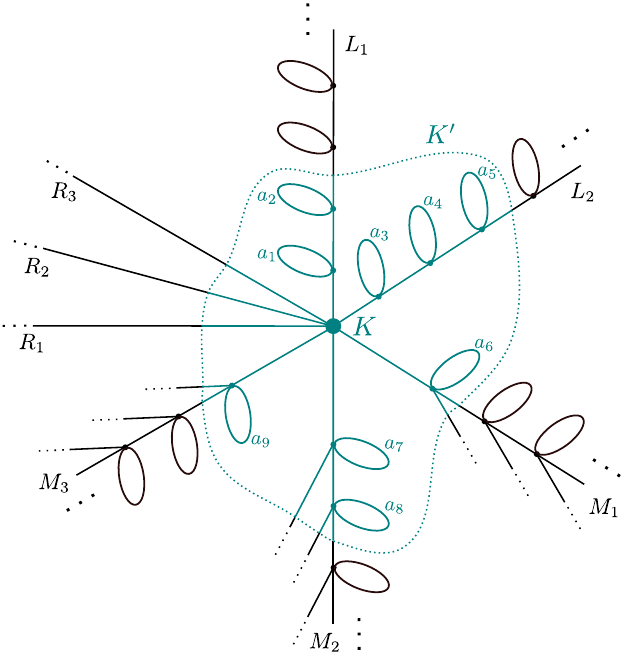}
		    \caption{Illustration of $K'$ in $\G$. Here we have $l=2, m=3, r=3$ and $\PMap(K') \cong F_9^{11} \rtimes \Aut(F_9)$. We may assume $K'$ contains: at least one loop from every monster summand, at least two loops from one of the monster summands, and initial segments of the ray summands, as well as $K$. If needed, we can further enlarge $K'$ such that it contains the vertices in the interior, and it only contains the entirety of the loops.}
	    \label{fig:compactcbgen}
    \end{figure}
    
    By \Cref{prop:RsubgroupDecomposition} and \Cref{rmk:compactMCG},
    we have that $\PMap(K') \cong F_{m}^{k} \rtimes \Aut(F_{m})$ for some $k > 0$ and $m=\rk(K')$.  We first check that $\la \cS \ra$ contains an Armstrong--Forrest--Vogtmann generating set for $\Aut(F_{m})$. Relabel the loops of $K'$ by $a_1,\ldots,a_m$ in the following manner. The loop of $L_{1}$ closest to $K$ is labeled $a_{1}$, the next loop in $L_{1}$ is $a_{2}$, continued until the loops of $L_{1}$ are exhausted. Then the next loop, say $a_{j+1}$, is the first loop on $L_{2}$ etc., until all of the loops in all of the $L_{i}$ are exhausted. Finally, continue relabeling by $a_\bullet$'s the loops in $M_{1}$ through $M_{m}$, in the same process. Note that when $l=0$, then $a_1$ and $a_2$ are contained in $M_1$.

    Note that we immediately have $\tau_{1},\ldots,\tau_m,\lambda_{12} \in \cV_{K} \subset \cS$. Therefore it remains to check that $\sigma_{i,i+1} \in \la \cS \ra$ for all $i=1,\ldots,m-1$. Each such $\sigma_{i,i+1}$ either swaps two adjacent loops in a single component of $K'\setminus K$ or swaps the last loop in a component of $K'\setminus K$ with the first loop in the next component. In the former case we already have that $\sigma_{i,i+1} \in \cV_{K}$. For the latter case, let $a_t$ be the first loop in the component of $K' \setminus K$ containing $a_i$.
    Then consider the loop swap $\sigma_{i,t}$ that swaps $a_i$ with $a_t$ (note those two loops could coincide, and then $\sigma_{i,t}$ is the identity) and let $\sigma_{t,i+1}$ be the loop swap that swaps $a_t$ with $a_{i+1}$, which is the first loop in the component of $K' \setminus K$ containing $a_{i+1}$. Then we have that $\sigma_{i,1} \in \cV_K$, $\sigma_{t,i+1} \in B$ and $\sigma_{i,i+1} = \sigma_{i,t} \sigma_{t,i+1} \sigma_{i,t} \in \la \cS\ra$. Thus we see that every Armstrong--Forrest--Vogtmann generator for the $\Aut(F_{m})$ subgroup of $\PMap(K') \cong F_{m}^{k} \rtimes \Aut(F_{m})$ is contained in $\la \cS \ra$. 

    Finally we need to be able to obtain each of the $k$ factors of $F_{m}$ in $\PMap(K')$. Each $F_{m}$ subgroup can be identified with the subgroup of the collection of word maps on an interval adjacent to the boundary of $K'$. Recall by \Cref{prop:RsubgroupDecomposition} there are $k+1$ such boundary adjacent intervals, so say $I_1,\ldots,I_{k+1}$. Since we have already generated the $\Aut(F_{m})$ subgroup of $\PMap(K')$ with $\cS$ and we can change the word of the word map using \Cref{lem:compositionLaw} and \Cref{lem:conjugationLaw}, it suffices to show that a \emph{single} word map on each interval $I_j$ that maps onto a generator of $F_{m}$ is in $\la \cS \ra$.
    However, we even have one such word map in $\cS$ already. Indeed, if $I_j$ is contained in some ray then we have already added a corresponding word map to $W$. Otherwise, if $I_j$ is contained in some monster summand, then there is an appropriate word map already in $\cV_{K}$ obtained by mapping $I_j$ over the first loop of that summand. We can thus conclude that $\PMap(K') \cong F_{m}^{k} \rtimes \Aut(F_{m})$ in contained in $\la \cS \ra$. 
\end{proof}

We are now ready to prove \Cref{thm:CBgenClassification}. Note that in the above lemma we never made use of the loop shifts in $H$. They will now be used to push any random mapping class into the closure of the compactly supported mapping classes. 

\begin{proof}[Proof of \Cref{thm:CBgenClassification}]
    As discussed in the beginning of the section, the only if direction comes from \Cref{prop:nonCBgenAccumulation} and \Cref{prop:nonCBgenInfEell}. Now we prove the if direction.
    When $\G$ is a tree, we have $\PMap(\G)=1$ by \Cref{prop:treePMap}. If $\G$ has finite rank,
    $\PMap(\G)$ is finitely generated by \Cref{prop:CBgenFinite}. Also if $\G$ is either the Loch Ness monster or Millipede monster graph, then by \Cref{THM:CBclassification}, $\PMap(\G)$ is CB. Hence we may
    assume $1 \le |E_\ell|<\infty$, there is no accumulation point in
    $E\setminus E_\ell$, and $\G$ is neither the Loch Ness monster nor the Millipede monster graphs.

    Let $\cS$ be as defined above; $\cS = \cV_K \cup W \cup B \cup H$. We will show that $\cS$ generates $\PMap(\G)$. Let $f \in \PMap(\G)$.
    If $|E_\ell|=1$, then $\PMap(\G) = \overline{\PMapc(\G)}$ by \Cref{thm:semidirectprod}, so we obtain $f \in \overline{\PMapc(\G)}$.
    Otherwise, if $|E_\ell| \ge 2$, then by postcomposing $f$ with primitive loop shifts in $H$, we may assume the flux of $f$ is zero with respect to any 2-partition of $E_\ell$. By \Cref{thm:fluxzeromaps}, we can assume $f \in \overline{\PMapc(\G)}$ for this case as well.
    
    Then there exists a compact set $K'$ containing $K$, and $g \in \PMapc(\G)$ such that $g$ is totally supported in $K'$ and $fg^{-1} \in \cV_K$. Therefore, it suffices to show that $g$ is contained in the group generated by $\cS$. Since $g$ is totally supported in $K'$, the map $g$
    can be identified with an element in $\PMap(K')$, which is contained in
    $\la \cS \ra$ by \Cref{lem:compactcbgen}. This concludes the proof that $\PMap(\G)$ is generated by $\cS$. Finally, $\cS$ is CB as it is the union of three finite sets $W,B$, and $H$ and the set $\cV_{K}$, which is CB by \Cref{prop:CBVK}.
\end{proof}

\section{Residual finiteness} \label{sec:RF}
In this section, we prove \Cref{thm:PMapRF}:
\begin{THM}[\Cref{thm:PMapRF}, revisited]\label{thm:PMapRF_revisited}
   $\PMap(\G)$ is residually finite if and only if $\G$ has a finite rank.
 \end{THM}

\subsection{Forgetful map}
\label{ssec:forgetful}
Throughout this section, we let $\G$ be a locally finite, infinite graph with no ends accumulated by loops. That is, $E_\ell(\G)=\emptyset$ but $E(\G) 
\neq \emptyset$. Fix an end $\alpha_0 \in E(\G)$.
Define $E_{<\infty}(\G)$ as the collection of finite subsets of $E(\G)$ containing $\alpha_0$:
\[
    E_{<\infty}(\G) = \{ \cE \subset E(\G): \alpha_0 \in \cE, \text{ and } |\cE|<\infty\}.
\]
For each $\cE \in E_{<\infty}(\G)$, we define the graph $\G_{\cE}$ as a subgraph of $\G$ such that:
\begin{itemize}
    \item $\rk \G_{\cE} = \rk \G$, and
    \item $E(\G_{\cE}) = \cE$.
\end{itemize}
Note $\G_\cE$ is properly homotopic to the core graph $\G_c$ of $\G$ with $|\cE|$ rays attached.

Now we use \Cref{prop:RsubgroupDecomposition}: $\PMap(\G) \cong \cR \rtimes \PMap(\G_c^*)$ if $E(\G) \setminus E_\ell(\G)$ is nonempty and compact.
In our case $\G$ is of infinite type and has no ends accumulated by loops, so $E(\G) \setminus E_\ell(\G) = E(\G)$ is automatically nonempty and compact. Now we denote by $\cR_\cE$ the $\cR$ subgroup for $\G_{\cE}$. Then we have a map $\rho_\cE: \cR \to \cR_\cE$ by `restricting' the domain to $E(\G_\cE)$. Namely, given a locally constant map $[f: E(\G) \to \pi_1(\G, \alpha_0)] \in \cR$, we define $\rho_\cE(f):= f|_{E(\G_\cE)}$, where we note that $f|_{E(\G_\cE)}:E(\G_\cE) \to \pi_1(\G,\alpha_0) = \pi_1(\G_\cE,\alpha_0)$ is a locally constant map to $\pi_1(\G_\cE,\alpha_0)$, so $\rho_\cE(f) \in \cR_\cE$.

\begin{LEM}
    The restriction map $\rho_\cE:\cR \to \cR_\cE$ is a group homomorphism.
\end{LEM}

\begin{proof}
    Recall the group operation in $\cR$ is the pointwise multiplication in $\pi_1(\G,\alpha_0)$.
    Hence the restriction on $f \cdot g$ for $f, g \in \cR$ is just the product of their restrictions:
    \[
        \rho_\cE(f\cdot g) = (f\cdot g)|_{E(\G_\cE)} = f|_{E(\G_\cE)} \cdot g|_{E(\G_\cE)} = \rho_\cE(f) \cdot \rho_\cE(g). \qedhere
    \]
\end{proof}

Observe $\G_c^* = (\G_\cE)_c^*$.
Hence, from the lemma we can build the homomorphism $\cF_\cE$ on $\PMap(\G)$ to $\PMap(\G_\cE)$ by just letting $\cF_\cE = \rho_\cE \times \Id$ on the decomposition $\PMap(\G) = \cR \rtimes \PMap(\G_c^*)$ given in \Cref{prop:RsubgroupDecomposition}:
\[
    \cF_\cE : \PMap(\G) \to \PMap(\G_\cE),
\]
 which we will call as the \textbf{forgetful homomorphism} to $\cE \subset E(\G)$.
 
\subsection{Finite rank if and only if residually finite} \label{ssec:RFiffFIN}
Now we prove \Cref{thm:PMapRF}.
The following lemma is used for the only if direction of the proof. Denote by $\SAut(F_n)$ the unique index 2 subgroup of $\Aut(F_n)$. 
\begin{FACT}[{\cite[Theorem 9.1]{baumeister2019smallest}}]
 \label{fact:minimalindex}
     There exists a strictly increasing sequence of integers $\{a_n\}_{n \ge 3}$ such that for $n \ge 3$, every nontrivial finite quotient of $\SAut(F_n)$ has cardinality of at least $a_n$. 
 \end{FACT}

 \begin{proof}[Proof of \Cref{thm:PMapRF}]
   Suppose that $\G$ has finite rank $n$. If $\G$ has no ends then $\PMap(\G)$ is isomorphic to $\Out(F_n)$, which is residually finite by \cite{grossman1974}. If $\G$ has only one end, then $\PMap(\G)$ is isomorphic to $\Aut(F_n)$, which is residually finite by \cite{baumslag1963}. If $\G$ has finitely many ends, then $\PMap(\G) \cong F_n^{{|E|-1}} \rtimes \Aut(F_n)$ which is again residually finite as both factors are residually finite and $F_n^{|E|-1}$ is finitely generated \cite[Theorem 1]{mal1958homomorphisms}.
   
   Now we assume $\G$ has finite rank and infinitely many ends. The proof is similar to the proof for infinite-type surfaces; \cite[Proposition 4.6]{PatelVlamis}. Let $f \in \PMap(\G)$ be a nontrivial element. Since $\G$ is of finite rank and $f$ is proper, it follows that $(\G \setminus \G_c) \cap \supp(f)$ is compact. In particular, there exists some finite set $\cE \subset E$ such that $(\G_\cE \setminus \G_c) \cap \supp(f)$ is still nonempty. This implies that the forgetful map $\cF_\cE: \PMap(\G) \to \PMap(\G_\cE)$ sends $f$ to a nontrivial element $\cF_\cE(f) \in \PMap(\G_\cE)$.
   However, we know that $\cE$ has finite end space so $\PMap(\G_\cE)$ is residually finite by the previous paragraph. Therefore, there exists a homomorphism $\psi: \PMap(\G_\cE) \to F$ for some finite group $F$ so that $\psi(\cF_\cE(f))$ is nontrivial. Thus $\PMap(\G)$ is residually finite.

   Conversely, let $\G$ have infinite rank and assume it is in standard form.
   Let $\{\G_{k}\}$ be a compact exhaustion of $\G$ by connected subgraphs. Then there exist non-decreasing sequences $\{n_k\}, \{e_k\}$ such that $\PMap(\G_k) \cong F_{n_k}^{e_k} \rtimes \Aut(F_{n_k})$. Here $e_k+1$ is the number of boundaries of $\G_k$ (i.e. the size of $\overline{\G \setminus \G_k} \cap \G_k$), and $n_k$ is the rank of $\G_k$. As $\G$ has infinite rank, we have $\lim_{k \to \infty}n_k = \infty.$ Also, note $\Aut(F_{n_k})$ has the unique index 2 subgroup $\SAut(F_{n_k})$ for each $k$, whose isomorphic copy in $\PMap(\G_k)$ we denote by $G_k$. The group $\Aut(F_{n_k})$ can be identified with the subgroup of mapping classes totally supported on the core graph of $\G_k$, and $G_{k} \cong \SAut(F_{n_{k}})$ with the set of those mapping classes that preserve orientation. Since the core graph of $\G_{k}$ is contained in the core graph of $\G_{k+1}$, and orientation preserving mapping classes on $\G_k$ are orientation preserving on $\G_{k+1}$, it follows that we have the inclusion $G_k \hookrightarrow G_{k+1}$. Hence the direct limit $\SAut_\infty(\G) := \varinjlim G_k$ is a well-defined subgroup of $\PMap(\G)$.
   
   We claim that $\SAut_{\infty}(\G)$ has no finite quotients. To do this, suppose $H$ is a proper normal subgroup of $\SAut_{\infty}(\G)$ with finite index $r \ge 2$.
   Then as $H$ is a proper subgroup of $\SAut_{\infty}(\G)$ and $\varinjlim G_k = \SAut_\infty(\G)$, it follows that there exists some $k_{0}$ such that whenever $k \ge k_{0}$, $G_k$ is not contained in $H$.
   Hence $H \cap G_k$ is a \emph{proper} normal subgroup of $G_k$. Note
   \[
        1 \neq [G_k : G_k \cap H] \le [\SAut_\infty(\G) : H] = r,
   \]
   but the minimal finite index of proper subgroups of $G_k \cong \SAut(F_{n_k})$ increases as $k$ does by \Cref{fact:minimalindex}. Therefore, $[G_k : G_k \cap H]$ cannot be uniformly bounded by $r$; giving a contradiction. Therefore $\SAut_\infty(\G)$ has no nontrivial finite quotient, implying that both $\PMap(\G)$ and $\Map(\G)$ are not residually finite. 
   \end{proof}

\section{Tits alternative}
\label{sec:TitsAlternative}

In a series of three papers \cite{BFH2000,BFH2005,BFH2004}, Bestvina, Feighn, and Handel prove that $\Out(F_{n})$ satisfies what we call the \textbf{strong Tits alternative}: Every subgroup either contains a nonabelian free group or is virtually abelian. The same was previously known for mapping class groups of compact surfaces by work of Ivanov, McCarthy, and Birman--Lubotzky--McCarthy \cite{Ivanov1984,McCarthy1985,BLM1983}. However, it was shown by Lanier and Loving \cite{LL2020} that this is not the case for big mapping class groups. They prove that big mapping class groups \emph{never} satisfy the strong Tits alternative by showing that they always contain a subgroup isomorphic to the wreath product $\Z \wr \Z$. In \cite{abbott2021finding}, the authors extend this idea to find many subgroups isomorphic to wreath products. Allcock \cite{Allcock2021} further showed that most big mapping class groups fail the (standard) Tits alternative by finding a poison subgroup that surjects onto a Grigorchuck group. A groups satisfies the \textbf{Tits alternative} if every subgroup either contains a nonabelian subgroup or is virtually solvable.  Note that some references require subgroups be finitely generated, but we do not need to make that restriction.

\subsection{Infinite rank: Fails to satisfy TA}
In this section, we find poison subgroups (analogous to the surface case) in $\PMap(\G)$ for graphs $\G$ of infinite rank.

\begin{THM}\label{thm:Wreath}
    Let $\G$ be a locally finite graph of infinite rank. Then $\PMap(\G)$ contains a subgroup isomorphic to $\Aut(F_n) \wr \Z$ for all $n\in \N$.
\end{THM} 

\begin{proof}
    Recall that to define the wreath product $G \wr \Z$, we need a $\Z$-indexed set of copies of $G$, denoted by $\{G_i\}_{i \in \Z}$. Then $\Z$ acts on the index set by translation, so it also acts on $\oplus_{i \in \Z} G_i$ by translation on the indices. Now set $G = \Aut(F_n)$ and denote by $\phi$ the translation action by $\Z$ on the direct sum. We then define
    \[
        \Aut(F_n) \wr \Z :=\left(\bigoplus_{\Z} \Aut(F_n)\right) \rtimes_{\phi} \Z. 
    \]  
    To realize this group as a subgroup of $\PMap(\G)$, we will find $\Z$ copies of $\Aut(F_n)$ together with a translation action.

    For each $n\in \N$, let $\Delta_n$ be the graph obtained from a line identified with $\R$ with a wedge of $n$ circles attached by an edge at each integer point; see \Cref{fig:TA_2end}. If $\G$ has at least two ends accumulated by loops, we can properly homotope $\G$ to have $\Delta_n$ as a subgraph. 

    \begin{figure}[ht!]
        \centering
        \includegraphics[width=.8\textwidth]{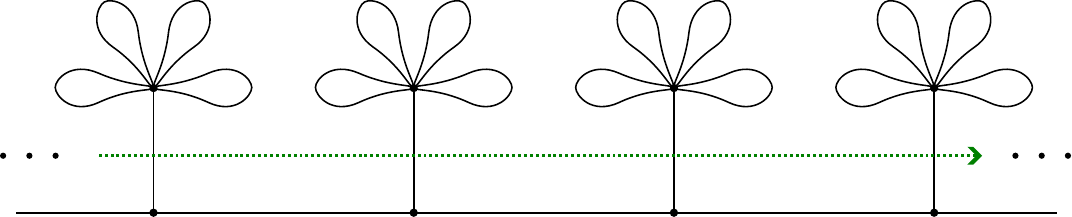}
        \caption{The graph $\Delta_4$, with two ends accumulated by roses with $4$ petals. It admits a translation of roses, denoted by the green dotted arrow. Up to proper homotopy, such graph arises as a subgraph of any graph with at least two ends accumulated by loops.}
        \label{fig:TA_2end}
    \end{figure}

     For each $i \in \Z$, let $R_i$ be the wedge of circles supported above the integer point $i$ in $\Delta_n \subset \Gamma$. Let $G_i$ be the subgroup of elements of $\PMap(\G)$ which are totally supported on $R_i$. Each $G_i$ is isomorphic to $\Aut(F_n)$ (see \Cref{rmk:compactMCG}) and the $G_i$'s have disjoint total support, so $\<G_i\>_{i \in \Z}\cong \oplus_{\Z} \Aut(F_n)$. There is a shift map, call it $h$, that translates along $\Delta_n$ by $+1$ on the line and maps $R_i$ to $R_{i+1}$ isometrically. Because $h^mG_i=G_{i+m}h^m$, the subgroup of $\PMap(\G)$ generated by $G_0$ and $h$ is isomorphic to $\Aut(F_n)\wr \Z$.

    In general, if $\G$ has at least one end accumulated by loops, we can embed a copy of $\Delta_n$ into $\G$ where the images of the two ends of $\Delta_n$ are not distinct. The corresponding ``shift map" will no longer be shifting between distinct ends, but this does not affect the construction of an $\Aut(F_n)\wr \Z$ subgroup.
\end{proof}

\Cref{thm:Wreath} immediately tells us that $\PMap(\G)$ fails the strong Tits alternative because $\Z \wr \Z$ is a subgroup of $\Aut(F_n) \wr \Z$. In \cite{Allcock2021}, Allcock shows that big mapping class groups of surfaces with infinite genus fail the (standard) Tits alternative.
His idea is to find elements of the mapping class group that ``look like'' the action of the Grigorchuck group on a rooted binary tree. Because these elements are not of finite order, the resulting subgroup of the mapping class group is an extension of the Grigorchuck group. When this same idea is implemented in the pure mapping class group of a graph, we instead find an exact copy of Grigorchuck's group. Many graphs, such as an infinite binary tree, also contain Grigorchuck's group as a subgroup of their \emph{full} mapping class group in the obvious way.

\begin{THM} \label{thm:Grigorchuck}
     Let $\G$ be a locally finite graph of infinite rank. Then $\PMap(\G)$ contains a subgroup isomorphic to the Grigorchuck group.
\end{THM}

\begin{proof}
    First, we define the proper homotopy equivalences called $a,b,c,d$ on an infinite binary tree $T$ as in \Cref{fig:Grigorchuck}. Note only $a$ swaps the level 1 branches. Each of the other three homotopy equivalences $b,c,d$ misses $(3k+1), 3k, (3k-1)$-st branch swaps for $k \ge 1$ respectively, as well as the level-1 swap. These four elements generate the Grigorchuck group, $G$ \cite{grigorchuck1980burnside,grigorchuck2008intermediate}.

    \begin{figure}[ht!]
        \centering
        \includegraphics[width=.8\textwidth]{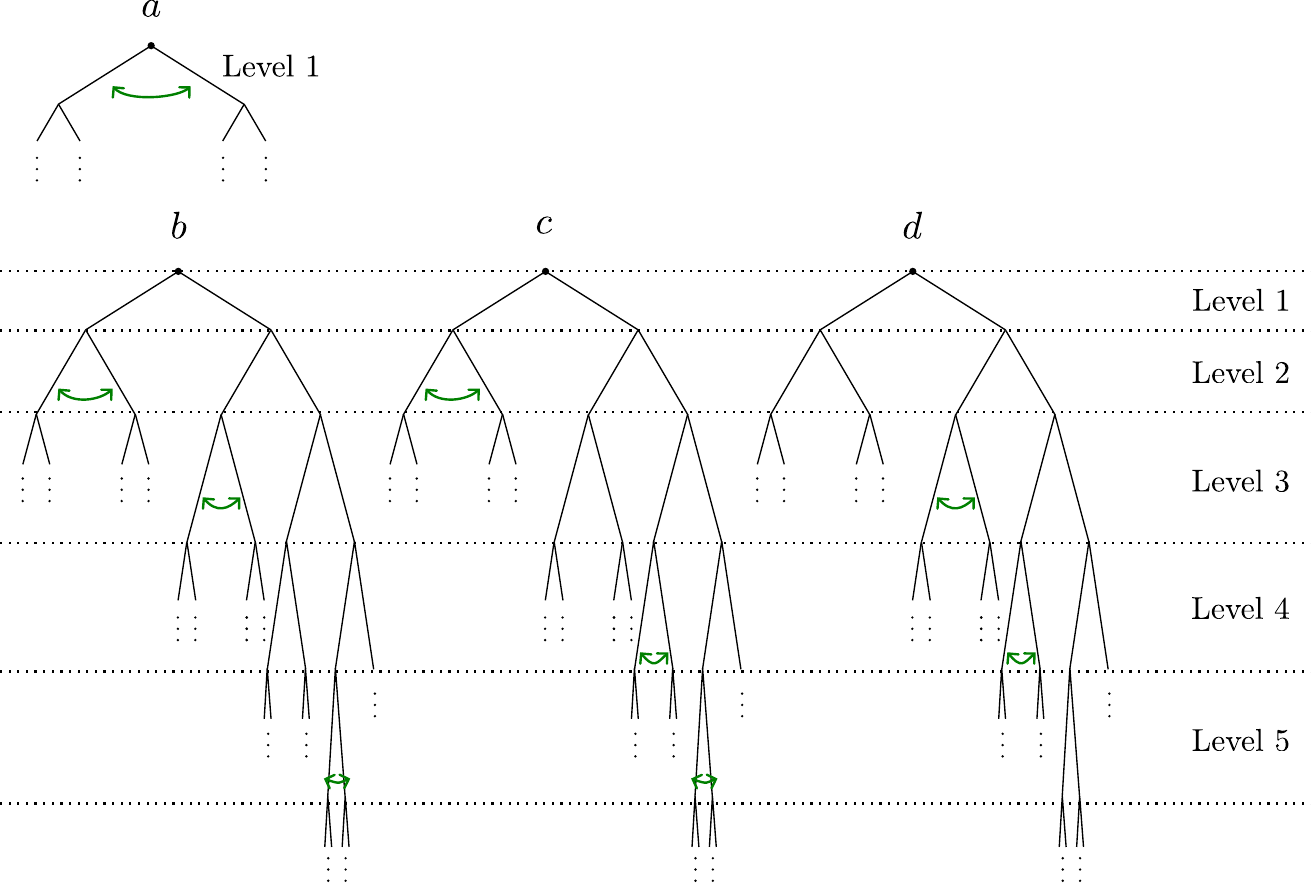}
        \caption{Proper homotopy equivalences $a,b,c$ and $d$ on infinite binary tree $T$. Each green arrow denotes the swap of the two subtrees induced by the swap of the two branches.}
        \label{fig:Grigorchuck}
    \end{figure}

    Now let $\Delta$ be the infinite graph with one end accumulated by loops, constructed as in \Cref{fig:1endGrigorchuck}. Specifically, we start with a ray and label the countably many vertices by $v_1,v_2,\ldots$ etc. Attach a finite binary tree $T_i$ of level $i$ to $v_i$ for each $i \ge 1$. Then attach a single loop at each leaf of the trees.
    For any graph $\G$ with infinite rank, we can apply a proper homotopy equivalence so that $\Delta$ is a subgraph. Hence, $\PMap(\Delta) \le \PMap(\G)$, so it suffices to find a Grigorchuck group inside $\PMap(\Delta)$.
    \begin{figure}[ht!]
        \centering
        \includegraphics[width=.8\textwidth]{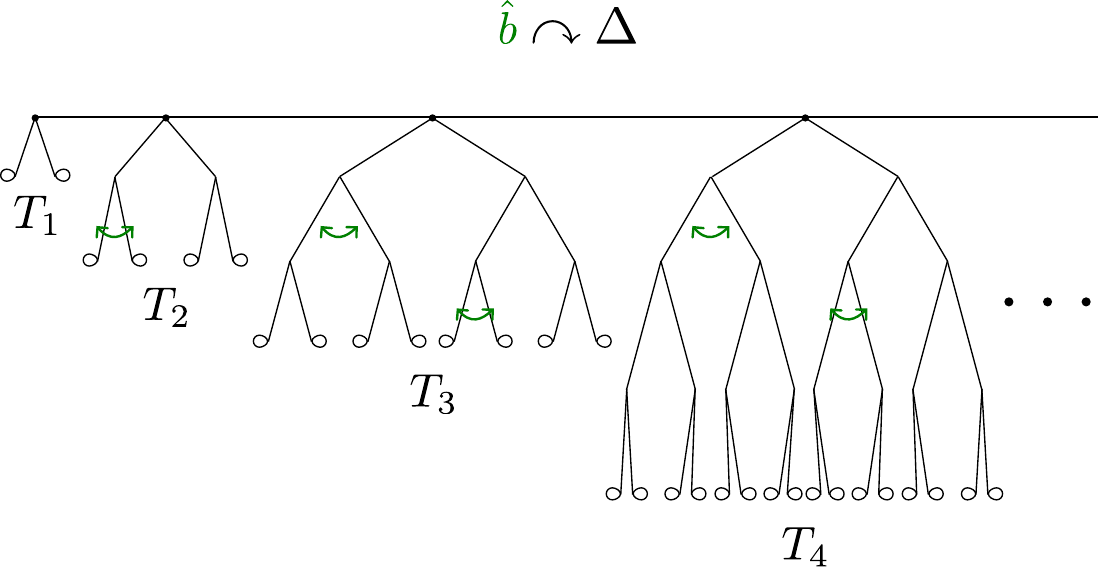}
        \caption{The graph $\Delta$ and the proper homotopy equivalence $\hat{b}$ on $T_1,T_2,\ldots \subset \Delta$ mimicking the definition of $b$ on $T$.}
        \label{fig:1endGrigorchuck}
    \end{figure}
    
    Define a proper homotopy equivalence $\hat{b}$ as the map on \emph{finite} binary trees $T_1,T_2,\ldots$ by `mimicking' $b$ defined on the \emph{infinite} binary tree $T$.
    See \Cref{fig:1endGrigorchuck} for an illustration of $\hat{b}$, denoted in green arrows. Similarly define $\hat{a},\hat{c}$ and $\hat{d}$ from $a,c$ and $d$. Denote by $\wt{a},\wt{b},\wt{c}$ and $\wt{d}$ the proper homotopy classes of $\hat{a},\hat{b},\hat{c}$ and $\hat{d}$ respectively. Following the same proof of \cite[Lemma 4.1]{Allcock2021}, we see that $\wt{a},\wt{b},\wt{c},\wt{d}$ satisfy exactly the defining relations of $G$, and $\wt{G}:=\<\wt{a},\wt{b},\wt{c},\wt{d}\>$ is isomorphic to the Grigorchuck group.
\end{proof}

\begin{COR}\label{cor:nonTA}
     Let $\G$ be a locally finite graph of infinite rank. Then $\PMap(\G)$ and $\Map(\G)$ fail the Tits alternative. 
\end{COR}

\subsection{Finite rank: Satisfies TA}
On the other hand, when $\G$ has finite rank, we get the following contrasting result.

\begin{THM}\label{thm:TA_finite}
    Let $\G$ be a locally finite graph with finite rank. Then $\PMap(\G)$ satisfies the Tits alternative. That is, every finitely generated subgroup is either virtually solvable or contains a free group.
\end{THM}

We first need the following stability property of the Tits alternative.

\begin{PROP}\label{prop:TAprops}
    Satisfying the Tits alternative is stable under subgroups, finite-index supergroups, and group extensions. More precisely,
    \begin{enumerate}
        \item Let $H \le G$. If $G$ satisfies the Tits alternative, then so does $H$.
        \item Let $H \le G$ with $[G:H]<\infty$. If $H$ satisfies the Tits alternative, then so does $G$.
        \item \textrm{(cf. \cite[Proposition 6.3]{Cantat2011})} Suppose the groups $K,G,H$ form a short exact sequence as follows:
        \[
            1 \longrightarrow K \longrightarrow G \longrightarrow H \longrightarrow 1.
        \]
        If $K$ and $H$ satisfy the Tits alternative, then so does $G$.
    \end{enumerate}
\end{PROP}

\begin{proof}
    Claim (1) holds because every subgroup of $H$ is a subgroup of $G$. Claim
    (2) will follow from (3), as finite groups are virtually trivial so satisfy the Tits alternative.
    
    Now we prove (3). Let $L \le G$ be a subgroup. Then we have the following commutative diagram:
    \[
    \begin{tikzcd}
    1 \rar &K \rar &G \rar{q} &H \rar &1 \\
    1 \rar &K \cap L \rar  \arrow[u, hook]  &L \rar{q} \arrow[u, hook] &q(L)  \rar \arrow[u, hook] &1
    \end{tikzcd}
    \]
    Indeed, $K\cap L \trianglelefteq L$ and $q(L) \cong L/(K\cap L) \le H$. By (1), both $K \cap L$ and $q(L)$ satisfy Tits alternative. If $K \cap L$ has $F_2$ as a subgroup, then $L$ has $F_2$ as a subgroup. If $q(L)$ has $F_2$ has a subgroup, then we can find a section of $q$ to lift $F_2$ inside $L$. Hence, we may assume both $K \cap L$ and $q(L)$ are virtually solvable. In this case, the following fact finishes the proof.
    \begin{FACT}[{\cite[Lemma 5.5]{dinh2012}, see also \cite[Lemme 6.1]{Cantat2011}}]
    \label{fact:vsol}
        Suppose $N$ is a normal subgroup of a group $G$. If both $N$ and $G/N$ are virtually solvable, then $G$ is virtually solvable.
    \end{FACT}\noindent
    Hence $L$ is virtually solvable, concluding that $G$ satisfies Tits alternative.
\end{proof}

Now we are ready to prove \Cref{thm:TA_finite}.

\begin{proof}[Proof of \Cref{thm:TA_finite}]
    Let $\rk \G = n.$
    Then we have the following short exact sequence \cite[Theorem 3.5]{AB2021}:
    \[
        1 \longrightarrow \cR \longrightarrow \PMap(\G) \longrightarrow \Aut(F_n) \longrightarrow 1,
    \] where $\cR$ is the group of locally constant functions from $E$ to $F_n$ with pointwise multiplication.  

    The subgroup of $\Out(F_{n+1})$ fixing one $\Z$ factor is naturally isomorphic to $\Aut(F_n)$. Recall that $\Out(F_{n+1})$ satisfies the Tits alternative by \cite{BFH2000}, so $\Aut(F_n)$ does too.  We will show that $\cR$ satisfies the (strong) Tits alternative, then \Cref{prop:TAprops} part (3), guarantees that $\PMap(\G)$ satisfies the Tits alternative as well.
    
    \begin{CLAIME}
        $\cR$ satisfies the strong Tits alternative.
    \end{CLAIME}
    Consider a (not necessarily finitely generated) subgroup $H \subset \cR$. Suppose there exist $\phi, \psi \in \cR$ that do not commute; so there exists an $e \in E$ such that $\phi(e)\psi(e) \neq \psi(e) \phi(e)$. Now we use the Ping Pong lemma to prove $\< \phi, \psi \> \cong F_2$, which will conclude that $\cR$ satisfies the strong Tits alternative. Let $X_{\phi(e)}$ and $X_{\psi(e)}$ as the set of words in $F_n$ that start with $\phi(e)$ and $\psi(e)$ respectively. We note $X_{\phi(e)}$ and $X_{\psi(e)}$ are disjoint as otherwise $\phi(e)=\psi(e)$, contradicting the assumption $\phi(e)\psi(e) \neq \psi(e)\phi(e)$. We consider the action of $H$ on $F_n$ as:
    \[
        \phi \cdot w := \phi(e)w, \qquad \text{for $\phi \in \cR$ and $w \in F_n$}.
    \]
    Then the same assumption $\phi(e)\psi(e) \neq \psi(e) \phi(e)$ implies that $\phi \cdot X_{\psi(e)} \subset X_{\phi(e)}$ and $\psi \cdot X_{\phi(e)} \subset X_{\psi(e)}$. Therefore, by the Ping-Pong lemma, we conclude $\<\phi, \psi\> \cong F_2$ so $\cR$ satisfies the (strong) Tits alternative.
\end{proof}

We now extend these results to determine which full mapping class groups satisfy the Tits alternative. 

\begin{COR} \label{cor:FullMapTA}
    Let $\G$ be a locally finite, infinite graph.
    Then $\Map(\G)$ satisfies the Tits alternative if and only if $\G$ has finite rank and finite end space.
\end{COR}

\begin{proof}
    We divide into cases. First, if $\G$ has at least one end accumulated by loops, then $\Map(\G)$ fails the Tits alternative by \Cref{cor:nonTA}. Otherwise, $\G$ has finite rank, and we continue to divide into cases. If $\G$ has finite end space, then $\Map(\G)$ is a finite extension of $\PMap(\G)$, so by \Cref{prop:TAprops} property (2), the full mapping class group $\Map(\G)$ satisfies the Tits alternative. If $\G$ has countable end space, then we can modify the proof of \Cref{thm:Grigorchuck} by replacing the loops with rays, to realize Grigorchuck's group as a subgroup of $\Map(\G)$. If the end space of $\G$ is uncountable, then there is a closed subset of the ends which is homeomorphic to the whole Cantor set, so contains Grigorchuck's group in the natural way, and again $\Map(\G)$ fails the Tits alternative. 
\end{proof}

The strong Tits alternative is not stable under group extensions (consider $\Z\wr \Z$). So, the best we could conclude about $\PMap(\G)$ from the decomposition as $\cR \rtimes \Aut(F_n)$ was \Cref{thm:TA_finite}.
However, our proof that $\cR$ satisfies the strong Tits alternative actually shows the slightly stronger statement: Any two elements of $\mathcal{R}$ which do not commute generate $F_2$. This property could be useful in answering the question, \begin{Q}\label{q:strongTA}
    If $\G$ is a locally finite graph of finite rank, does $\PMap(\G)$ satisfy the strong Tits alternative?
\end{Q}

\bibliographystyle{alpha}
\bibliography{bib}

\end{document}

%% file: pics/cantorhom.pdf_tex
%% Creator: Inkscape 1.2.1 (9c6d41e410, 2022-07-14), www.inkscape.org
%% Accompanies image file 'cantorhom.pdf' (pdf, eps, ps)
\begingroup%
  \makeatletter%
  \providecommand\color[2][]{%
    \errmessage{(Inkscape) Color is used for the text in Inkscape, but the package 'color.sty' is not loaded}%
    \renewcommand\color[2][]{}%
  }%
  \providecommand\transparent[1]{%
    \errmessage{(Inkscape) Transparency is used (non-zero) for the text in Inkscape, but the package 'transparent.sty' is not loaded}%
    \renewcommand\transparent[1]{}%
  }%
  \providecommand\rotatebox[2]{#2}%
  \newcommand*\fsize{\dimexpr\f@size pt\relax}%
  \newcommand*\lineheight[1]{\fontsize{\fsize}{#1\fsize}\selectfont}%
  \ifx\svgwidth\undefined%
    \setlength{\unitlength}{148.7446193bp}%
    \ifx\svgscale\undefined%
      \relax%
    \else%
      \setlength{\unitlength}{\unitlength * \real{\svgscale}}%
    \fi%
  \else%
    \setlength{\unitlength}{\svgwidth}%
  \fi%
  \global\let\svgwidth\undefined%
  \global\let\svgscale\undefined%
  \makeatother%
  \begin{picture}(1,0.6543108)%
    \lineheight{1}%
    \setlength\tabcolsep{0pt}%
    \put(0,0){\includegraphics[width=\unitlength,page=1]{cantorhom.pdf}}%
    \put(0.32103919,0.4824797){\color[rgb]{0,0,0}\makebox(0,0)[lt]{\lineheight{1.25}\smash{\begin{tabular}[t]{l}$0$\end{tabular}}}}%
    \put(0.12482335,0.22863016){\color[rgb]{0,0,0}\makebox(0,0)[lt]{\lineheight{1.25}\smash{\begin{tabular}[t]{l}$00$\end{tabular}}}}%
    \put(0.01149387,0.0859654){\color[rgb]{0,0,0}\makebox(0,0)[lt]{\lineheight{1.25}\smash{\begin{tabular}[t]{l}$000$\end{tabular}}}}%
    \put(0.67066,0.22863016){\color[rgb]{0,0,0}\makebox(0,0)[lt]{\lineheight{1.25}\smash{\begin{tabular}[t]{l}$10$\end{tabular}}}}%
    \put(0.55733044,0.0859654){\color[rgb]{0,0,0}\makebox(0,0)[lt]{\lineheight{1.25}\smash{\begin{tabular}[t]{l}$100$\end{tabular}}}}%
    \put(0.24962841,0.0859654){\color[rgb]{0,0,0}\makebox(0,0)[lt]{\lineheight{1.25}\smash{\begin{tabular}[t]{l}$010$\end{tabular}}}}%
    \put(0.7954649,0.0859654){\color[rgb]{0,0,0}\makebox(0,0)[lt]{\lineheight{1.25}\smash{\begin{tabular}[t]{l}$110$\end{tabular}}}}%
  \end{picture}%
\endgroup%

%% file: pics/oneendblowup.pdf_tex
%% Creator: Inkscape 1.2.1 (9c6d41e410, 2022-07-14), www.inkscape.org
%% Accompanies image file 'oneendblowup.pdf' (pdf, eps, ps)
\begingroup%
  \makeatletter%
  \providecommand\color[2][]{%
    \errmessage{(Inkscape) Color is used for the text in Inkscape, but the package 'color.sty' is not loaded}%
    \renewcommand\color[2][]{}%
  }%
  \providecommand\transparent[1]{%
    \errmessage{(Inkscape) Transparency is used (non-zero) for the text in Inkscape, but the package 'transparent.sty' is not loaded}%
    \renewcommand\transparent[1]{}%
  }%
  \providecommand\rotatebox[2]{#2}%
  \newcommand*\fsize{\dimexpr\f@size pt\relax}%
  \newcommand*\lineheight[1]{\fontsize{\fsize}{#1\fsize}\selectfont}%
  \ifx\svgwidth\undefined%
    \setlength{\unitlength}{576.74558676bp}%
    \ifx\svgscale\undefined%
      \relax%
    \else%
      \setlength{\unitlength}{\unitlength * \real{\svgscale}}%
    \fi%
  \else%
    \setlength{\unitlength}{\svgwidth}%
  \fi%
  \global\let\svgwidth\undefined%
  \global\let\svgscale\undefined%
  \makeatother%
  \begin{picture}(1,1.09893887)%
    \lineheight{1}%
    \setlength\tabcolsep{0pt}%
    \put(0,0){\includegraphics[width=\unitlength,page=1]{oneendblowup.pdf}}%
    \put(0.00717461,0.72711133){\color[rgb]{0,0,0}\makebox(0,0)[lt]{\lineheight{1.25}\smash{\begin{tabular}[t]{l}$\iota(A_{i_{1}})$\end{tabular}}}}%
    \put(0,0){\includegraphics[width=\unitlength,page=2]{oneendblowup.pdf}}%
    \put(0.10312154,0.72711133){\color[rgb]{0,0,0}\makebox(0,0)[lt]{\lineheight{1.25}\smash{\begin{tabular}[t]{l}$\iota(A_{i_{2}})$\end{tabular}}}}%
    \put(0,0){\includegraphics[width=\unitlength,page=3]{oneendblowup.pdf}}%
    \put(0.19879517,0.72711133){\color[rgb]{0,0,0}\makebox(0,0)[lt]{\lineheight{1.25}\smash{\begin{tabular}[t]{l}$\iota(A_{i_{3}})$\end{tabular}}}}%
    \put(0,0){\includegraphics[width=\unitlength,page=4]{oneendblowup.pdf}}%
    \put(0.2947421,0.72711133){\color[rgb]{0,0,0}\makebox(0,0)[lt]{\lineheight{1.25}\smash{\begin{tabular}[t]{l}$\iota(A_{i_{4}})$\end{tabular}}}}%
    \put(0,0){\includegraphics[width=\unitlength,page=5]{oneendblowup.pdf}}%
    \put(0.55861244,0.58797936){\color[rgb]{0,0,0}\makebox(0,0)[lt]{\lineheight{1.25}\smash{\begin{tabular}[t]{l}$\iota(A_{i_{1}})$\end{tabular}}}}%
    \put(0.65455937,0.58797936){\color[rgb]{0,0,0}\makebox(0,0)[lt]{\lineheight{1.25}\smash{\begin{tabular}[t]{l}$\iota(A_{i_{2}})$\end{tabular}}}}%
    \put(0.75023301,0.58797936){\color[rgb]{0,0,0}\makebox(0,0)[lt]{\lineheight{1.25}\smash{\begin{tabular}[t]{l}$\iota(A_{i_{3}})$\end{tabular}}}}%
    \put(0.84617994,0.58797936){\color[rgb]{0,0,0}\makebox(0,0)[lt]{\lineheight{1.25}\smash{\begin{tabular}[t]{l}$\iota(A_{i_{4}})$\end{tabular}}}}%
    \put(0.64203506,0.00333431){\color[rgb]{0,0,0}\makebox(0,0)[lt]{\lineheight{1.25}\smash{\begin{tabular}[t]{l}$\iota(A_{i_{1}})$\end{tabular}}}}%
    \put(0.73798199,0.00333431){\color[rgb]{0,0,0}\makebox(0,0)[lt]{\lineheight{1.25}\smash{\begin{tabular}[t]{l}$\iota(A_{i_{2}})$\end{tabular}}}}%
    \put(0.83365563,0.00333431){\color[rgb]{0,0,0}\makebox(0,0)[lt]{\lineheight{1.25}\smash{\begin{tabular}[t]{l}$\iota(A_{i_{3}})$\end{tabular}}}}%
    \put(0.03589604,0.14302357){\color[rgb]{0,0,0}\makebox(0,0)[lt]{\lineheight{1.25}\smash{\begin{tabular}[t]{l}$\iota(A_{i_{1}})$\end{tabular}}}}%
    \put(0.13184296,0.14302357){\color[rgb]{0,0,0}\makebox(0,0)[lt]{\lineheight{1.25}\smash{\begin{tabular}[t]{l}$\iota(A_{i_{2}})$\end{tabular}}}}%
    \put(0.2275166,0.14302357){\color[rgb]{0,0,0}\makebox(0,0)[lt]{\lineheight{1.25}\smash{\begin{tabular}[t]{l}$\iota(A_{i_{3}})$\end{tabular}}}}%
  \end{picture}%
\endgroup%

%% file: pics/wedgesum.pdf_tex
%% Creator: Inkscape 1.2.2 (b0a84865, 2022-12-01), www.inkscape.org
%% Accompanies image file 'wedgesum.pdf' (pdf, eps, ps)
\begingroup%
  \makeatletter%
  \providecommand\color[2][]{%
    \errmessage{(Inkscape) Color is used for the text in Inkscape, but the package 'color.sty' is not loaded}%
    \renewcommand\color[2][]{}%
  }%
  \providecommand\transparent[1]{%
    \errmessage{(Inkscape) Transparency is used (non-zero) for the text in Inkscape, but the package 'transparent.sty' is not loaded}%
    \renewcommand\transparent[1]{}%
  }%
  \providecommand\rotatebox[2]{#2}%
  \newcommand*\fsize{\dimexpr\f@size pt\relax}%
  \newcommand*\lineheight[1]{\fontsize{\fsize}{#1\fsize}\selectfont}%
  \ifx\svgwidth\undefined%
    \setlength{\unitlength}{272.6565549bp}%
    \ifx\svgscale\undefined%
      \relax%
    \else%
      \setlength{\unitlength}{\unitlength * \real{\svgscale}}%
    \fi%
  \else%
    \setlength{\unitlength}{\svgwidth}%
  \fi%
  \global\let\svgwidth\undefined%
  \global\let\svgscale\undefined%
  \makeatother%
  \begin{picture}(1,1.14242864)%
    \lineheight{1}%
    \setlength\tabcolsep{0pt}%
    \put(0,0){\includegraphics[width=\unitlength,page=1]{wedgesum.pdf}}%
    \put(0.68572265,0.98420197){\color[rgb]{0,0,0}\makebox(0,0)[lt]{\lineheight{1.25}\smash{\begin{tabular}[t]{l}$l$\end{tabular}}}}%
    \put(0,0){\includegraphics[width=\unitlength,page=2]{wedgesum.pdf}}%
    \put(0.74095573,0.11883243){\color[rgb]{0,0,0}\makebox(0,0)[lt]{\lineheight{1.25}\smash{\begin{tabular}[t]{l}$m$\end{tabular}}}}%
    \put(0,0){\includegraphics[width=\unitlength,page=3]{wedgesum.pdf}}%
    \put(0.03874455,0.55770612){\color[rgb]{0,0,0}\makebox(0,0)[lt]{\lineheight{1.25}\smash{\begin{tabular}[t]{l}$r$\end{tabular}}}}%
    \put(0,0){\includegraphics[width=\unitlength,page=4]{wedgesum.pdf}}%
    \put(0.54151048,0.56568736){\color[rgb]{0,0,0}\makebox(0,0)[lt]{\lineheight{1.25}\smash{\begin{tabular}[t]{l}$K$\end{tabular}}}}%
  \end{picture}%
\endgroup%

%% file: kernel_of_flux_maps.bbl
\newcommand{\etalchar}[1]{$^{#1}$}
\begin{thebibliography}{BKMM12}

\bibitem[AB21]{AB2021}
Yael Algom{-}Kfir and Mladen Bestvina.
\newblock Groups of proper homotopy equivalences of graphs and {N}ielsen
  realization.
\newblock {\em arXiv preprint arXiv:2109.06908}, 2021.

\bibitem[ADMQ90]{ayala1990proper}
R.~Ayala, E.~Dominguez, A.~M\'{a}rquez, and A.~Quintero.
\newblock Proper homotopy classification of graphs.
\newblock {\em Bull. London Math. Soc.}, 22(5):417--421, 1990.

\bibitem[AFV08]{AFV2008presentation}
Heather Armstrong, Bradley Forrest, and Karen Vogtmann.
\newblock A presentation for {${\rm Aut}(F_n)$}.
\newblock {\em J. Group Theory}, 11(2):267--276, 2008.

\bibitem[AHL{\etalchar{+}}21]{abbott2021finding}
Carolyn~R. Abbott, Hannah Hoganson, Marissa Loving, Priyam Patel, and Rachel
  Skipper.
\newblock Finding and combining indicable subgroups of big mapping class
  groups.
\newblock {\em arXiv preprint arXiv:2109.05976}, 2021.

\bibitem[All21]{Allcock2021}
Daniel Allcock.
\newblock Most big mapping class groups fail the {T}its alternative.
\newblock {\em Algebr. Geom. Topol.}, 21(7):3675--3688, 2021.

\bibitem[APV20]{APV2020}
Javier Aramayona, Priyam Patel, and Nicholas~G Vlamis.
\newblock The first integral cohomology of pure mapping class groups.
\newblock {\em International Mathematics Research Notices},
  2020(22):8973--8996, 2020.

\bibitem[Bau63]{baumslag1963}
Gilbert Baumslag.
\newblock Automorphism groups of residually finite groups.
\newblock {\em J. London Math. Soc.}, 38:117--118, 1963.

\bibitem[BDR20]{BDR2020}
Juliette Bavard, Spencer Dowdall, and Kasra Rafi.
\newblock Isomorphisms between big mapping class groups.
\newblock {\em Int. Math. Res. Not. IMRN}, (10):3084--3099, 2020.

\bibitem[BFH00]{BFH2000}
Mladen Bestvina, Mark Feighn, and Michael Handel.
\newblock The {T}its alternative for {${\rm Out}(F_n)$}. {I}. {D}ynamics of
  exponentially-growing automorphisms.
\newblock {\em Ann. of Math. (2)}, 151(2):517--623, 2000.

\bibitem[BFH04]{BFH2004}
Mladen Bestvina, Mark Feighn, and Michael Handel.
\newblock Solvable subgroups of {${\rm Out}(F_n)$} are virtually {A}belian.
\newblock {\em Geom. Dedicata}, 104:71--96, 2004.

\bibitem[BFH05]{BFH2005}
Mladen Bestvina, Mark Feighn, and Michael Handel.
\newblock The {T}its alternative for {${\rm Out}(F_n)$}. {II}. {A} {K}olchin
  type theorem.
\newblock {\em Ann. of Math. (2)}, 161(1):1--59, 2005.

\bibitem[BKMM12]{BKMM2012}
Jason Behrstock, Bruce Kleiner, Yair Minsky, and Lee Mosher.
\newblock Geometry and rigidity of mapping class groups.
\newblock {\em Geom. Topol.}, 16(2):781--888, 2012.

\bibitem[BKP19]{baumeister2019smallest}
Barbara Baumeister, Dawid Kielak, and Emilio Pierro.
\newblock On the smallest non-abelian quotient of {${\rm Aut}(F_n)$}.
\newblock {\em Proc. Lond. Math. Soc. (3)}, 118(6):1547--1591, 2019.

\bibitem[BLM83]{BLM1983}
Joan~S. Birman, Alex Lubotzky, and John McCarthy.
\newblock Abelian and solvable subgroups of the mapping class groups.
\newblock {\em Duke Math. J.}, 50(4):1107--1120, 1983.

\bibitem[BV00]{BV2000}
Martin~R. Bridson and Karen Vogtmann.
\newblock Automorphisms of automorphism groups of free groups.
\newblock {\em J. Algebra}, 229(2):785--792, 2000.

\bibitem[Can11]{Cantat2011}
Serge Cantat.
\newblock Sur les groupes de transformations birationnelles des surfaces.
\newblock {\em Ann. of Math. (2)}, 174(1):299--340, 2011.
\newblock corrected version (2012) available from
  \href{https://perso.univ-rennes1.fr/serge.cantat/Articles/cremona_long.pdf}{https://perso.univ-rennes1.fr/serge.cantat/Articles/cremona\_long.pdf}.

\bibitem[DHK23]{DHK2023}
George Domat, Hannah Hoganson, and Sanghoon Kwak.
\newblock Coarse geometry of pure mapping class groups of infinite graphs.
\newblock {\em Adv. Math.}, 413:Paper No. 108836, 2023.

\bibitem[Din12]{dinh2012}
Tien-Cuong Dinh.
\newblock Tits alternative for automorphism groups of compact {K}\"{a}hler
  manifolds.
\newblock {\em Acta Math. Vietnam.}, 37(4):513--529, 2012.

\bibitem[DP20]{DP2020}
George Domat and Paul Plummer.
\newblock First cohomology of pure mapping class groups of big genus one and
  zero surfaces.
\newblock {\em New York Journal of Mathematics}, 26:322--333, 03 2020.

\bibitem[Dud61]{dudley1961}
R.~M. Dudley.
\newblock Continuity of homomorphisms.
\newblock {\em Duke Math. J.}, 28(4):587--594, 12 1961.

\bibitem[FH07]{FH2007}
Benson Farb and Michael Handel.
\newblock Commensurations of {${\rm Out}({\rm F}_n)$}.
\newblock {\em Publ. Math. Inst. Hautes \'{E}tudes Sci.}, (105):1--48, 2007.

\bibitem[GP08]{grigorchuck2008intermediate}
Rostislav Grigorchuk and Igor Pak.
\newblock Groups of intermediate growth: {A}n introduction.
\newblock {\em Enseign. Math. (2)}, 54(3-4):251--272, 2008.

\bibitem[Gri80]{grigorchuck1980burnside}
R.~I. Grigor\v{c}uk.
\newblock On {B}urnside's problem on periodic groups.
\newblock {\em Funktsional. Anal. i Prilozhen.}, 14(1):53--54, 1980.

\bibitem[Gro75]{grossman1974}
Edna~K. Grossman.
\newblock On the residual finiteness of certain mapping class groups.
\newblock {\em J. London Math. Soc. (2)}, 9:160--164, 1974/75.

\bibitem[Har86]{Harer1986}
John~L. Harer.
\newblock The virtual cohomological dimension of the mapping class group of an
  orientable surface.
\newblock {\em Invent. Math.}, 84(1):157--176, 1986.

\bibitem[Hat95]{Hatcher1995}
Allen Hatcher.
\newblock Homological stability for automorphism groups of free groups.
\newblock {\em Comment. Math. Helv.}, 70(1):39--62, 1995.

\bibitem[Hat02]{Hatcher2002}
A~Hatcher.
\newblock {\em Algebraic Topology}.
\newblock Algebraic Topology. Cambridge University Press, 2002.

\bibitem[HHMV18]{HMV2018}
Jes\'{u}s Hern\'{a}ndez~Hern\'{a}ndez, Israel Morales, and Ferr\'{a}n Valdez.
\newblock Isomorphisms between curve graphs of infinite-type surfaces are
  geometric.
\newblock {\em Rocky Mountain J. Math.}, 48(6):1887--1904, 2018.

\bibitem[Hil23]{hill2023largescale}
Thomas Hill.
\newblock Large-scale geometry of pure mapping class groups of infinite-type
  surfaces.
\newblock {\em arXiv preprint arXiv:2309.00124}, 2023.

\bibitem[HV98]{HV1998}
Allen Hatcher and Karen Vogtmann.
\newblock Rational homology of {${\rm Aut}(F_n)$}.
\newblock {\em Math. Res. Lett.}, 5(6):759--780, 1998.

\bibitem[HW20]{HW2020}
Camille Horbez and Richard~D. Wade.
\newblock Commensurations of subgroups of {${\rm Out}(F_N)$}.
\newblock {\em Trans. Amer. Math. Soc.}, 373(4):2699--2742, 2020.

\bibitem[Iva84]{Ivanov1984}
N.~V. Ivanov.
\newblock Algebraic properties of the {T}eichm\"{u}ller modular group.
\newblock {\em Dokl. Akad. Nauk SSSR}, 275(4):786--789, 1984.

\bibitem[Iva88]{Ivanov1988}
N.~V. Ivanov.
\newblock Automorphisms of {T}eichm\"{u}ller modular groups.
\newblock In {\em Topology and geometry---{R}ohlin {S}eminar}, volume 1346 of
  {\em Lecture Notes in Math.}, pages 199--270. Springer, Berlin, 1988.

\bibitem[Iva97]{Ivanov1997}
Nikolai~V. Ivanov.
\newblock Automorphism of complexes of curves and of {T}eichm\"{u}ller spaces.
\newblock {\em Internat. Math. Res. Notices}, (14):651--666, 1997.

\bibitem[Khr90]{Khramtsov1990}
D.~G. Khramtsov.
\newblock Completeness of groups of outer automorphisms of free groups.
\newblock In {\em Group-theoretic {I}nvestigations ({R}ussian)}, pages
  128--143. Akad. Nauk SSSR Ural. Otdel., Sverdlovsk, 1990.

\bibitem[LL20]{LL2020}
Justin Lanier and Marissa Loving.
\newblock Centers of subgroups of big mapping class groups and the {T}its
  alternative.
\newblock {\em Glas. Mat. Ser. III}, 55(75)(1):85--91, 2020.

\bibitem[Mal58]{mal1958homomorphisms}
Anatoly~I Mal’cev.
\newblock On homomorphisms onto finite groups.
\newblock {\em Fluchen. Zap. Ivanovskogo Gos. Ped. Inst}, 18:49--60, 1958.
\newblock English translation in: Amer. Math. Soc. Transl. Ser. 2, 119 (1983)
  67-79.

\bibitem[McC85]{McCarthy1985}
John McCarthy.
\newblock A ``{T}its-alternative'' for subgroups of surface mapping class
  groups.
\newblock {\em Trans. Amer. Math. Soc.}, 291(2):583--612, 1985.

\bibitem[McC86]{McCarthy1986}
John~D. McCarthy.
\newblock Automorphisms of surface mapping class groups. {A} recent theorem of
  {N}. {I}vanov.
\newblock {\em Invent. Math.}, 84(1):49--71, 1986.

\bibitem[McK77]{McKenzie1977automorphism}
Ralph McKenzie.
\newblock Automorphism groups of denumerable {B}oolean algebras.
\newblock {\em Canadian J. Math.}, 29(3):466--471, 1977.

\bibitem[Mon75]{Monk1975automorphism}
J.~Donald Monk.
\newblock On the automorphism groups of denumerable {B}oolean algebras.
\newblock {\em Math. Ann.}, 216:5--10, 1975.

\bibitem[MR23]{mann2023large}
Kathryn Mann and Kasra Rafi.
\newblock Large-scale geometry of big mapping class groups.
\newblock {\em Geom. Topol.}, 27(6):2237--2296, 2023.

\bibitem[Nie24]{Nielsen1924}
Jakob Nielsen.
\newblock Die {I}somorphismengruppe der freien {G}ruppen.
\newblock {\em Math. Ann.}, 91(3-4):169--209, 1924.

\bibitem[PV18]{PatelVlamis}
Priyam Patel and Nicholas~G. Vlamis.
\newblock Algebraic and topological properties of big mapping class groups.
\newblock {\em Algebr. Geom. Topol.}, 18(7):4109--4142, 2018.

\bibitem[Ros22]{rosendal2022}
Christian Rosendal.
\newblock {\em Coarse geometry of topological groups}, volume 223 of {\em
  Cambridge Tracts in Mathematics}.
\newblock Cambridge University Press, Cambridge, 2022.

\bibitem[SC20]{schaffer2020}
Anschel Schaffer-Cohen.
\newblock Graphs of curves and arcs quasi-isometric to big mapping class
  groups.
\newblock {\em arXiv preprint arXiv:2006.14760}, 2020.
\newblock To appear in \emph{Groups, Geometry, and Dynamics}.

\bibitem[Use]{OutFinfinity2011}
User83827.
\newblock Is {$\textrm{Out}(F_{\infty})$} residually finite?
\newblock Mathematics Stack Exchange.
\newblock \url{https://math.stackexchange.com/q/57867} (version: 2011-08-17).

\end{thebibliography}
